\title{Dehn twists exact sequences through Lagrangian cobordism}

\documentclass[11pt]{article}

\usepackage{amsfonts}
\usepackage{amsthm}
\usepackage{latexsym, amssymb, bbm}
\usepackage{xcolor}
\usepackage{graphicx}
\usepackage{amsmath}
\usepackage{enumerate}
\usepackage[all]{xy}
\usepackage[margin=1.3in]{geometry}
\usepackage[mathscr]{eucal}
\usepackage{latexsym, amssymb, bbm, mathrsfs,mathabx}
 \usepackage{caption} \usepackage{subcaption}

\usepackage{amsmath,amsthm,amsfonts,amssymb,graphicx,psfrag}
\usepackage{color}

\usepackage{hyperref}

\usepackage{tikz}
\usepackage{comment}

\newtheorem{thm}{Theorem}[section]

\newtheorem{prop}[thm]{Proposition}
\newtheorem{lemma}[thm]{Lemma}
\newtheorem{rmk}[thm]{Remark}
\newtheorem{question}[thm]{Question}
\newtheorem{defn}[thm]{Definition}
\newtheorem{corr}[thm]{Corollary}
\newtheorem{eg}[thm]{Example}

\theoremstyle{definition}
\newtheorem{example}[thm]{Example}

\newcommand{\wt}[1]{\widetilde{#1}}
\newcommand{\wh}[1]{\widehat{#1}}
\newcommand{\wc}[1]{\widecheck{#1}}

\newcommand{\ov}[1]{\overline{#1}}

\newcommand{\bdf}{\begin{defn}}
\newcommand{\edf}{\end{defn}}

\newcommand{\bthm}{\begin{thm}}
\newcommand{\ethm}{\end{thm}}

\newcommand{\blem}{\begin{lemma}}
\newcommand{\elem}{\end{lemma}}

\newcommand{\bcor}{\begin{corr}}
\newcommand{\ecor}{\end{corr}}

\newcommand{\bprop}{\begin{prop}}
\newcommand{\eprop}{\end{prop}}

\newcommand{\brmk}{\begin{rmk}}
\newcommand{\ermk}{\end{rmk}}

\newcommand{\bpf}{\begin{proof}}
\newcommand{\epf}{\end{proof}}

\newcommand{\bex}{\begin{example}}
\newcommand{\eex}{\end{example}}

\newcommand{\beq}{\begin{equation}}
\newcommand{\eeq}{\end{equation}}

\numberwithin{equation}{section}

\def\eD{\EuScript{D}}
\def\eE{\EuScript{E}}
\def\eF{\EuScript{F}}
\def\eG{\EuScript{G}}
\def\eH{\EuScript{H}}

\def\eL{\EuScript{L}}

\def\eU{\EuScript{U}}
\def\eV{\EuScript{V}}

\def\sA{\EuScript{A}}
\def\sB{\EuScript{B}}
\def\sC{\EuScript{C}}

\def\sG{\EuScript{G}}

\def\sI{\EuScript{I}}

\def\sM{\EuScript{M}}

\def\sP{\EuScript{P}}

\def\sY{\EuScript{Y}}

\def\C{\mathbb{C}}

\def\K{\mathbb{K}}

\def\bP{\mathbb{P}}

\def\R{\mathbb{R}}
\def\Z{\mathbb{Z}}

\def\CP{\mathbb{CP}}
\def\RP{\mathbb{RP}}
\def\bP{\mathbb{P}}
\def\bD{\mathbb{D}}

\def\fH{\mathbf{H}}

\def\fJ{\mathbf{J}}

\def\bt{\mathbbm{t}}

\def\bk{\mathbf{k}}

\def\cA{\mathcal{A}}

\def\cD{\mathcal{D}}
\def\cE{\mathcal{E}}
\def\cF{\mathcal{F}}

\def\cL{\mathcal{L}}
\def\cM{\mathcal{M}}

\def\cO{\mathcal{O}}

\def\cS{\mathcal{S}}

\def\CO{\mathcal{C}\mathcal{O}}

\def\fuk{\EuScript{F}uk}
\def\ov{\overline}

\def\a{\alpha}

\def\e{\epsilon}

\def\l{\lambda}
\def\la{\lambda}

\def\vp{\varphi}

\def\w{\omega}

\def\vp{\varphi}
\def\p{\phi}

\def\la{\langle}
\def\ra{\rangle}

\def\mr{\mu_{RS}}

\def\xkm2{\overline{X}_{k-2}}

\def\ul{\underline}

\def\mr{\mathring}
\def\lar{\looparrowright}

\def\qno{q_{-1}}
\def\vp{\varpi}

\begin{document}

\author{Cheuk Yu Mak\thanks{C. Y. M. is supported by NSF-grant DMS 1065927.} and Weiwei Wu\thanks{W. W. is supported by CRM-ISM postdoctoral fellowship.}}

\AtEndDocument{\bigskip{\footnotesize
  \textsc{Cheuk Yu Mak, School of Mathematics, University of Minnesota, Minneapolis, MN 55455} \par
  \textit{E-mail address}: \texttt{makxx041@math.umn.edu} \par

}\bigskip{\footnotesize
  \textsc{Weiwei Wu, Centre de recherches math\'ematiques,
      Universit\'e de Montreal,
      2920 Chemin de la tour,
      Montreal (Quebec) H3T 1J4} \par
  \textit{E-mail address}: \texttt{wuweiwei@crm.umontreal.ca} \par

}}

\date{September 25, 2015}
\maketitle

\begin{abstract}
In this paper we introduce the following new ingredients: (1) rework on part of the Lagrangian surgery theory; (2) constructions of Lagrangian cobordisms on product symplectic manifolds; (3) extending Biran-Cornea Lagrangian cobordism theory to the immersed category.

As a result, we manifest Seidel's exact sequences (both the Lagrangian version and the symplectomorphism version), as well as Wehrheim-Woodward's family Dehn twist sequence (including the codimension-1 case missing in the literature) as consequences of our surgery/cobordism constructions.

 Moreover, we obtain an expression of the autoequivalence of Fukaya category induced by Dehn twists along Lagrangian $\RP^n$, $\CP^n$ and $\mathbb{HP}^n$, which matches Huybrechts-Thomas's mirror prediction of the $\CP^n$ case modulo connecting maps.  We also prove the split generation of any symplectomorphism by Dehn twists in $ADE$-type Milnor fibers.

\end{abstract}

\tableofcontents

\section{Introduction}

\subsection{Motivations and overview}

The celebrated Lagrangian cobordism theory introduced by Biran and Cornea in their sequel papers \cite{BC13}\cite{BC2}\cite{BCIII} has achieved great success encapsulating information of the triangulated structures of the Fukaya category.  A particularly attractive application is that they establish the long-expected relation between Lagrangian surgeries \cite{LS91}\cite{PolSurg} and the mapping cones in Fukaya categories.

A primary purpose of this paper is to revisit such surgery-cobordism relations with emphasis on applications to Dehn twists.  The underlying philosophy of our approach is to understand the functors between Fukaya categories via Lagrangian cobordisms.  This functor-level point of view has been exploited in several other contexts by many authors \cite{MWWFunctor}\cite{WWfamily}\cite{Gan13} \cite{AS15} etc.

We explore this direction through the eyes of Lagrangian cobordisms and corresondences.  Intuitively, one may regard Lagrangian correspondences as symplectic mirrors of kernels of Fourier-Mukai transforms.  The observation is, almost all exact sequences involving Lagrangian Dehn twists can be interpreted as cone relations between these ``kernels".  Explicitly, Lagrangian cobordism constructions geometrically realize all these cones on the correspondence level and provides a completely analogous picture on the symplectic side, versus various twist constructions on derived categories. This point of view greatly simplifies the proof of several known exact sequences and leads to new cone relations in Floer theory such as Lagrangian $\CP^n$-twists, verifying a conjecture due to Huybrechts-Thomas.

To this end, much work needs to be done on the general geometric framework.  We first reconstructed and extended the well-known Lagrangian surgery from connected sums to fiber sums, using a coordinate-free approach.  The construction is intentionally designed to have many variants for our applications and future exploration.  Also, we extended Biran-Cornea's Lagrangian cobordism formalism to the immersed category.  This last part also contains new ingredients: we adapted a \textit{bottleneck} trick from \cite{BC2}
to immersed cases, which achieves compactness in some cases when the infinity ends are not even cylindrical.

\subsection{Flow surgeries and flow handles}

Recall that for two Lagrangians $L_1\pitchfork L_2=\{x\}$, their Lagrangian surgery at $x$ is given by adding an explicit Lagrangian handle in the Darboux chart \cite{LS91}\cite{PolSurg}.  Then a Lagrangian cobordism can be obtained by using ``half" of a Lagrangian handle of one dimension higher \cite{BC13}.  This line of thoughts has led to remarkable breakthroughs in both constructions of new examples of Lagrangian submanifolds and
cobordism theory.

To implement this construction to Lagrangian ``fiber sums" (surgery along clean intersections), the patching of local models requires more delicate consideration on the connection of normal bundles.  On top of that, in most of our applications, the main difficulty is to show that the resulting manifold is Hamiltonian isotopic to certain given Lagrangians, usually those obtained by Lagrangian Dehn twists.

Our basic idea to solve both problems at once is to use a reparametrized geodesic flow, mimicking the original construction of Dehn twist by Seidel, to produce a new Lagrangian surgery operation called the \textbf{flow surgery} (See Section \ref{s:surgery}).  This flow surgery recovers the usual Lagrangian surgery when the auxiliary data is chosen appropriately, but has much better flexibility.  For example, the resulting Lagrangian handle needs not be diffeomorphic to a puntured ball (or a bundle with punctured-ball fibers in the clean surgery case).  Moreover, the Biran-Cornea's cobordism construction via surgeries fits into this framework easily as well.

The main examples we have are the following (see Section \ref{s:AdmissToDehn} and \ref{s:ImmersedConstruction} for relevant definitions).

\bthm\label{t:surgeries} Let $S^n\subset M$ be a Lagrangian sphere and $S\subset M$ be a Lagrangian submanifold diffeomorphic to either $\RP^n$, $\CP^n$ or $\mathbb{HP}^n$.  Let $\tau_{S^n}$ and $\tau_S$ denote the corresponding Dehn twists. One has the following surgery equalities up to hamiltonian isotopies in $M\times M^-$:

\begin{enumerate}[(1)]
\item $(S^n\times (S^n)^-)\#\Delta_M=Graph(\tau_{S^n}^{-1})$,
\item $\wt C\#\Delta_M=Graph(\tau_C^{-1})$, where $C\subset M$ is a spherically coisotropic submanifold.
\item $(S\times S^-)\#(S\times S^-)\#\Delta_M=Graph(\tau^{-1}_S)$,
\item $S\#(S\#L)=S_\looparrowright \#L=\tau_S (L)$ for any Lagrangian $L$. Here $S_\looparrowright$ is an immersed Lagrangian sphere associated to $S$.
\item $\wt C_P\#\wt C_P\#\Delta_M=Graph(\tau_{C_P}^{-1})$, where $C_P\subset M$ is a projectively coisotropic submanifold.
\end{enumerate}

\ethm

The surgery equalities immediately lead to the existence of corresponding Lagrangian cobordisms.  Note that in case (1), a similar cobordism construction was established in \cite{AS15} using Lefschetz fibrations independently.

\brmk\label{r:RPHP} Formal proofs will only be given in the case of $S^n$ and $\CP^n$, since $\mathbb{HP}^n$ and $\RP^n$ cases will follow from the proof of $\CP^n$ word-by-word.  The common feature we used for these manifolds we used is the existence of a metric $g_S$ of the following property: for any point $x\in S$, the injectivity radius $x$ equals $\pi$, and $S\backslash B_x(\pi)$ is a smooth closed submanifold.

\ermk

We also include a detailed discussion on gradings involved in Lagrangian surgeries.  This benefits us in two aspects: we use a grading assumption to exclude bubblings in immersed Floer theory (Section \ref{s:immersedFloer}), and it allows us to compute the connecting maps later on (Section \ref{s:connectingMaps}).  But we emphasize the grading is a vital part of the foundation of Lagrangian surgeries for an intrinsic reason.  Consider the simple case when all involved Lagrangians are $\Z$-graded and embedded, according to the cone relation proved in \cite{BC13}, the algebra instructs a surgery happen only at \textbf{degree zero} cocycles.  This principle was noticed first by Paul Seidel \cite{SeGraded}.

 Such a principle interprets several known phenomena in a uniform way.  First of all the positive and negative surgeries at the same point should be viewed as two different cones $Cone(L_0\xrightarrow{c}L_1)$ and $Cone(L_1[-n]\xrightarrow{c^\vee[-n]}L_0)$, which are apriori very different.  When the resolved intersections have mixed degrees, in many cases this leads to obstructions in Floer theory, as exemplified in \cite[Chapter 10]{FOOO_Book}.  In better situations when resolved intersections have zero degree mod $N$, the surgery at least results in collapse of gradings, which can also be checked directly on the Maslov classes.

As for our applications, we extend this principle to clean surgeries.  The upshot is that, for two graded Lagrangians with $L_0\cap L_1=D$ being a clean intersection with zero Maslov index, $L_1$ and $L_0[dim(D)+1]$ can be glued as a graded Lagrangian.  This matches well with predictions from homological algebra dictated by Lagrangian Floer theory with clean intersections \cite{FOOO_Book}.  It also extends the surgery exact sequence to clean intersection case.

\subsection{Cone relations in functor categories via Lagrangian cobordisms}

From the surgery equalities in Theorem \ref{t:surgeries} and the corresponding cobordisms, we immediately recover Seidel's exact sequence and Wehrheim-Woodward's family Dehn twist sequence on the functor level, assuming all exactness/monotonicity conditions discussed in Section \ref{s:proofLES}:

\bthm[see \cite{WWfamily}]\label{t:sphereTwist} When $S^n\subset M$ is a Lagrangian sphere, there is a cone in the $Aut(Tw\fuk(M))$ that

\beq\label{e:sphereCone}\xymatrix{ hom(S^n,-)\otimes S^n\ar[r]& id\ar[d] \\ &\tau_{S^n}\ar[lu]^(.4){[1]}}\eeq

When $C\subset M$ be a spherically coisotropic submanifold, there is a cone in $Aut(Tw\fuk(M))$ that

\beq\label{e:familyCone}\xymatrix{ C^t\circ C\ar[r]& id\ar[d] \\ &\tau_{C}\ar[lu]^(.45){[1]}}\eeq

\ethm

New information is obtained through our methods.
In the Lagrangian sphere case, as pointed out to us by Octav Cornea, since our proof does not involve any energy estimates, it shows that Seidel's exact sequence holds over $\Z/2\Z$ in  monotone cases, versus over Novikov rings in the literature.  For the family Dehn twist case, in addition to the improvement in coefficients, we also cover the coisotropic dimension one case when $\pi_1(M)=1$.  For other symplectic manifolds, our method reduces the problem of proving exact sequences to checking the exactness/monotonicity condition for codimension one spherically coisotropic manifolds, which is way more concrete.

In another direction, since our construction holds for arbitrary symplectic manifolds, when combined with the general framework due to Fukaya-Oh-Ohta-Ono \cite{FOOO_Book}, it yields a proof of Seidel's exact sequence in arbitrary symplectic manifold.  This is part of an ongoing work \cite{WXgen}.

As a consequence of the cone relations in functor categories, we also consider the auto-equivalences of $Aut(\fuk(W))$, for $W$ a Milnor fiber of $ADE$-type singularities (The generalization of the result from $A$-type singularities to $DE$-type singularities was suggested to us by Ailsa Keating).  In \cite{SeGraded}\cite{Seidelbook} it was proved that $\fuk(W)$ is split generated by the vanishing cycles.  Moreover, in \cite{KS02} it is shown that there is a braid group embedded into $Symp_c(W)$ when $W$ is an $A_n$-Milnor fiber induced by Dehn twists along the standard vanishing cycles.  It is natural to ask whether this braid group indeed is the whole mapping class group, i.e. $\pi_0(Symp_c(W))$.  In $4$-dimensional cases, the explicit topology of the whole $Symp_c(W^4)$ can be completely understood \cite{Ev11}\cite{Wu14}, thus establishing the isomorphism $\pi_0(Symp_c(W^4))=Br_{n+1}$.

The situation in higher dimension, however, is much more complicated: it is still a widely open problem whether $\pi_0 Ham_c(B^{2m})=\{1\}$ for $m\ge 3$.  On the other hand, \cite{DE14} showed that there exists exotic parametrization of the sphere itself, so that the Dehn twist along the exotic parametrization already gives a symplectomorphism which is different from the standard Dehn twists (although \textit{a priori} it is still unclear if it isotropic to a composition of Dehn twists in $ADE$-type Milnor fiber).  In light of this result, it seems hard to expect in dimension$>4$ a similar nice description of $\pi_0(Symp_c(W^4))=Br_{n+1}$ as in low dimensions.

As a consequence, we turn to a categorical reduction of the problem.  In other words, we consider

\begin{question}\label{q:Eq} Does every $\phi\in Symp_c(W)$ induce an autoequivalence $\Phi_\phi\in Aut(\fuk(W))$ which is isomorphic to one that is induced by a composition of Dehn twists along vanishing cycles?

\end{question}

We are able to prove a weaker version of Question \ref{q:Eq}, which is a reminiscence of Seidel's split generation result of vanishing cycles, as well as the fully faithfulness of M'au-Wehrheim-Woodward's $A_\infty$-functor in certain subcategories, proved in \cite{AS10}\cite{Gan13}.

\bthm[Theorem \ref{t:AnFun}]\label{t:AnFun-0} Let $W$ be a Milnor fiber of an $ADE$-type singularity.  For any compactly supported symplectomorphism $\phi\in Symp_c(W)$, $\Phi_\phi\in D^\pi Aut(\fuk(W))$ is split generated by functors induced by Lagrangian Dehn twists along the standard vanishing cycles and their compositions.
\ethm

\subsection{The Huybrechts-Thomas conjecture and projective twists}
There is a natural extension of Dehn twists construction along spheres to arbitrary rank-one symmetric spaces, which is known for a long time.  As Seidel discovered the long exact sequence associated to a Dehn twist along spheres, the spherical twists, as the mirror auto-equivalences, also received much attention \cite{ST01}.  Also, such a cone relation on the $A$-side has become a foundational tool in the study of homological mirror symmetry, especially in the Picard-Lefschetz theory \cite{Seidelbook}.

  It has long been curious since that, what the auto-equivalence corresponding to the Dehn twists along a rank-one symmetric space is.  On the $B$-side, Huybrechts-Thomas \cite{HT06} defined \textit{$\bP^n$-objects} on derived categories of smooth algebraic varieties, as well as a corresponding new auto-equivalence called the \textit{$\bP^n$-twist}.  They then conjectured the $\bP^n$-twist is exactly the mirror auto-equivalence of the one induced by a Dehn twist along Lagrangian $\CP^n$ on the Fukaya categories.  Richard Harris studied the problem in $A_\infty$ contexts and formulated the corresponding algebraic twist on $A$-side \cite{Ha11}.  The only missing link to the actual geometry of Lagrangian submanifolds, remains unproved for years.

As an application of the surgery equalities in Theorem \ref{t:surgeries}, we show:

\bthm\label{t:HTtrue} Let $S\subset M$ be a Lagrangian $\CP^n$, and $L\subset M$ a Lagrangian submanifold.  Then Huybrechts-Thomas conjecture is true up to determination of connecting maps.\ethm

The proof of Theorem \ref{t:HTtrue} follows from the construction of a cobordism representing an iterated cone on the functor level, see Theorem \ref{t:surgeries} and Lemma \ref{l: graded identity for graph of CPm/2}.  Our method applies well on $\RP^n$ or $\mathbb{HP}^n$, and should extend to other Lagrangians whose geodesics are closed with rational proportions such as Cayley plane or their finite covers.  These are supposed to be the mirror of $\bP^n$-like objects except for a change in gradings for non-trivial self-hom's.  A family version of projective twist is also given, see Theorem \ref{t:familyPn}.

\brmk
While it is not difficult to find examples of Lagrangian $\RP^n$ in problems in symplectic topology \cite{Sher14}\cite{Wu14}, the search of interesting examples of Lagrangian $\CP^n$ is more intriguing.  In \cite{HT06} the authors suggested several sources of $\bP^n$-objects in derived categories.  An interesting instance is given by pull-back sheaves of a Lagrangian fibration on a hyperk\"ahler manifold.  From the SYZ point of view, this should correspond to a Lagrangian $\CP^n$ section on the SYZ mirror.  While the role of $\bP^n$ objects on either side of mirror symmetry remains widely open so far, it is interesting to know whether such objects split generate either side of mirror symmetry.
\ermk

\brmk

In a different direction, the $\bP^n$-cone relation should be interested in understanding some basic problems in symplectic topology, such as mapping class groups of a symplectic manifold and the search of exotic Lagrangian submanifolds.  For instance, while a Lagrangian $\CP^n$-twist is always smoothly isotopic to identity, it is usually not \textbf{Hamiltonian} isotopic to identity.  A simplest model result along this line is to generalize Seidel's twisted Lagrangian sphere construction \cite{Se99}: in the plumbing of three $T^*\CP^n$, the iterated Dehn twists along $\CP^n$ in the middle should generate an infinite subgroup in the symplectic mapping class group.

\ermk

\brmk
With Theorem \ref{t:surgeries} the projective twist cone formula easily generalizes to $\RP^n$ and $\mathbb{HP}^n$.  The only difference between the formulas is the grading shift of the first term, as specified in Theorem \ref{t:catPn}.

$\RP^n$ also gains a special feature: in this case the associated sphere $S_\looparrowright$ is equivalent to $\RP^n$ equipped with a nontrivial $\Z_2$-local system in the Fukaya category (see \cite{Da12}\cite{AB14}\cite{Sher14}).  Therefore, the iterated cone relation can be packaged into a long exact sequence without invoking the immersed Floer theory.

\ermk

\subsection{Immersed Lagrangian cobordism theory and computation of connecting maps}\label{s:introImm}

As a technique of independent interests, we have extended Biran-Cornea's Lagrangian cobordism formalism to the immersed objects, on the level of Donaldson-Fukaya category in Section \ref{s:immersedCob}.  The idea follows largely that of \cite{BC13} and \cite{AJ10}.  For the case at hand, we have restricted ourselves to the exact setting for simplification, and mostly followed Alston-Bao's exposition \cite{AB14}.  The upshot is as expected, that the existence of an immersed Lagrangian cobordism incurs a quasi-isomorphism between certain mapping cones coming from Floer theory.

However, the actual proof is far from straightforward.  The key issue is the cleanness of self-intersections of the immersed cobordisms, which is required for the well-behaviors of moduli spaces of pseudo-holomorphic curves.  It is not hard to establish a cobordism theory naively following Biran-Cornea's definition in embedded categories and assume the required geometric transversality, but this will not even cover the simplest application at hand.

\begin{figure}[h]
\centering

\includegraphics[scale=0.85]{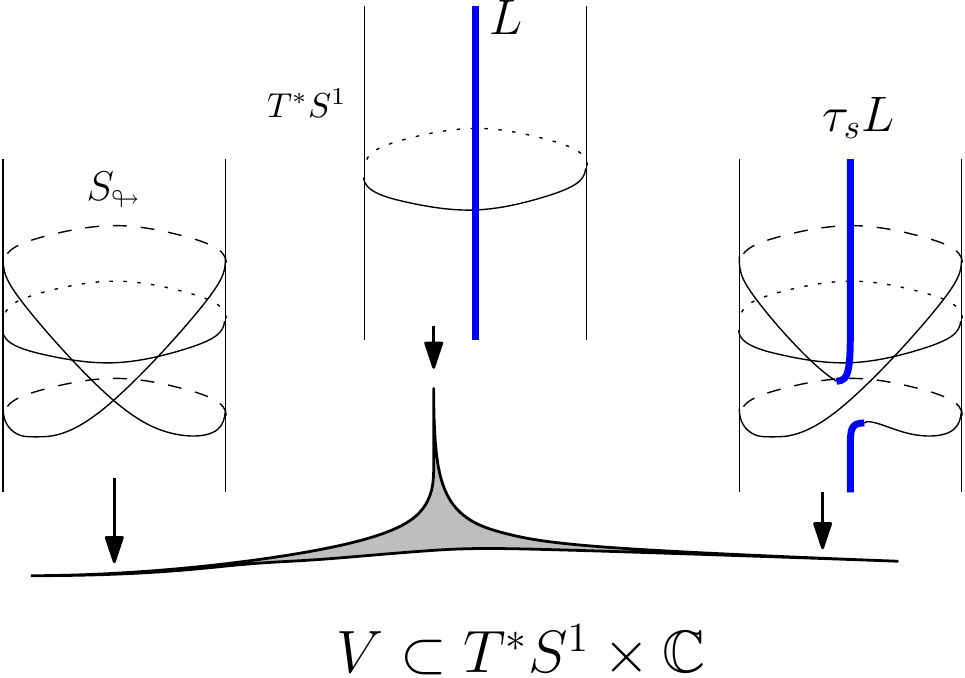}

\caption{An immersed Lagrangian cobordism: a surgery in $T^*S^1$}
\label{fig: sample cobordism}
\end{figure}

\bex The simplest instance of a projective twist can be demonstrated concretely in $M=T^*S^1$, see Figure 1.  Here we consider $S_\looparrowright\subset M$ be an immersed circle with a unique transverse immersed point.  $L$ is given by the cotangent fiber at a point, and we assume it passes through the unique immersed point of $S$.  While the base is regarded as an $\RP^1$, by definition, the surgery $S\#L=\tau_{\RP^1}L$, where the surgery is perform through one of the branches of $S$ at the immersed point.

This surgery can be recast into a Lagrangian cobordism in $T^*S^1\times\C$.  The cobordism can be constructed so that it naively satisfies Biran-Cornea's definition, i.e. outside $T^*S^1\times K$ for some compact set $K$, it is a union of products between rays and immersed Lagrangian submanifolds in $T^*S^1$.  However, it is evident that the self-intersection cannot be clean since they form a ray. In general any surgery process involving resolution of an immersed point will suffer from the same caveat.  Therefore, we need a modification for the Floer theory to be well-defined.

\eex

Our solution to the problem above is to use a \textit{bottleneck} trick, which is a specific perturbation on $L\times \R^\pm$ so that its projection has the shape of a double cone, see Section \ref{s:immersedCob} for details.  This idea was adapted from \cite{BC2}, where a \textit{bottleneck} referred to a particular intersection pattern between two infinity ends $L_i\times\R$, $i=0,1$. Although we only deal with Donaldson-Fukaya categories in our setting, considerable amounts of new issues need to be addressed since most of our curves cannot actually be projected, which is also the main catch to prove statements on the $A_\infty$ level.  This will appear independently in the future.

An application of this general immersed cobordism framework is to give an alternative proof to the projective twist formulae in the special case when $L\pitchfork S=\{x\}$ for any possibly immersed Lagranigan $L\subset M$.   Using Theorem \ref{t:surgeries} (4), this approach is closer to the more prevalent viewpoint on the relation between Dehn twists and surgeries.

 A bonus point of this alternative approach is we could ``compute" the connecting maps which is difficult for general Lagrangian cobordisms.  The immersed formalism along with a simple algebraic trick extract enough information to determine almost all relevant mapping cones up to quasi-isomorphisms we covered in this paper, see Section \ref{s:connectingMaps}.  In particular, when $L\pitchfork S=\{x\}$, and assuming an $A_\infty$ version of the immersed cobordism formalism, we are able to improve Theorem \ref{t:HTtrue} by matching Huybretches-Thomas conjecture also on the connecting morphisms, thus yielding an affirmative answer to their question.
 \bigbreak
 \noindent{\bf A note on coefficients.}\\\\
 Throughout this paper we will use coefficient $\Z/2$ or $\Lambda_{\Z/2}$.  Using characteristic zero coefficients is possible up to checking orientations for the general framework on Lagrangian cobordisms for \cite{BC13, BC2}.  As far as the coefficient is concerned, however, the Huybrechts-Thomas conjecture might not hold in general for $\C$-coefficients, for example, $\CP^{2k}$ are not spin thus will be constrained in Fukaya categories defined over $\Z/2$ in many cases.

\section*{Acknowledgement} This project was initiated from a discussion with Luis Haug, and have benefited significantly from conversations with Octav Cornea, who suggested the key idea of bottleneck among many other invaluable inputs.  We thank both of them for their help and interests over the whole period when this paper was being prepared.  Garrett Alston, Erkao Bao, Paul Biran, Kenji Fukaya, Ailsa Keating, Yanki Lekili, Conan Leung, Tian-Jun Li, Yin Li, Yong-Geun Oh, Kaoru Ono, Egor Shelukhin, Nick Sheridan and Richard Thomas gave insightful comments to our early draft.  C.-Y. M. thanks Tian-Jun Li for his continuous support and encouragement, and the authors are both grateful to CRM for providing an inspiring environment.

\begin{comment}
To explain our main examples, first assume that $S^n\subset M$ be a Lagrangian sphere.  Then we showed in Example \ref{ex:SnSndiagonal} that there is a surgery equality in $(M\times M^-,\w\oplus\w^-)$

\beq(S^n\times (S^n)^-)\#\Delta_M=Graph(\tau_{S^n}),\eeq

where $\tau_{S^n}$ denotes the Dehn twist.  In its family version, assume $C\subset M$ is a spherically coisotropic submanifold (see Example \ref{ex:familyDehnSurg} for the definitions), one may also show

\beq\wt C\#\Delta_M=Graph(\tau_C).\eeq

However, our main concern follows a different streamline.  We apply our new construction to the understanding of various version of Dehn twists: this yields a new proof to Seidel's exact sequence, Wehrheim-Woodward's family Dehn twist sequence, as well as a new result on understanding the Dehn twist along Lagrangian submanifolds diffeomorphic to a rank-one symmetric spaces.  Even for the known cases (Dehn twists along a Lagrangian spheres and the family version), our new approach adds new information to the picture, as detailed below.
\vskip 5mm
\noindent\textbf{\large Spherical Dehn twists and their family version.}

\end{comment}
\vskip 5mm
\noindent$\bullet$ \textbf{Conventions.}\\

Throughout the paper, we assume any Lagrangian submanifold $L\subset M$ of a symplectic manifold $(M,\w)$ under consideration is \textit{exact}
or \textit{monotone}, which means:
\begin{itemize}
\item(exactness) $\w=d\theta$ for some $\theta\in\Omega^1(M)$, and $\theta|_L=df$ for some smooth function $f$ on $L$.
\item(monotonicity) For any $\alpha\in \pi_2(M,L)$, $\w(\alpha)=\lambda\mu(\alpha)$.  Here $\lambda>0$ and $\mu$ denotes the Maslov class.
\end{itemize}
All Lagrangians are assumed to be proper, and non-compact exact Lagrangian embeddings are assumed to have cylindrical end, unless specified otherwise.

The Hamiltonian vector field of a Hamiltonian function $h$ is defined by $\iota_{X_h}\w=\w(X_h,-)=-dh$ and the time-$t$ flow under $X_h$ is denoted as $\phi^h_t$.

We will also denote $M^-=(M,-\w)$ to be the \textit{negation of the symplectic manifold} $(M,\w)$.

\section{Dehn twist and Lagrangian surgeries}

\subsection{Dehn twist}

Let $S$ be a connected closed manifold equipped with a Riemannian metric $g(\cdot,\cdot)$ such that every geodesic is closed of length $2\pi$.
We identify $T^*S$ with $TS$ by $g$ and switch freely between the two.
The following lemma is well-known.

\begin{lemma}
The Hamiltonian $\vp: T^*S \to \mathbb{R}$ defined by
$$\vp(\xi)=\|\xi\|$$
for all $q\in S$ and $\xi \in T^*_qS$ has its Hamiltonian flow $X_{\vp}$ coincides with the normalized geodesic flow on $T^*S \backslash \{0_{section}\}$.

\end{lemma}

To define Dehn twist, we need to introduce an auxiliary function.
We first consider the case when $S$ is not diffeomorphic to a sphere.
For $\epsilon >0$ small, we define a \textbf{Dehn twist profile} to be a smooth function $\nu^{Dehn}_{\epsilon}: \mathbb{R}^+ \to \mathbb{R}$  such that

\begin{enumerate}[(1)]
\item $\nu^{Dehn}_{\epsilon}(r)=2\pi-r$ for $r \ll \epsilon$,

\item $0 < \nu^{Dehn}_{\epsilon}(r)< 2\pi$ for all $r < \epsilon$, and

\item $\nu^{Dehn}_{\epsilon}(r)=0$ for $r \ge \epsilon$

\end{enumerate}

\begin{defn}\label{defn: Dehn twist on cotangent bundles}
If $S$ is not diffeomorphic to a sphere, the model Dehn twist $(\tau_S,\nu^{Dehn}_{\epsilon})$ on $T^*S$ is given by
$$\tau_S(\xi)=\phi^{\vp}_{\nu^{Dehn}_{\epsilon}(\| \xi \|)}(\xi)$$
on $T^*S-\{0_{section}\}$ and identity on the zero section.
\end{defn}

We will simply write $\tau_S$ instead of $(\tau_S,\nu^{Dehn}_{\epsilon})$.

When $S$ is diffeomorphic to a sphere, the \textbf{spherical Dehn twist profile} $\nu^{Dehn}_{\epsilon}$ is picked with (1)(2) above replaced by

\begin{enumerate}[(1')]
\item $\nu^{Dehn}_{\epsilon}(r)=\pi-r$ for $r \ll \epsilon$, and

\item $0 < \nu^{Dehn}_{\epsilon}(r)< \pi$ for all $r < \epsilon$

\end{enumerate}

In this case, Dehn twist $(\tau_S,\nu^{Dehn}_{\epsilon})$ is defined analogously but antipodal map is used to extend smoothly along the zero section instead of the identity map.

\begin{example}\label{eg: RP^1 twist on cotangent bundle of circle}
Let $(q,p) \in \mathbb{R}/2\pi \mathbb{Z} \times \mathbb{R} = T^*S^1$ be equipped with the standard symplectic form $\omega_{S^1}=dp \wedge dq$.
For a spherical profile $\nu^{Dehn}_{\epsilon}$, $(\tau_{S^1},\nu^{Dehn}_{\epsilon})$ is defined by

\begin{displaymath}
 \tau_{S^1}(q,p)= \left\{
\begin{array}{lrr}
(q+\nu_\epsilon^{Dehn}(||p||)\frac{p}{||p||},p) &\text{for $p\neq0$}\\
(q+\pi,0) &\text{for $p=0$}
%(q+\nu^{Dehn}_{\epsilon}(p),p)  & \text{for $p > 0$}\\
%(0,q+\pi)  & \text{for $p = 0$} \\
%(p,q-\nu^{Dehn}_{\epsilon}(-p)) & \text{for $p < 0$}
\end{array}
\right.
\end{displaymath}

Consider the double cover
$\iota_{double}: T^*S^1 \to T^*\mathbb{RP}^1=\mathbb{R}/2\pi \mathbb{Z} \times \mathbb{R}$ given by
$$\iota_{double}(q,p)=(2q,\frac{1}{2}p)=(\wt{q},\wt{p})$$
For $(\wt{q},\wt{p})=\iota_{double}(q,p) \in T^*\mathbb{RP}^1$, we define
$$T(\wt{q},\wt{p})=\iota_{double} \circ \tau_{S^1}(q,p)$$
which is independent of the choice of $(q,p)$ as lift of $(\wt{q},\wt{p})$.  It is an easy exercise to show that $T$ is Hamiltonian isotopic to $\tau_{\mathbb{RP}^1}$
for the push-forward Dehn twist profile.
Also, if we identify $T^*\mathbb{RP}^1$ with $T^*S^1$ so that $\tau_{S^1}$ is well-defined on $T^*\mathbb{RP}^1$,
then $T$ is also Hamiltonian isotopic to $\tau_{S^1}^2$ for an appropriate choice of spherical profile.
\end{example}

This example has the following well-known immediate generalizations.

\begin{lemma}\label{lemma: alternative definition of real projective twist}
Let $\iota_{double}: T^*S^n \to T^*\mathbb{RP}^n $ be the symplectic double cover obtained by double cover of the zero section.
For $(\wt{q},\wt{p})=\iota_{double}(q,p) \in T^*\mathbb{RP}^n$,
$$T(\wt{q},\wt{p})=\iota_{double} \circ \tau_{S^n}(q,p)$$
is well-defined and
$T$ is Hamiltonian isotopic to $\tau_{\mathbb{RP}^n}$ for an appropriate choice of auxiliary function defining $\tau_{\mathbb{RP}^n}$.

If $n > 1$, the choice of auxiliary function defining $\tau_{\mathbb{RP}^n}$ is irrelevant up to Hamiltonian isotopy.
\end{lemma}

\begin{lemma}
For $T^*S^2=T^*\mathbb{CP}^1$, $\tau_{S^2}^2$ is Hamiltonian isotopic to $\tau_{\mathbb{CP}^1}$.

\end{lemma}

As usual, one may globalize the model Dehn twist.

\begin{defn}
A Dehn twist along $S$ in $M$ is a  compactly supported symplectomorphism defined by the model Dehn twist as above in a
Weinstein neighborhood of $S$ and extended by identity outside.
\end{defn}

For more details and the dependence of choices used to define $\tau_S$, see \cite{Se99} and \cite{Se03}.

\subsection{Lagrangian surgery through flow handles}\label{s:surgery}

\subsubsection{Surgery at a point}

We first recall the definition of a Lagrangian surgery at a transversal intersection from \cite{LS91}\cite{PolSurg} and \cite{BC13}.

\bdf\label{d:admissible}
Let $a(s),b(s)\in\R$.  A smooth curve $\gamma(s)=a(s)+ib(s) \in \mathbb{C}$ is called {\bf $\lambda$-admissible} if

$\bullet$ $(a(s),b(s))=(-s+\lambda,0)$ for $s \le 0$

$\bullet$ $a'(s),b'(s) <  0$ for $s \in (0,\epsilon)$, and

$\bullet$ $(a(s),b(s))=(0,-s)$ for $s \ge \epsilon$.

\edf

%\vspace{7em}
\begin{figure}[h]
\includegraphics[scale=0.6]{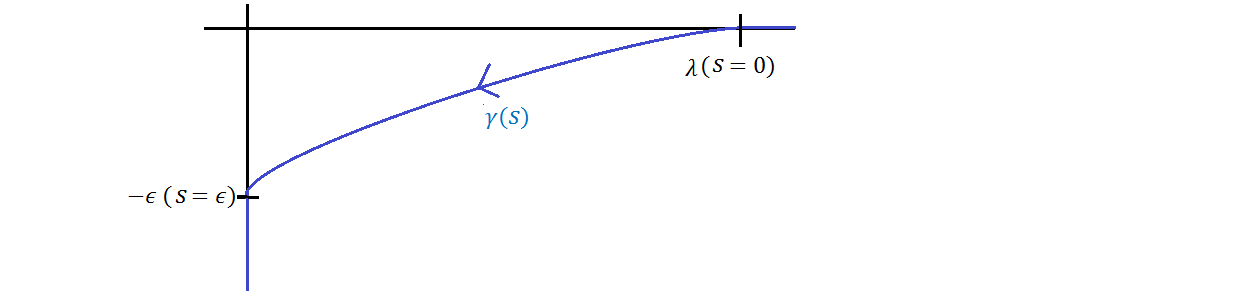}

\centering

\caption{Picture of an admissible curve.}
\label{figure1}
\end{figure}

The part of a $\lambda$-admissible curve with $s\in [0,\epsilon]$ can be captured by $\nu(r)=a(b^{-1}(-r))$.  The main property of an admissible curve can be rephrased as follows.
\begin{enumerate}[(1)]
\item $\nu_\lambda(0)=\lambda>0$, and $\nu_\lambda'(r)<0$ for $r\in(0,\epsilon)$.
\item $\nu_\lambda^{-1}(r)$ and $\nu_\lambda(r)$ has vanishing derivatives of all orders at $r=\lambda$ and $r=\epsilon$, respectively.
\end{enumerate}

\begin{figure}
\centering
\begin{subfigure}[b]{0.5\textwidth}
        \centering
        \includegraphics[height=1.2in]{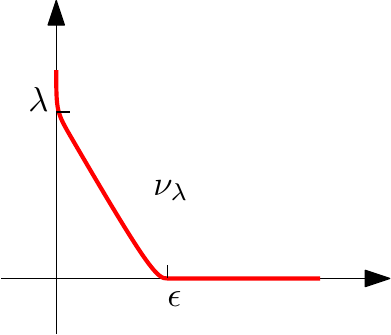} \caption{graph of an admissible function $\nu_\lambda$}
      %  \includegraphics[height=1.2in]{a}
       % \caption{Lorem ipsum}
    \end{subfigure}%
    ~
    \begin{subfigure}[b]{0.5\textwidth}
        \centering
       \includegraphics[height=1.2in]{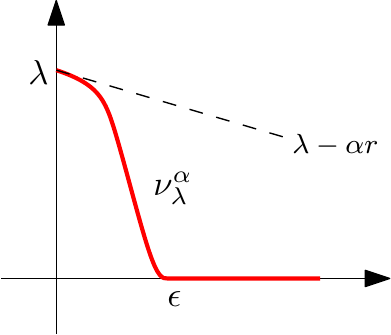} \caption{graph of a semi-admissible function $\nu_\lambda^\alpha$}
    \end{subfigure}
    \caption{Admissible and semi-admissible functions}
\end{figure}

Such a function will also be called \textbf{$\lambda$-admissible}.  We will frequently use the two equivalent descriptions  of admissibility interchangeably.

We also define a class of \textbf{semi-admissible functions}, by relaxing (2) to
\vskip3mm

\noindent$(2')$\quad $\nu_\l'(0)=-\alpha\in[-\infty,0]$.  Here $\a=\infty$ if $\nu_\l$ is admissible.
\break

Note that in all definitions of (semi-)admissibility there is an extra variable $\epsilon$.  We will see that the dependence on $\epsilon$ is not significant in this paper: we fix $\epsilon$ for each pair of Lagrangian submanifolds $(L_1, L_2)$ once and for all.  In any surgery constructions appearing later, the resulting surgery manifold yields a smooth family of isotopic Lagrangian submanifolds as $\epsilon$ vaires.  As a result we will suppress the dependence of $\epsilon$ unless necessary.

Given a $\lambda$-admissible curve $\gamma$, define the handle
$$H_{\gamma}=\{(\gamma(s)x_1,\dots,\gamma(s)x_n)| s,x_i \in \mathbb{R}, \sum x_i^2 =1 \} \subset \mathbb{C}^n$$
\begin{lemma}\label{lemma: Lagrangian hanlde is Lagrangian}
For an $\lambda$-admissible $\gamma$, $H_{\gamma}$ is a Lagrangian submanifold of $(\mathbb{C}^n,\sum dx_i \wedge dy_i)$.
\end{lemma}

\begin{proof}
It suffices to observe that $T_{\gamma(s)x}H_{\gamma}=Span_{\mathbb{R}}\{\gamma'(s)x\} \oplus \gamma(s)T_{x}S^{n-1}$,
for $x=(x_1,\dots,x_n) \in S^{n-1} \subset \mathbb{R}^n$.
\end{proof}

As a consequence, we have

\begin{corr}\label{c:SurgeryAtPoint}
Let $L_1,L_2 \subset (M,\omega)$ be two Lagrangians transversally intersecting at $p$.
Let $\iota:U \to M$ be a Darboux chart with a standard complex structure so that
$\iota^{-1}(L_1) \subset \mathbb{R}^n$ and $\iota^{-1}(L_2) \subset i\mathbb{R}^n$,
then one can obtain a Lagrangian $L_1\#_p L_2$ by attaching a Lagrangian handle $\iota(H_{\gamma})$ to $(L_1\cup L_2)\backslash\iota(U)$.
\end{corr}

The Lagrangian $L_1\#_pL_2$ is called a \textbf{Lagrangian surgery} from $L_1$ to $L_2$ following \cite{LS91,PolSurg}.
Note that, the Lagrangian $L_2\#_pL_1$ obtained by performing Lagrangian surgery from $L_2$ to $L_1$ is in general not even smoothly isotopic to  $L_1\#_pL_2$.

Now, we present an new approach of performing Lagrangian surgery which also motivates the definition of Lagrangian surgery along clean intersections.

\begin{defn}\label{d:pointHandle}
Given the zero section $L\subset T^*L$, a Riemannian metric $g$ on $L$ (hence inducing one on $T^*L$) and a point $x\in L$, we define the {\bf flow handle} $H_{\nu}$ with respect to a $\lambda$-admissible function $\nu$ to be
$$H_\nu=\{\phi^{\vp}_{\nu(\| p \|)}(p)\in T^*L: p \in (T_x^*L)_\epsilon\backslash \{x\}  \}, $$
\noindent where $(T_x^*L)_\epsilon$ denotes the cotangent vectors at $x\in L$ with length $\le\epsilon$
\end{defn}

\brmk\label{r:tilde} $\phi^{\vp}_{\nu(||p||)}$ is the time-$1$ Hamiltonian flow of $\wt\nu(||p||)$, where $\wt\nu'(s)=\nu(s)$.  For this reason, the reader should keep in mind that $H_\nu$ is automatically Lagrangian for any choice of admissible $\nu$.  For our purpose, the discussion on $\nu$ will be more flexible so we suppress the role of the actual Hamiltonian function $\wt\nu$ unless otherwise specified.

\ermk

\blem\label{l:flowGluePoint} Let $S_\l(T^*_xL)$ be the radius $\lambda$-sphere in the tangent plane of $x$.  If $exp: S_\l(T^*_xL)\rightarrow L$ is an embedding, and $\partial H_\nu=exp(S_\l(T^*_xL))\subset L$ divides $L$ into two components, then $H_\nu$ glues with exactly one of the components to form a smooth Lagrangian submanifold coinciding with $T^*_x L$ outside a compact set for a $\l$-admissible $\nu$.

\elem

\bpf  The only thing to prove is the smoothness of gluing on $\partial H_\nu=exp(S_\l(T^*_x))$.  Note that near $\partial H_\nu$, the handle is a smooth section over the open shell $exp(B_\l(T^*_x)\backslash B_{\l-\delta}(T^*_x))$ which is a smooth open manifold.  Moreover, the section has vanishing derivatives for all orders on the boundary due to the assumption of admissibility on $\nu(r)$ near $r=0$.  The conclusion follows.

\epf

\begin{example}\label{lemma: flow handle=Lagrangian handle}
One may match the Lagrangian handle $H_\gamma$ and flow handle $H_\nu$ for admissible $\gamma$ and its corresponding admissible $\nu$ (See Definition \ref{d:admissible} and the paragraph after it) via the identification between $T^*\mathbb{R}^n$ and $\mathbb{C}^n$.

To see this, the flow handle is given by
$$H_{\nu}=\{\phi^{\vp}_{\nu(||p||)}(0,p)=(0+\frac{p}{\|p\|}\cdot\nu(||p||),p): p\in (T^*_0\R^n)_\e\}$$

We now identify $T^*\R^n$ with $\C^n$ by sending $(q,p)\mapsto q-ip$, which matches the symplectic form $dp\wedge dq$ and $\frac{1}{-2i}dz\wedge d\bar z=dx\wedge dy$.  Then by definition

\begin{eqnarray*}
(\nu(\| p \|)\frac{p}{\|p\|},p)&\mapsto&\nu(\| p \|)\frac{p}{\|p\|}-ip \\
&=& (a(s)+ib(s))x
\end{eqnarray*}
by a change of variable $s=b^{-1}(-||p||)$ and $x=\frac{p}{\|p\|}$.  By this identification, we will simply use $H_\nu$ to denote both handles.

\end{example}

\begin{corr}\label{c:FlowSurgPT}
Let $L_1,L_2 \subset (M,\omega)$ be two Lagrangians transversally intersecting at $p$.
Under the assumption in Lemma \ref{l:flowGluePoint}, one can obtain a Lagrangian $L_1\#^{\nu}_pL_2$
by gluing (1) $L_2\backslash U$, (2) the Lagrangian flow handle $H_\nu$, and (3) an open set in $L_1$ given by Lemma \ref{l:flowGluePoint}.
For appropriately chosen $\nu$, $L_1\#_p^{\nu}L_2$ coincides with $L_1 \#_pL_2$ defined in Corollary \ref{c:SurgeryAtPoint}.
\end{corr}

\begin{example}\label{e: injectivity radius}
Let $r(p)$ be the injectivity radius of $p$.  For different choices of $\nu(r)$ with $\nu(0)<r(p)$, these handles will define a family of different Lagrangian surgeries which are all Lagrangian isotopic to each other.  In the case when $L_1$ is simply-connected, they are Hamiltonian isotopic.

The situation becomes more interesting when $\nu(0)>r(p)$.  Some simple instances are given by $S=\RP^n$, $\CP^n$ or any finite cover of rank-one symmetric space.  Take $\CP^n$ and its standard Fubini-Study metric as an example, for any $p\in S$, the flow surgery can be performed for $kr(p)<\nu(0)<(k+1)r(p)$ for any $k\in\Z$.  Later we will see that such surgeries are indeed iterated surgeries in the ordinary sense (although surgeries along clean intersections will be involved).

\end{example}

\begin{example}
A less standard example is essentially given by exotic spheres in \cite{Se14}.  Given any $f\in Diff^+(S^{n-1})$ and form an exotic sphere $S_f=B_-\cup_{f} B_+$.  There is a Riemannian metric so that all geodesics starting from $0_\pm$ are closed, through each other, and of the same length (\cite[Lemma 2.1]{Se14}).
 Take $p=0_-\in B_-$, when $\l$ is below the injectivity radius, the surgery is the original one considered in \cite{PolSurg}.  When $\nu(0)>r(p)$, the generalized surgery defined above is identified with an iterated surgery along $p$ and $q=0_+\in B_+$ in a successive order, which is exactly the family constructed in \cite{Se14} by the geodesic flow.

\end{example}

The following lemma can be found in \cite{Se99}, but we feel it instructive to sketch the proof from the point of view of flow handles to make our discussion complete.

\begin{lemma}[\cite{Se99}]\label{l:Surgery=Dehn}
Let $x \in S^n$ be a point and consider $L=\tau_{S^n}(T^*_{x}S^n) \subset T^*S^n$.
Then $S^n \#_x T^*_{x}S^n$ is Hamiltonian isotopic to $L$ by a compactly supported Hamiltonian.
\end{lemma}

\begin{proof}

Let $A: S^n \to S^n$ be the antipodal map.
We consider open geodesic balls $B_{\pi}(x)$ and $B_{\pi}(A(x))$ of radius $\pi$ centered at $x$ and $A(x)$, respectively.
It gives two symplectomorphism $f_{x}: T^*B_{\pi}(x) \to T^*S^n\backslash T^*_{A(x)}S^n$ and $f_{A(x)}: T^*B_{\pi}(A(x)) \to T^*S^n\backslash T^*_{x}S^n$
under which we have
$$f_{x}^{-1}(L)= \{\nu^{Dehn}_{\epsilon}(|p|)\frac{p}{|p|},p) \in T^*B_{\pi}: p \in \mathbb{R}^n\backslash \{0\}\}$$
and
$$
f_{A(x)}^{-1}(L)=\{(\pi-\nu^{Dehn}_{\epsilon}(|p|))\frac{p}{|p|},p) \in T^*B_{\pi}: p \in B_{\epsilon}(0)\}
$$

On the other hand, suppose $\nu=\nu_{\lambda}$ is such that $\nu_{\lambda}(0)=\lambda < \pi=r(x)$.
Then $f_{x}^{-1}(H_{\nu_{\lambda}})$ is given by
$$f_{x}^{-1}(H_{\nu_{\lambda}})= \{(\nu_{\lambda}(|p|)\frac{p}{|p|},p) \in T^*B_{\pi}: p \in \mathbb{R}^n \backslash \{0\}\} \cup \{  (q,0)\in T^*B_{\pi}: q \in B_{\pi}(0) \backslash B_{\lambda}(0)\}
$$

Let $\delta>0$ be such that $\nu_\e^{Dehn}(r)=\pi-r$ for $r<\delta$. We can pick $\nu_{\lambda}$ such that $\nu_{\lambda}(r)=\nu^{Dehn}_{\epsilon}(r)$ for $r \ge \delta$.
The resulting $S^n \#_x T^*_{x}S^n$ hence coincides with $L$ outside $T^*B_{\delta}(A(x))$.
Inside $T^*B_{\delta}(A(x))$, even though $\nu^{Dehn}_{\epsilon}$ is not an admissible function, both $S^n \#_x T^*_{x}S^n$ and $L$ are graphs of the zero section. Therefore, $S^n \#_x T^*_{x}S^n$ is Lagrangian isotopic to $L$ and hence Hamiltonian isotopic to $L$ by a compactly supported Hamiltonian.
\end{proof}

\begin{rmk}\label{r:non-admissibleSurgery}
For semi-admissible $\nu^\a$ that is not admissible, the gluing with $L_1$ cannot be smooth in general.  Lemma \ref{l:Surgery=Dehn} is an instance when a surgery using a semi-admissible profile $\nu^{Dehn}_{\e}$ yields a smooth Lagrangian submanifold.  Intuitively, the lemma regards $\nu^{Dehn}_{\epsilon}$ as a degenerate case of an admissible function.  The point is that, when $\lambda=r(p)$, we only need to glue $Cl(H_{\nu})$ with $L_2\backslash U$, where $Cl(\cdot)$ denotes the closure.

In the case when a semi-admissible function defines a smooth Lagrangian surgery manifold, we will continue to denote it as $L_1\#_p^{\nu^\alpha_\l}L_2$.
This applies to other surgeries along clean intersections and will be used several more times in a parametrized version in the paper.

\end{rmk}

\subsubsection{Surgery along clean intersection}\label{s:ordinaryClean}

Let $L_1$ and $L_2$ be two Lagrangians in $(M,\omega)$ which intersect cleanly at a submanifold $D$.
In other words, we have $T_pD=T_pL_1 \cap T_pL_2$ for all $p \in D$.
The following well-known local proposition due to Pozniak allows us to extend the definition of flow handles to this case.

\begin{prop}[\cite{Poz99}]\label{p:Poz99}
Let $L_1,L_2 \subset (M,\omega)$ be two closed embedded Lagrangians with clean intersection at $L_1 \cap L_2 =D$.
Then there is a symplectomorphism $\varphi$ from a neighborhood $U$ of $0_{section}$in $T^*L_1$ to $M$ such that
$\varphi(0_{section})=L_1$ and $\varphi^{-1}(L_2)\subset N^*_D$, where $0_{section}$ is the zero section and $N^*_D$ is the conormal bundle of $D$ in $L_1$.

\end{prop}

\begin{defn}\label{d:Dflow}
We define the {\bf flow handle for $D\subset L$} with respect to an admissible function $\nu$ to be
$$H_{\nu}^D=\{\phi^{\vp}_{\nu(\| \xi \|)}(\xi)\in T^*L: \xi \in (N^*_D)_{\epsilon} \backslash D  \}$$
where $(N^*_D)_{\epsilon}$ consists of covectors in the conormal bundle of $D$ in $L$ with length $\le \epsilon$.
\end{defn}

\blem\label{l:flowGlueClean} Let $S_{\lambda}(N^*_D)$
be the radius-$\lambda$ sphere bundle in conormal bundle of $D$.
If $exp: S_{\lambda}(N^*_D)\rightarrow L$ is an embedding, and $\partial H_{\nu_\l}^D\subset L$ divides $L$ into two components,
then $H_\nu^D$ glues with exactly one of the components to form a smooth Lagrangian submanifold coinciding with $N^*_D$ outside a compact set.

\elem

The proof is exactly the same as Lemma \ref{l:flowGluePoint} and we omit it.
As in the transversal intersection case, the surgery is always well-defined when we choose $\nu(0)=\lambda<r(D)$,
the injectivity radius of $D$ along normal directions.  Using Proposition \ref{p:Poz99} we globalize the construction as follows.

\begin{corr}\label{c:flowGlueClean}
Let $L_1,L_2 \subset (M,\omega)$ be two Lagrangians intersecting cleanly along $D$.
By choosing a metric on $L_1$, a symplectic embedding $\iota: T_\epsilon^*L_1 \to M$ such that  $\iota(0_{section})=L_1$ and $\iota^{-1}(L_2)\subset N^*_D$,
one can obtain a Lagrangian $L_1\#_D^{\nu} L_2$ by attaching a Lagrangian flow handle $\iota(H^D_{\nu})$ to $(L_1\backslash U_1)\cup ( L_2\backslash U_2)$, with $U_i\subset L_i$ appropriate open neighborhoods of $D$, and $\epsilon$ being sufficiently small.
\end{corr}

As in Example \ref{e: injectivity radius}, we denote $L_1\#_D^{\nu} L_2$ by $L_1\#_D L_2$ if $\lambda < r(D)$.

\subsection{$E_2$-flow surgery and its family version}\label{s:E2surgery}

So far we have only used the geodesic flows on the whole $T^*L$
to construct Lagrangian handles, but more flexibility will prove useful in our applications.
Heuristically, our previous constructions have taken advantage of the fact that $||p||$ has a well-defined Hamiltonian flow on the whole cotangent bundle except for the zero section.
More crucially, the resulting flow handle should have an embedded boundary into $L_1$ (or at least fiber over its image).
Indeed, any Hamiltonian function with such properties will suffice for defining a meaningful Lagrangian handle.

A variant of the flow handle can therefore be defined as follows.
Let $L=K_1^{n-m} \times K_2^m$ be a product manifold equipped with product Riemannian metric.
Then there is an orthogonal decomposition $T^*L=E_1 \oplus E_2$ given by the two factors respectively.
 Let $D\subset L$ be of codimension $m$ and transverse to $\{p\} \times K_2$ for all $p \in K_1$.
 Suppose $\pi_2:T^*L\rightarrow E_2$ be the projection to $E_2$,
 one may then use the function $\vp_{\pi}=||\pi_2(\cdot)||_g$ to define a new flow handle.
 Note that $\vp_{\pi}=||\pi_2(\cdot)||_g$ is smooth on $T^*L\backslash E_1$.

\begin{defn}\label{d:subbundleHandle}
In the situation above, we define the {\bf $E_2$-flow handle for $D$} (or {\bf flow handle along $E_2$-direction}) with respect to an $\lambda$-admissible $\nu_\l$ to be
$$H_{\nu_{\l}}^{D,E_2}=\{\phi^{\vp_{\pi}}_{\nu_{\l}(\| \pi_2(\xi) \|)}(\xi)\subset T^*L: \xi \in (N^*_D)_{\epsilon,E_2} \backslash D  \}.$$
where $(N^*_D)_{\epsilon,E_2}$ consists of covectors $\xi$ in the conormal bundle of $D$ in $L$ such that $\|\pi_2(\xi)\| \le \epsilon$.
\end{defn}

We note that for any point $\xi=(\xi_1,\xi_2) \in E_1 \oplus E_2$,
$\phi^{\vp_{\pi}}_{t}(\xi)=(\xi_1,\phi^{\vp}_{t}(\xi_2))$
so $E_2$-flow is the normalized (co)geodesic flow on the second factor and trivial on the first factor.

Let $S_{\lambda}(E_2|_D)$
be the radius-$\lambda$ sphere bundle of $E_2$ over $D$.
We consider $exp^{E_2}_{\lambda}:S_{\lambda}(E_2|_D) \to L$, which is the exponential map restricted on $S_{\lambda}(E_2|_D)$
along the leaves of the foliation given by second factor.
We define the \textbf{$E_2$-injectivity radius} $r^{E_2}(D)$ of $D$ as the supremum of $\lambda$
such that $exp^{E_2}_{s}$ is an embedding for all $s < \lambda$.

\blem\label{l:flowGlueCleanE2}
Let $D \subset L=K_1 \times K_2$ be of dimension $n-m$ and transversal to $\{p\} \times K_2$ for all $p \in K_1$.
If $exp^{E_2}_\lambda: S_{\lambda}(E_2|_D)\rightarrow L$ is an embedding and
$\partial H_\nu^{D,E_2}\subset L$ divides $L$ into two components,
then $H_\nu^{D,E_2}$ glues with exactly one of the components of $L$ to
form a smooth Lagrangian submanifold coinciding with $N^*_D$ outside a compact set.

\elem

\begin{proof}
The proof is again similar as Lemma \ref{l:flowGlueClean}.
\end{proof}

Similar to the cases we considered before, if $L_1=K_1 \times K_2$ and $L_2$ are Lagrangians cleanly intersecting at $D$ as above,
we can add an $E_2$-flow handle to $L_1 \cup L_2$ outside a tubular neighborhood of $D$ to get
a new Lagrangian submanifold for $\lambda < r^{E_2}(D)$.
We will denote the resulting Lagrangian submanifold by $L_1\#_{D,E_2} L_2$, called the \textbf{surgery from $L_1$ to $L_2$}.
We remark that $H_{\nu}^{D,E_2}$ coincide with $N^*_D$ when $\| \pi_2(\xi) \| > \epsilon$, so for $E_2$ flow surgery,
we have to take out a neighborhood from $N^*_D$ slightly larger than the $\epsilon$-neighborhood of $D$ in $N^*_D$.
This fact is not essential when $\epsilon$ is small.

We now define a family version for $E_2$-flow surgery.  Assume the situation from Definition \ref{d:subbundleHandle} that we have a smooth manifold pair $(L,D)$ and a decomposition $T^*L=E_1\oplus E_2$.  Let $(\eL, \eD)$ be another smooth manifold pair so that:

\begin{enumerate}[(i)]
\item $(\eL, \eD)$ has a compatible fiber bundle structure over a smooth base $B$, that is,

    $$\xymatrix{ L\ar[r]& \eL\ar[d] \\ &B}\hskip 1cm\text{and, }\hskip 1cm \xymatrix{ D\ar[r]& \eD\ar[d] \\ &B}$$
where the two bundle structures are compatible with the inclusion $\eD\hookrightarrow\eL$.
\item The structure group $G\subset Isom(L)$, the isometry group of $L$, and it preserves $E_1$ and $E_2$.

\end{enumerate}

Assumptions above allow us to glue $T^*L$ via the given bundle data, yielding a symplectic fiber bundle $\eE\rightarrow B$ with fiber $T^*L$.  All previous symplectic constructions on $T^*L$ are now functorial hence can be glued over $B$.  For example, $N^*_D L$ glues into $N^*_\eD\eL$ hence fits into Pozniak's setting of clean intersection.  When $\eE$ is regarded as a vector bundle over $\eL$, it comes with a natural splitting $\eE=\eE_1\oplus\eE_2$ from local charts.  Hence, the $E_2$-handle $H_\nu^{D,E_2}$ can be constructed fiberwisely on $N^*\eD$, which gives a smooth handle $\eH_\nu\subset \eE$.  The fact that $\eH_\nu\hookrightarrow T^*\eL\supset\eE$ is indeed a Lagrangian embedding can be check on local charts $\eU\subset B$.

\begin{lemma}\label{l:familyE2flow}
  For two cleanly intersecting Lagrangians $\eL_0, \eL_1 \subset (M^{2n},\w)$, if $(\eL_0, \eD=\eL_0\cap\eL_1)$ satisfies (i)(ii) above,
  then family $E_2$-surgery between $\eL_0$ and $\eL_1$ can be performed and gives a
  Lagrangian submanifold $\eL_0 \#^{\nu}_{\eD,E_2} \eL_1$ of $(M,\w)$.
\end{lemma}

\brmk It is easy to see that our construction works word-by-word as long as there is a decomposition of vector bundle $T^*L=E_1\oplus E_2$.  However, one needs to imposed technical conditions to make $exp_\lambda^{E_2}:S_\lambda\to L$ an embedding even for small $\lambda$.  An easy condition is to assume $E_2$ is integrable at least near $D$, but it should also work in some cases when $E_2$ is completely non-integrable near $D$ but integrable outside a small neighborhood.  Considerations along this line might result in delicate constructions of new Lagrangian submanifolds.

\ermk

\section{Perturbations: from surgeries to Dehn twists}\label{s:AdmissToDehn}

This section contains the technical part which passes from Lagrangian surgeries to Dehn twists in several applications.  The general idea is the same as Lemma \ref{l:Surgery=Dehn}, which may also interpreted as passing from admissible profiles to semi-admissible ones.
This is realized as local perturbations of the surgery Lagrangians.

  We first explain how this works in the $\CP^n$ case, then give a proof of Theorem \ref{t:surgeries}(1)(2)(3)(5) using family versions of this observation.

\subsection{Fiber version}

In this section, we are interested in $L$ being $\mathbb{RP}^n$, $\mathbb{CP}^{\frac{m}{2}}$ or $\mathbb{HP}^n$
equipped with the Riemannian metric such that every geodesic is closed of length $2\pi$.  All actual proofs will be given only in the case of $\CP^n$ but are easily generalized.
Let $x \in L$ be a point and $F_x$ its cotangent fiber inside $T^*L$.
We also let $D=\{y \in L| dist(x,y)=\pi\}$ be the \textbf{divisor opposite to $x$}.% and $L\#^{\nu^\a}F_x=$.

\begin{lemma}\label{l:admissibleToDehnFiber}
Let $x \in L$ be a point and $\nu_{\lambda_i}$ be $\lambda_i$-admissible functions such that $(k-1)\pi< \lambda_i < k\pi$ for some
positive integer $k$ for both $i=1,2$.
Then $L\#^{\nu_{\lambda_i}}_xF_x$ are isotopic for $i=1,2$ by a compactly supported Hamiltonian.

Moreover, if we choose a semi-admissible function $\nu^{\alpha}_{k\pi}: (0,\infty) \to [0,k\pi)$ that is
monotonic decreasing and all orders of derivatives vanish at $r=\epsilon$
such that  $\nu^{\alpha}_{k\pi}(r)=k\pi-\alpha r$ near $r=0$ ($\alpha \ge 0$),
then $L\#^{\nu^{\alpha}_{k\pi}}_xF_x$ (See Remark \ref{r:non-admissibleSurgery})
is a smooth Lagrangian that is isotopic to $L\#^{\nu_{\lambda_i}}_xF_x$ by a compactly supported Hamiltonian.

Furthermore, these Hamiltonian isotopies can be chosen to be invariant under isometric action of $L$ fixing $x$.
\end{lemma}

\begin{corr}\label{c:admissibleToDehnFiber}
For $\pi< \lambda < 2\pi$ and $L$ being $\mathbb{RP}^n$, $\mathbb{CP}^{\frac{m}{2}}$ or $\mathbb{HP}^n$,
$L\#^{\nu_{\lambda}}_xF_x$ is Hamiltonian isotopic to $\tau_L(F_x)$ for an admissible $\nu_\lambda$.
\end{corr}

\begin{proof}
Observe that when $\alpha=1$ and $k=2$, $\nu^{\alpha}_{k\pi}(r)$ is a Dehn twist profile.
The Corollary follows from Lemma \ref{l:admissibleToDehnFiber}.
\end{proof}

\begin{figure}[h]\label{fig: isotopy of admissible function 1}
\centering
\includegraphics[scale=1]{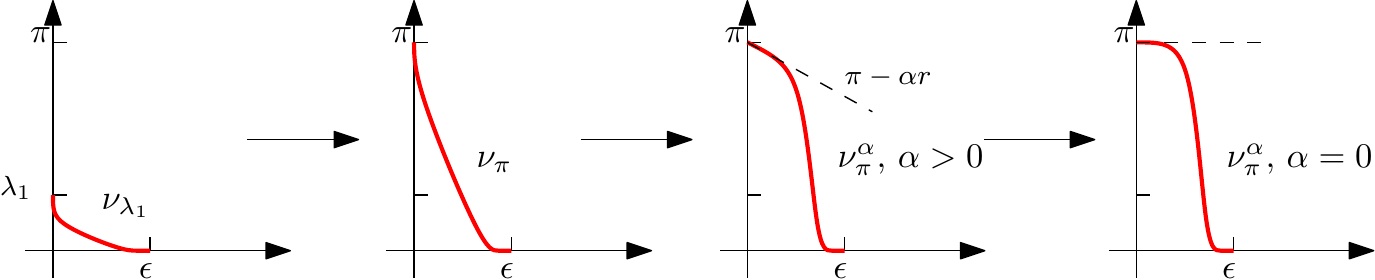}

\caption{Isotopy from $\nu_{\lambda_1}$ to $\nu_{\pi}$ to $\nu^{\alpha}_\pi$ (Dehn twist profile).}
\end{figure}

\begin{proof}[Proof of Lemma \ref{l:admissibleToDehnFiber}]
For the first statement, we observe that the space of $\lambda$-admissible function for $(k-1)\pi< \lambda < k\pi$ is connected.
A smooth isotopy $\{\nu_t\}$ from $\nu_{\lambda_1}$ to $\nu_{\lambda_2}$ in this space results in a smooth Lagrangian isotopy from $L\#^{\nu_{\lambda_1}}_xF_x$ to $L\#^{\nu_{\lambda_2}}_xF_x$ since $\partial H_{\nu^t}$ does not pass any critical locus.
This is a Hamiltonian isotopy because $H^1(L\#^{\nu_{\lambda_1}}_xF_x, \partial^\infty (L\#^{\nu_{\lambda_1}}_xF_x);\R)=0$ (cf. Example \ref{e: injectivity radius}).

For the second statement, we only consider the case that $k=1$ and $L=\mathbb{CP}^{\frac{m}{2}}$, and
the remaining cases are similar.
In this case, denote $\nu^\alpha=\nu_\pi^{\alpha}$, then $Cl(H_{\nu^{\alpha}})\backslash H_{\nu^{\alpha}}=D=\mathbb{CP}^{\frac{m}{2}-1}$.
We pick a local chart $U\subset L$ with local coordinates $(q_1,\dots,q_m)$ adapted to $D$ in the sense that
$U \cap D=\{q_1=q_2=0\}$ and $c(t)=(tq_1,tq_2,q_3,\dots,q_m)$ are normalized geodesics for any $(q_1,\dots,q_m)$.
It induces canonically a Darboux chart $T^*U$ in $T^*L$.
We write a point in $T^*U$ as $(q_a,q_b,p_a,p_b)$, where $q_a=(q_1,q_2)$, $q_b=(q_3,\dots,q_m)$ and similarly for $p_a$ and $p_b$.
Since $H_{\nu^{\alpha}}$ is defined by the geodesic flow, one may directly verify
$$T^*U \cap H_{\nu^{\alpha}}= \{ (q_a,q_b,p_a,0)| q_a=-\alpha p_a \neq 0\}$$
$$T^*U \cap D=\{ (0,q_b,0,0)\}$$
Therefore, it is clear that $H_{\nu^{\alpha}}$ and $D$ can be glued smoothly to become $Cl(H_{\nu^{\alpha}})$.
The gluing from $H_{\nu^{\alpha}}$ to $F_x-B_{\epsilon}$ is the same as the admissible case.
It results in a smooth Lagrangian $L\#^{\nu^{\alpha}}_xF_x$.

Finally, we want to show that $L\#^{\nu^{\alpha}}_xF_x$ is Hamiltonian isotopic to $L\#^{\nu_{\lambda_i}}_xF_x$.
We can assume $\alpha \neq 0$, by a Hamiltonian perturbation if necessary.  Locally near $D$, we have
\beq\label{e:loc} T^*U \cap (H_{\nu^{\alpha}} \cup D)=
\{ (-\alpha p_a,q_b,p_a,0)\}=\{(q_a,q_b,-\frac{1}{\alpha}q_a,0)\}\eeq
\beq T^*U \cap L= \{ (q_a,q_b,0,0) \}\eeq

It is clear that there is a small $\delta >0$
such that $(H_{\nu^{\alpha}} \cup D) \cap T^*B_{\delta}(D)$ is the graph of
$d(-\frac{1}{2 \alpha} dist^2(\cdot,D))$ over $B_{\delta}(D)$, where $B_{\delta}(D)$ is the $\delta$ neighborhood of $D$ in $L$. Take a smooth decreasing function $f(r):[0,\delta]\to\R$ so that $f=0$ near $r=0$ and $f(r)=-\frac{1}{2\a}r$ near $r=\delta$.
Denote $f_t(r)=tf(r)-(1-t)\frac{1}{2\a}r$.

Then the graph of $d(f_t \circ dist^2(\cdot,D))$ can be patched with $H_{\nu^{\alpha}} \backslash T^*B_{\delta}(D)$
to give a Hamiltonian isotopy from $L\#^{\nu^{\alpha}}_xF_x$ to $L\#^{\nu_{\lambda}}_xF_x$
for some admissible $\nu_{\lambda}$ with $0< \lambda < \pi$.
We remark that the Hamiltonian isotopy is invariant under the $Isom(L)_x$, isometric group of $L$ fixing $x$.
This concludes the proof.

\end{proof}

We remark that another point of view of the Lagrangian isotopy from $L\#^{\nu^{\alpha}}_xF_x$ to $L\#^{\nu_{\lambda}}_xF_x$ is that it is induced from a smooth isotopy relative to end points from the curve $\{(r,\nu^{\alpha}(r)) \in [0,\epsilon] \times [0,\pi] | r \in (0,\epsilon] \} \cup \{(0,\pi)\}$ to the $\lambda$-admissible curve defined by $\nu_{\lambda}$.

Later we will see that, when the surgery profile $\nu_\lambda$ has $\lambda$ exceeding the injectivity radius, there is no cobordism directly associated to such a surgery.  To fit such a surgery to the cobordism framework, in general we need to decompose the surgery into several steps.  The following lemma shows how this works in the case for $\CP^n$ (which easily generalizes to $\RP^n$ and $\mathbb{HP}^n$).

\begin{lemma}\label{l:cpnFiber}
Let $x \in \mathbb{CP}^{\frac{m}{2}}$ be a point and $F_x \subset T^*\mathbb{CP}^{\frac{m}{2}}$ the corresponding cotangent fiber.
Let $D=\{y \in \mathbb{CP}^{\frac{m}{2}}| dist(x,y)=\pi\}$ be the divisor opposite to $x$.
Then there is an embedded Lagrangian $Q_x \subset T^*\mathbb{CP}^{\frac{m}{2}}$ such that
\begin{enumerate}[(1)]
\item $Q_x=F_x$ away from a neighborhood of zero section,
\item $Q_x$ is Hamiltonian isotopic to $\mathbb{CP}^{\frac{m}{2}}\#_xF_x$,
\item $Q_x$ intersects cleanly with $\mathbb{CP}^{\frac{m}{2}}$ at $D$,
\item $\mathbb{CP}^{\frac{m}{2}}\#_D Q_x$ is Hamiltonian isotopic to $\tau_{\mathbb{CP}^{\frac{m}{2}}}(F_x)$
\end{enumerate}
As a result, as far as Hamiltonian isotopy class is concerned, we have $\mathbb{CP}^{\frac{m}{2}}\#_D(\mathbb{CP}^{\frac{m}{2}}\#_xF_x)=\tau_{\mathbb{CP}^{\frac{m}{2}}}(F_x)$.
\end{lemma}

\begin{proof}

Choose a semi-admissible profile $\nu^0_\pi$ such that $\nu^0_\pi=\pi$ near $r=0$ and let $Q_x=L\#^{\nu^0_{\pi}}_xF_x$.  (1)(3) follows from definition, and (2) is a consequence of Lemma \ref{l:admissibleToDehnFiber}.

To see (4), note that near $D$,
$Q_x$ coincide with the $\epsilon'$-disk conormal bundle at $D$ for some $\epsilon'\ll\epsilon$.  Therefore, $\CP^{m/2}\#^{\nu_\lambda}_D Q_x$ is identical to $\CP^{m/2}\#_D^{\nu_{\lambda+\pi}}F_x$ for any $0<\lambda<\pi$ and an appropriate choice of $\nu_{\lambda+\pi}$ (see Figure \ref{fig:isotopy-admToDT2} for the demonstration).  The latter is then Hamiltonian isotopic to $\tau_{\CP^{m/2}}(F_x)$ by Corollary \ref{c:admissibleToDehnFiber}.
\end{proof}

\begin{figure}[h]
\centering
\includegraphics[scale=1.2]{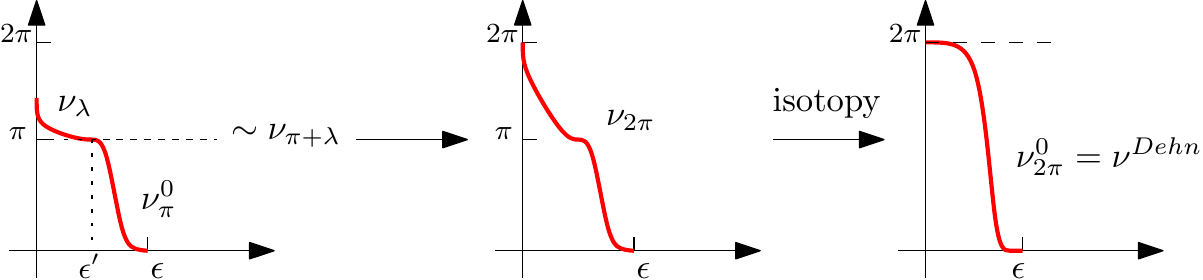}
\caption{Left and middle: identification of $\CP^{m/2}\#_D^{\nu_\lambda}G_x$ and $\CP^{m/2}\#_D^{\nu_{\lambda+\pi}}F_x$.  Right: isotopy from an admissible function to a Dehn twist profile.}
\label{fig:isotopy-admToDT2}
\end{figure}

Due to the symmetry of $\CP^{m/2}$, we have an alternative description to the Dehn twist of $F_{x_0}$, yielding another proof for Lemma \ref{l:cpnFiber}.  Essentially, this description only changes the role of the base and the fiber, but leads to a particularly handy criterion for the isotopy type of $\tau_{\CP^{m/2}}F_{x_0}$, which will be used in later sections. Denote the isometry group of $\CP^{m/2}$ as $G$ and the subgroup of it fixing $x_0$ as $G_{x_0}$. There is an induced $G_{x_0}$-action on $T^*\CP^{m/2}$.

\blem\label{l:rotLag}Let $\gamma(t)$ be a normalized geodesic on $\CP^{m/2}$ starting and ending at $x_0$.  Let $c(t)$ be a (rescaled) lift of $\gamma(t)$ in $T^*\CP^{m/2}$, that is, $c(t)=(\gamma(t),f(t)\gamma'(t))$ for some smooth $f(t)$ defined on $[0,2\pi]$ such that $f(\pi)\neq0$ and $f(0)>0$ (recall that $\gamma'(t)$ is identified with its dual).
Then the orbit $G_{x_0}\cdot c(t)$ is a Lagrangian which is possibly immersed.

Moreover, assume further

\begin{enumerate}[(a)]
\item $$\frac{d^n(f^{-1}(t))}{dt^n}(f(0))=0$$ for all $n\ge1$, and $f$ is strictly decreasing near $t=0$,
\item $f(2\pi)=0$ and $f'(t)<0$ when $t\in(2\pi-\delta,2\pi]$ for some small $\delta>0$.
\end{enumerate}

Then $G_{x_0} \cdot c(t)$ can be extended to a proper Lagrangian immersion $L_f$, such that
\begin{enumerate}[(i)]
\item it overlaps with $F_{x_0}$ along $\{ p \in F_{x_0}: \|p\|>f(0) \}$,
\item it is isotopic to $\tau_{\CP^{\frac{m}{2}}}F_{x_0}$ through a smooth family of Lagrangian immersions with property (i).
\end{enumerate}
\elem

\bpf For the first assertion, $G_{x_0}\cdot c(t)$ is the graph of $d(F \circ dist_{x_0}(\cdot))$ when $t\in(0,\pi)$, and $d(F \circ (2\pi-dist_{x_0}(\cdot)))$ when $t\in(\pi,2\pi)$, for $F'(t)=f(t)$.  The condition (a) guarantees the smoothness of gluing with $F_{x_0}$, and (b) the smoothness at $t=2\pi$.  The smoothness and Lagrangian properties at the critical set $D$ and $x_0$ can be checked directly and using that $f(\pi) \neq 0$ and Lagrangian property is a closed condition.

For the last isotopy statement, find an isotopy of smooth functions, within the class of those satisfying $(a)(b)$, from $f(t)$ to some $g(t)$ which is strictly monotonic (decreasing) in $[0,2\pi]$.  This induces an isotopy of Lagrangian immersions from $L_f$ to some $L_g$.

Consider $L\#^{\nu^{Dehn}}_{x_0}F_{x_0}$ as in Corollary \ref{c:admissibleToDehnFiber}.  This Lagrangian and $L_g$ are both $G_{x_0}$-invariant, it is not hard to check that with $\nu^{Dehn}=g^{-1}$, the two Lagrangians coincide.  The conclusion hence follows.
\epf

\brmk\label{r:rotLag} As the proof showed, one should heuristically regard $f(t)$ as the inverse function of certain admissible function $\nu$.

The possible immersion points appears if and only if there is $t_0<\pi$, such that $f(t_0)=-f(2\pi-t_0)$.  Otherwise, all above assertions can be improved to the class of embedded Lagrangians.

\ermk

\subsection{Product version}

In this section we prove Theorem \ref{t:surgeries} (1)(3). The proofs here are similar to that in the last section, and should be considered as family versions of it.
In this subsection, we use $S$ to denote $S^n$, $\mathbb{RP}^n$, $\mathbb{CP}^{\frac{m}{2}}$ or $\mathbb{HP}^n$
equipped with the Riemannian metric such that every closed geodesic is of length $2\pi$.

For the moment, let $S \subset (M,\w)$ be a Lagrangian sphere.
One may consider the clean surgery of $L_1=S\times S^-$ and $L_2=\Delta$ in $M\times M^-$.
In this case, they cleanly intersect at $D=\Delta_{S} \subset S \times S^-$.
In Definition \ref{d:subbundleHandle}, take $E_2=S \times (T^*S)^- \subset T^*S\times (T^*S)^-$,
$E_1=T^*S \times S^- \subset T^*S\times (T^*S)^-$
and an $\pi$-admissible function $\nu_\pi$.

Now consider a point $(p,p)\in \Delta$ in a Weinstein neighborhood of $L_1$, where $p$ can be considered as a point on $T_\epsilon^*S$ for a small $\epsilon>0$.
The flow in Definition \ref{d:subbundleHandle} defines a symplectomorphism fixing the first coordinate in $(T^*S \times (T^*S)^-) \backslash E_1$;
when restricted to $\Delta_{S}$, the
$E_2$-flow sends $(p,p)\mapsto(p, \phi^{\vp}_{\nu_\pi(||p||)}(p))$.  This is exactly the graph of $\tau_{S}^{-1}$ (the inverse owes to the negation of symplectic form on $M^-$), except that we have used an admissible profile for the handle which is not a Dehn twist profile. Lemma \ref{l:adjust} below ensures that this could be compensated by a local Hamiltonian perturbation.
Hence modulo Lemma \ref{l:adjust}, this shows that $(S\times S^-)\#^{\nu_{\pi}}_{\Delta_{S},E_2}\Delta=Graph(\tau^{-1}_{S})$.
The whole construction applies when $S$ is $\mathbb{RP}^n$, $\mathbb{CP}^{\frac{m}{2}}$ or $\mathbb{HP}^n$, except that the admissible profile has $\nu(0)=2\pi$.

\begin{figure}[h]
\includegraphics[scale=0.6]{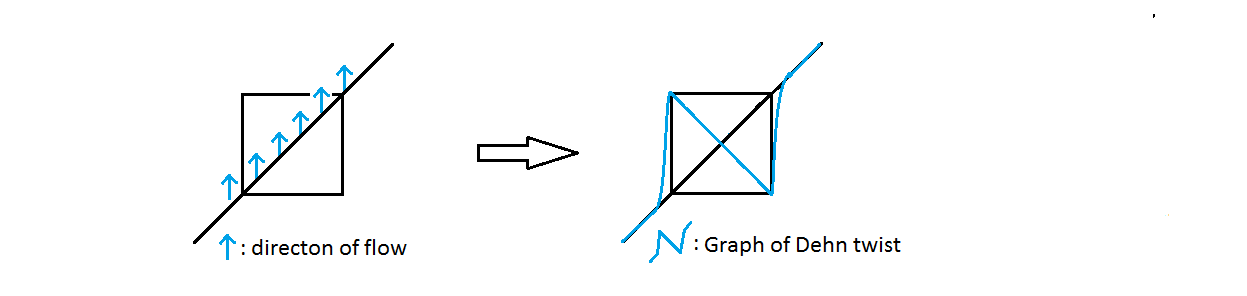}
\caption{$E_2$-flow surgery resulting the graph of Dehn Twist.}
\label{fig: GraphofDehnTwist}
\end{figure}

\begin{lemma}\label{l:adjust}
Let $S$ be $S^n$, $\mathbb{RP}^n$, $\mathbb{CP}^{\frac{m}{2}}$ or $\mathbb{HP}^n$.
Let $\nu_{\lambda_i}$ be $\lambda_i$-admissible functions such that $(k-1)\pi< \lambda_i < k\pi$ for some positive integer $k$ for both $i=1,2$.
Then the $E_2$-flow surged Lagrangian manifolds $(S\times S^-) \#^{\nu_{\lambda_i}}_{\Delta_S,E_2} \Delta$ above by $\nu_{\lambda_i}$ are Hamiltonian isotopic.

Moreover, if we choose a semi-admissible function $\nu_{k\pi}^{\alpha}$ such that  $\nu_{k\pi}^{\alpha}(r)=k\pi-\alpha r$ near $r=0$ ($\alpha \ge 0$),
then $(S\times S^-) \#^{\nu_{k\pi}^{\alpha}}_{\Delta_S,E_2} \Delta$
is a smooth Lagrangian that is Hamiltonian isotopic to $(S\times S^-) \#^{\nu_{\lambda_i}}_{\Delta_S,E_2} \Delta$.

Furthermore, these Hamiltonian isotopy can be chosen to be $Isom(S)$ invariant, where $Isom(S)$ is the diagonal isometry group in $Isom(S)\times Isom(S)$
acting on $T^*S \times (T^*S)^-$.
\end{lemma}

We have the following Corollary whose proof is similar to Corollary \ref{c:admissibleToDehnFiber}

\begin{corr}\label{c:admissibleToDehnProduct}
For $\pi< \lambda < 2\pi$ and $S$ being $\mathbb{RP}^n$, $\mathbb{CP}^{\frac{m}{2}}$ or $\mathbb{HP}^n$ (resp. $0< \lambda < \pi$ and $S=S^n$),
$(S \times S^-)\#_{\Delta_S,E_2}^{\nu_{\lambda}}\Delta$ is Hamiltonian isotopic to $Graph(\tau_S^{-1})$.
\end{corr}

\begin{proof}[Proof of Lemma \ref{l:adjust}]
The proof of the first statement is exactly the same as Lemma \ref{l:admissibleToDehnFiber}.
For the second statement, we again only consider the case that $k=1$ and $S=\mathbb{CP}^{\frac{m}{2}}$ and the remaining cases are similar.

Define $D^{op}=\{(x,y) \in S \times S| dist(x,y)=\pi\}$.
Projection to $x$ equips $D^{op}$ with a $\mathbb{CP}^{\frac{m}{2}-1}$-bundle structure over $S=\mathbb{CP}^{\frac{m}{2}}$.
Therefore, the neighborhood of $D^{op}$ in $S \times S$ is the total space of a fiber bundle $\wt\eV\rightarrow S$, whose fiber is a topological $\cO(1)$-bundle $\eV$ over $\mathbb{CP}^{\frac{m}{2}-1}$.
We pick a local trivialization $U^B \times U^F$ of $\wt\eV$ for $U^B\subset S$ and $U^F\subset \eV$.
Readers should note that, the product decomposition $U^B\times U^F$ is not compatible with that of $L=S\times S^-$
, but $\{q\} \times U^F$ is an open set of the second factor of $S$ for any $q \in U^B$.

Consider a choice of local coordinates
$(q^B,q^F)=(q^B_1,\dots,q^B_m,q^F_1,\dots,q^F_m)$
adapted to $D^{op}$ in the sense that
$(U^B \times U^F) \cap D^{op}=\{q^F_1=q^F_2=0\}$ and $c(t)=(q^B,tq^F_1,tq^F_2,q^F_3,\dots,q^F_m)$ are normalized geodesics for all  $(q^B,q^F)$.
It induces canonically a Weinstein neighborhood $T^*(U^B \times U^F)$ in $T^*(S\times S^-)$.
%(note: the product in $U \times V$ do not respect to the product %of $L \times L$).
We write a point in $T^*(U^B \times U^F)$ as $(q^B,p^B,q^F_a,q^F_b,p^F_a,p^F_b)$, where $q^F_a=(q^F_1,q^F_2)$, $q^F_b=(q^F_3,\dots,q^F_m)$ and similarly for $p^F_a$ and $p^F_b$.
Since $H_{\nu^{\alpha}}$ is defined by the parametrized geodesic flow when restricted on the second factor of $T^*S$, we have a parametrized version of \eqref{e:loc}
\beq\label{e:E2local} \begin{aligned}&T^*(U^B \times U^F)  \cap H_{\nu^{\alpha}}\\
=&\{ (q^B,p^B,-\alpha p^F_a,q^F_b,p^F_a,0)| p^B \neq 0, \phi^{\vp}_{\nu^{\alpha}(\|p^B\|)}(q^B,p^B)=(-\alpha p^F_a,q^F_b,p^F_a,0)\}\end{aligned}\eeq

Here, both $\p^{\vp_B}_{\nu^\a_\pi(||p^B||)}(q^B,p^B)$ and $(-\alpha p^F_a,q^F_b,p^F_a,0)$ are considered as points in $T^*_\e S$ although they belong to different factors.
Therefore, in $ T^*(U^B \times U^F)  \cap H_{\nu^{\alpha}}$,
fixing $q^B$ and letting $p^B$ go to $0$ linearly leads to
fixing $q^F_b$ and letting $p^F_a$ go to zero linearly.  All above conclusions can be glued across different charts.  Therefore, one can see that $H_{\nu^{\alpha}}$ and $D^{op}$ can be glued smoothly.
The fact that $H_{\nu^{\alpha}}$ can be glued smoothly with $\Delta$ is
because all order of derivatives of $\nu^{\alpha}$ vanish at $r=\epsilon$.
%, where $\Delta_{B_{\epsilon}(S)}=\{(p,p)\in\Delta:||\pi_2(p)||<\e\}$ is the $\epsilon$ neighborhood of
%$\Delta_S$ in $\Delta$.
It results in a smooth Lagrangian, which we denote as $(S\times S^-) \#^{\nu^{\alpha}}_{\Delta_S,E_2} \Delta$.

Finally, to show that $(S\times S^-) \#^{\nu^{\alpha}}_{\Delta_S,E_2} \Delta$ is Hamiltonian isotopic to $(S \times S^-) \#^{\nu_{\lambda_1}}_{\Delta_S,E_2} \Delta$.
We can choose $\{\nu^{t}\}_{t \in [\alpha,\infty) }$ interpolating $\nu^\alpha$ and $\nu_{\l_1}$
as in the proof of Lemma \ref{l:admissibleToDehnFiber}.
This is a smooth Lagrangian isotopy which is invariant under the diagonal $Isom(S)$ action.

\end{proof}

In parallel to Lemma \ref{l:cpnFiber}, we have the following.

\begin{lemma}\label{l:CPnCob}
Let $S$ be $\mathbb{RP}^n$, $\mathbb{CP}^{\frac{m}{2}}$ or $\mathbb{HP}^n$ and $D^{op}=\{(x,y)\in S \times S^-| dist(x,y)=\pi\}$.
Up to Hamiltonian isotopy in $T^*S \times (T^*S)^{-}$, we have
$(S \times S^-) \#_{D^{op},E_2}(( S \times S^-)
\#_{\Delta_S,E_2} \Delta)=Graph(\tau_{S}^{-1})$.
\end{lemma}

\begin{proof}
The proof is similar to that of Lemma \ref{l:cpnFiber} and we again assume $S=\mathbb{CP}^{\frac{m}{2}}$.
We use the function $\nu^{0}_{\pi}$ in Lemma \ref{l:cpnFiber} to define
$\cS=( S \times S )
\#_{\Delta_L,E_2}^{\nu^{0}_{\pi}} \Delta$, which is Hamiltonian isotopic to $( S \times S^-)
\#_{\Delta_S,E_2} \Delta$ by Lemma \ref{l:adjust}.
Now, $\cS$ intersects $S\times S^-$ cleanly along $D^{op}$.
We can perform another $E_2$-flow surgery using semi-admissible profiles from $S \times S^-$ to $\cS$ to obtain the result, by Lemma \ref{l:adjust} and Lemma \ref{l:cpnFiber}.

\end{proof}

\subsection{Family versions}\label{s:familyTwist}

One may also generalize the above example to the case of family Dehn twists \cite{WWfamily}.
Recall that a \textbf{spherically fibered coisotropic manifold} $i:C^{2n-l}\hookrightarrow M^{2n}$ is a coisotropic submanifold so that there is a fibration $\rho: C\rightarrow B^{2n-2l}$ over a symplectic base, while the fibers are null-leaves $S^l$.  In other words, $\rho^*\w_B=i^*\w_M$.  Moreover, we equip the fibers with round metric such that geodesics are closed of length $2\pi$ and ask the structure group of $\rho$ lies in $SO(l+1)$.

A neighborhood $U$ of $C$ can be symplectically identified with $T_\epsilon^*S^l\times_{SO(l+1)} P$, where $P$ is the principal $SO(l+1)$-bundle associated to $C$ and $T_\epsilon^*S^l$ consists of the cotangent vectors with norm less than $\epsilon$.  The \textbf{family Dehn twist $\tau_C$} can then be defined fiberwisely as the fiberwise Hamiltonian function $\wt\nu^{Dehn}_\epsilon(||p||)$ (see Remark \ref{r:tilde}) is preserved by the structure group.
With respect to the fiberwise metric $g^v$, the function $h(\cdot)=\wt\nu(||\cdot||_{g^v})$ defines a flow along fibers whose
time-$1$ map is the desired Dehn twist (with a continuation over $C$ defined by fiber-wise antipodal map on $C$).

Now consider a Lagrangian embedding $\wt C:=C\times_B C\hookrightarrow M\times M$.  Explicitly, the image of this map is

$$\wt C=\{(x, y)\in C\times C\subset M\times M: \pi(x)=\pi(y)\},$$

\noindent where $\pi: C\rightarrow B$ is the $S^l$-bundle projection.
Indeed, $\wt C=C^t\circ C$ is a composition Lagrangian in the sense of \eqref{e:composition}.  Here we have abused the notation by identifying $C$ with its Lagrangian image in $B\times M$ defined by

$$\{(x, y)\in B\times C\subset B\times M: \pi(y)=x\}.$$

Note that $\wt C$ is a fiber bundle with fiber $S^l\times S^l$ and structural group the diagonal $SO(l+1)$.

Consider a point $(x, p)\in U$ where Dehn twist is performed.  Here $x\in B$ and $p\in T_\epsilon^*S^l$: this is an abuse of notation because $p$ is only well-defined up to an $SO(l+1)$ action.  Any point contained in $\Delta\cap (U\times U^{-})\subset M\times M^-$ thus takes the form $((x,p),(x,p))$.
In this setting, the graph of $\tau_C^{-1} $ in $ U \times U^-$ consist of points
$$Graph(\tau_C^{-1})=\{((x,p),(x,\phi^{\vp}_{\nu^{Dehn}(\|p\|)}(p)))| x\in B, p\in T_\epsilon^*S^l\}$$
where $p$ is again only well-defined up to an $SO(l+1)$ action, and $Graph(\tau_C^{-1})$ coincides with $\Delta$ outside $U \times U^-$.

Similar as before, we want to realize $Graph(\tau_C^{-1})$ as a surgery from $\wt{C}$ to $\Delta$.
In this case, we want to perform a family $E_2$-surgery.

In the notation of Section \ref{s:E2surgery}, let $\eL=\wt C, \eD=\Delta\cap\wt C$.  The $S^l\times S^l$-bundle structure (over $B$) on $\wt C$ is the needed bundle structure on $\eL$.  The restriction of $\Delta$ on a fiber $T^*_\epsilon\eL_b$ for $b\in B$ is precisely $\Delta_{T^*_\epsilon S^l}$, and the fiberwise $E_2$-flow is taken as the $E_2$-flow along the second factor of $S^l$, as described in the previous subsection.

Hence the whole situation on a fiber $T_\epsilon^*\eL_b=T_\epsilon^*(S^l\times S^l)$ is identical to the one in the previous subsection and defines a global Lagrangian handle $\eH_{\nu_\pi}$ by patching the local trivializations.

By the same token, we can define {\bf projectively fibered coisotropic manifold} which is a coisotropic manifold with null-leaves complex (or real, quaternionic) projective spaces. Family Dehn twists for these spaces are defined similarly.

\begin{lemma}\label{l:admissibleToDehnFamily}
Let $C \subset (M,\omega)$ be a spherically (resp. projectively) coisotropic submanifold with base $B$.
Let $\nu_{\lambda_i}$ be $\lambda_i$-admissible functions such that $(k-1)\pi< \lambda_i < k\pi$ for some positive integer $k$ for both $i=1,2$.
Then the family $E_2$-flow surged Lagrangian manifolds $\widetilde{C} \#^{\nu_{\lambda_i}}_{\eD,E_2} \Delta$ above by $\nu_{\lambda_i}$ are Hamiltonian isotopic.

Moreover, if we choose a semi-admissible function $\nu_{k\pi}^{\alpha}: (0,\infty) \to [0,k\pi)$ such that
$\nu_{k\pi}^{\alpha}(r)=k\pi-\alpha r$ near $r=0$ ($\alpha \ge 0$), then
$\widetilde{C} \#^{\nu^{\alpha}_{k\pi}}_{\eD,E_2} \Delta$
is a smooth Lagrangian that is Hamiltonian isotopic to $\wt{C} \#^{\nu_{\lambda_i}}_{\eD,E_2} \Delta$ .
\end{lemma}

\begin{corr}\label{l:fiberedTwist}
For spherically (resp. projectively) coisotropic submanifold $C$, the family $E_2$-flow clean surgery $\wt C\#_{\eD,E_2} \Delta$ (resp. $\wt C\#_{\eD^{op},E_2} \wt C\#_{\eD,E_2} \Delta$) is Hamiltonian isotopic to $Graph(\tau_C^{-1})$.
Here $\eD^{op}$ is a $D^{op}$-bundle over the base $B$ and $D^{op}$ is as in Lemma \ref{l:CPnCob}.
\end{corr}

\begin{proof}[Proof of Lemma \ref{l:admissibleToDehnFamily}]

We give the proof for the spherical case and the other cases are similar.
Since the construction in Lemma \ref{l:adjust} is $SO(l+1)$ invariant, we can apply Lemma \ref{l:adjust} to $\wt{C}$ and $\Delta \cap \wt{U}$
to obtain the desired Lagrangian isotopy from $\wt{C} \#^{\nu_{\lambda_1}}_{\eD,E_2} \Delta$ to $\wt{C} \#^{\nu_{\lambda_2}}_{\eD,E_2} \Delta$
and from $\widetilde{C} \#^{\nu^{\alpha}}_{\eD,E_2} \Delta$  to $\widetilde{C} \#^{\nu_{\lambda_i}}_{\eD,E_2} \Delta$.

What remains to show is that the Lagrangian isotopies are Hamiltonian isotopies.
We prove the case where the Lagrangian isotopy is from
$\widetilde{C} \#^{\nu^{\alpha}}_{\eD,E_2} \Delta$  to $\widetilde{C} \#^{\nu_{\lambda_i}}_{\eD,E_2} \Delta$.
The other case is similar.
Denote the Lagrangian isotopy as $\iota_{\eL,t} : \eL \to M \times M^-$.
Notice that the Lagrangian isotopy $\iota_{\eL,t}$ restricting to each fiber $\iota_{L,t}:L \to T^*S^l \times (T^*S^l)^{-1}$ is a Hamiltonian isotopy and hence an exact isotopy
(i.e $\alpha_0=\iota_{L,t}^*(\omega_{can} \oplus -\omega_{can})(\frac{\partial \iota_{L,t}}{\partial t},\cdot)$ is exact).
Since the fiberwise symplectic form and the isotopy are $SO(l+1)$-invariant, so is $\alpha_0$ and its primitive.
These primitives on fibers can be patched together to form the primitive of
$\alpha=\iota_{\eL,t}^*(\omega_{M} \oplus -\omega_{M})(\frac{\partial \iota_{\eL,t}}{\partial t},\cdot)$
and hence $\iota_{\eL,t}$ is an exact, thus a Hamiltonian isotopy.

Alternatively, one may also patch the Hamiltonian isotopy from Lemma \ref{l:adjust}.  We will leave the details to the reader.

\end{proof}

Corollary \ref{c:fiberedTwist} is now an immediate consequence of \cite{l:admissibleToDehnFamily} by setting $k=1$ for spherical case and $k=2$ for the projective cases.

\section{Gradings and energy}\label{s:gradingEnergy}

In this section we discuss the gradings in Lagrangian surgeries.  We follow mostly the exposition in \cite{AB14} to review the definition of gradings in subsection \ref{s:reviewGrading}.
The subsequent subsections provide computation for a sufficient and necessary criterion to perform graded surgeries.
Starting from the next section, all surgeries between graded Lagrangian will be graded surgeries.
Our discussion stay in the $\Z$-graded case but the corresponding results for $\Z/N$-gradings can be obtained by modifying our argument using the setting in \cite{SeGraded} and the statements will be a mod-$N$ reduction of what we have here.

\subsection{Basic notions}\label{s:reviewGrading}

Let $(M^{2n},\omega)$ be an exact symplectic manifold with a primitive one form $\alpha$ for $\omega$, equipped with an $\omega$-compatible almost complex structure $J$ making $M$ pseudo-convex at infinity.
We also assume $2c_1(M)=0$ and fix once and for all a nowhere-vanishing section $\Omega^2$ of $(\Lambda_{\mathbb{C}}^{top}(T^*M,J))^{\otimes 2}$.

Let $L$ be a connected manifold without boundary (not necessarily compact) and $\iota_L:L \to M$ a proper exact Lagrangian immersion (i.e. $\iota^*_L \alpha$ is exact).
A {\bf grading} on $(L,\iota_L)$ (sometimes simply denoted as $\iota_L$) is defined as a continuous function $\theta_L: L \to \mathbb{R}$ such that $e^{2\pi i \theta_L}=Det^2_{\Omega} \circ D\iota_L$, where $Det^2_{\Omega}$ is defined as
$$Det^2_{\Omega}(\Lambda_p)=Det^2_{\Omega}(X_1,\dots,X_n)=\frac{\Omega(X_1,\dots,X_n)^2}{\|\Omega(X_1,\dots,X_n)^2\|} \in S^1$$
for any Lagrangian plane $\Lambda_p \subset T_pM$ at a point $p$ and any choice of a basis $\{X_1,\dots,X_n\}$ for $\Lambda_p$.

Given two transversal Lagrangian planes $\Lambda_0,\Lambda_1$ (of dimension $n$) at the same point with a choice of $\theta_0,\theta_1$ such that $e^{2\pi i \theta_j}=Det^2_{\Omega}(\Lambda_j)$ for both $j$, we can identify them as graded Lagrangian vector subspaces of $\C^n$.
The index of the pair $(\Lambda_0,\theta_0)$ and $(\Lambda_1,\theta_1)$ is defined as
\beq\label{e:index}Ind((\Lambda_0,\theta_0),(\Lambda_1,\theta_1))=n+\theta_1-\theta_0-2Angle(\Lambda_0,\Lambda_1)\eeq
where $Angle(\Lambda_0,\Lambda_1)=\sum\limits_{j=1}^n \beta_j$ and $\beta_j \in (0,\frac{1}{2})$ are such that there is a unitary basis $u_1,\dots,u_n$ of $\Lambda_0$ satisfying $\Lambda_1=Span_{\mathbb{R}}\{e^{2\pi i \beta_j}u_j\}_{j=1}^n$.

In general, when $\Lambda_0 \cap \Lambda_1 =\Lambda \neq \{0\}$, the definition of index for the pair
$(\Lambda_0,\theta_0)$ and $(\Lambda_1,\theta_1)$ is the same as above with the definition of $Angle(\Lambda_0,\Lambda_1)$ modified as follows.
Pick a path of Lagrangian planes $\Lambda_t$ from $\Lambda_0$ to $\Lambda_1$ such that

$\bullet$ $\Lambda \subset \Lambda_t \subset \Lambda_0 + \Lambda_1$ for all $t \in [0,1]$, and

$\bullet$ the image $\overline{\Lambda_t}$ of $\Lambda_t$ inside the symplectic vector space $(\Lambda_0+\Lambda_1)/\Lambda$ is the positive definite path from $\overline{\Lambda_0}$ to $\overline{\Lambda_1}$.

Let $\beta_t$ be a continuous path of real numbers such that $e^{2 \pi i \beta_t}=Det^2_{\Omega}(\Lambda_t)$. Then, the Lagrangian angle is defined as
$$2Angle(\Lambda_0,\Lambda_1)=\beta_1-\beta_0$$

\begin{defn}
For two graded Lagrangian immersions $(\iota_{L_1},\theta_1)$, $(\iota_{L_2},\theta_2)$ (not necessarily distinct), and points $p_j \in L_j$ for $j=1,2$ such that $\iota_1(p_1)=\iota_2(p_2)=p$,
the index for the ordered pair $(p_1,p_2)$ is
$$Ind(p_1,p_2)=Ind(((\iota_1)_*T_{p_1}L_1,\theta_1(p_1)),((\iota_2)_*T_{p_2}L_2,\theta_2(p_2)))$$
\end{defn}

We also use the notation $Ind(L_1|_p,L_2|_p)$ to denote $Ind(p_1,p_2)$ if $\iota_1^{-1}(p)=\{p_1\}$ and $\iota_2^{-1}(p)=\{p_2\}$.
Note that if $L_1$ intersects $L_2$ cleanly along $D$ and if $D$ is connected, than $Ind(L_1|_p,L_2|_p)=Ind(L_1|_q,L_2|_q)$ for all $p,q \in D$.
In this case, we denote the index as $Ind(L_1|_D,L_2|_D)$.

\begin{eg}\label{e:gradingShift}
For a graded Lagrangian immersion $(\iota_{L},\theta)$ and an integer $k$, $\iota_{L}[k]$ is defined as $\iota_{L}[k]=(\iota_{L},\theta-k)$.
In particular, we have $$Ind(\iota_{L_1}[k]|_D,\iota_{L_2}[k']|_D)=Ind(\iota_{L_1}|_D,\iota_{L_2}|_D)+k-k'$$
\end{eg}

\begin{eg}\label{e:IndexRnNRk}
Let $M=\mathbb{C}^{n}$ be equipped with the standard symplectic form, complex structure and complex volume form.
Let $L_1=\mathbb{R}^{n}=\{y_1=\dots=y_n=0\}$ and $L_2=\{x_1=\dots=x_{n-k}=y_{n-k+1}=\dots=y_n=0\}$ be
two Lagrangian planes for some $0 \le k \le n$.
We have $Det^2_{\Omega}(L_1)=1$ and $Det^2_{\Omega}(L_2)=(-1)^{n-k}$.
Let $\theta_{L_1}=n-k-1$ and $\theta_{L_2}=\frac{n-k}{2}$ be the grading of $L_1$ and $L_2$.
Then, we have $Ind(L_1|_0,L_2|_0)=(n)+\frac{n-k}{2}-(n-k-1)-2(n-k)(\frac{1}{4})=k+1$.
\end{eg}

\begin{rmk}\label{r:propagateGrading}
For a Lagrangian isotopy $\Phi=(\Phi_t)_{t \in [0,1]}: L \times [0,1] \to (M,\w)$, if $\Phi_0$ is equipped with grading $\theta_0$, then
the induced grading on $\Phi_1$ is defined as follows.
There is a uniquely way to extend $\theta_0: L \times \{0\} \to \R$ continuously to $\theta: L \times [0,1] \to \R$
such that $e^{2\pi i \theta(\cdot,t)}=Det^2_{\Omega} \circ D\Phi_t(\cdot)$ and the induced grading on $\Phi_1$ is defined by $\theta(\cdot,1)$.
\end{rmk}

\begin{eg}\label{Example: perturbation of double cover}
Let $L=\mathbb{R} \subset (\mathbb{R}^2,dx \wedge dy)$ and identify the latter with $\mathbb{C}$ equipped with the standard complex volume form.
Consider $h: \R \to \mathbb{R}$ given by $h(q)=c\frac{q^2}{2}$ for some constant $c$.
The graph of $dh$, $Graph(dh)$, is given by $\{(q,cq) \in T^*L| q \in L \}$.
By letting $q=x$ and $p=-y$ to identify $\mathbb{C}$ with $T^*L$, $Graph(dh)$ is given by $\{(x,-cx) \in \mathbb{C}\}$.
Under our convention of Hamiltonian flow, $Graph(dh)$ is the time $1$ Hamiltonian flow of $L$ under Hamiltonian $-h \circ \pi: T^*L \to \R$, where $\pi:T^*L \to L$ is the projection.
If we give a grading to $L$ and induces it to a grading on $Graph(dh)$ by the Hamiltonian isotopy,
then
\begin{displaymath}
Ind(L|_0, Graph(dh)|_0)= \left \{
\begin{array}{lr}
1  & \text{ if $c \le 0$} \\
0 & \text{ if $c >0$}
\end{array}
\right.
\end{displaymath}
In short, the index equals the Morse index of $h$ if $c \neq 0$.  We call the grading defined above an \textbf{induced grading on $Graph(dh)$}.
\end{eg}

\begin{eg}\label{e:MorseBottIndex}
Let $L=\mathbb{R}^n \subset (\mathbb{C}^n,dx_i \wedge dy_i)$.
Consider $h: L \to \mathbb{R}$ given by $h(q)=c\sum\limits_{j=1}^k \frac{q_j^2}{2}$.
If we let $q_i=x_i$ and $p_i=-y_i$ to identify $\mathbb{C}^n$ with $T^*L$
and give the induced grading to $Graph(dh)$ by a grading of $L$ and the Hamiltonian isotopy induced by $-h \circ \pi$,
then
\begin{displaymath}
Ind(L|_{\R^{n-k}}, Graph(dh)|_{\R^{n-k}})= \left \{
\begin{array}{lr}
n & \text{ if $c \le 0$} \\
n-k & \text{ if $c >0$}
\end{array}
\right.
\end{displaymath}
where $\R^{n-k}$ is the last $n-k$ $q_i$ coordinates.
\end{eg}

\begin{corr}\label{c:MorseBott}
Let $h:L \to \R$ be a Morse-Bott function with Morse-Bott maximum at critical submanifold $D_1$ of dimension $k_1$
and minimum at $D_2$ of dimension $k_2$.
If $L \subset T^*L$ is graded and $Graph(dh)$ is equipped with grading induced from that of $L$ and the Hamiltonian isotopy induced by $-h \circ \pi$,
then $Ind(L|_{D_1}, Graph(dh)|_{D_1})=n$ and $Ind(L|_{D_2}, Graph(dh)|_{D_2})=n-k_2$
\end{corr}

\bdf\label{d:branchJ}

For a Lagrangian immersion $(\iota_L,\theta_L)$, we define
$$R_L=R_{\iota_L}:=\{(p,q) \in L \times L| \iota(p)=\iota(q), p \neq q \}$$
and call it the \textbf{set of branch jump }.

\edf

\begin{eg}\label{Example: double cover}
Let $M=T^*\mathbb{RP}^n$ be equipped with the canonical one form and symplectic form with $n>1$.
Fix a choice of $J$ and $\Omega^2$.
Let $\iota_L: L=S^n \to \mathbb{RP}^n:=\ul{L}$ be the double cover of the zero section.
Note that we can equip $\ul{L}$ with a grading $\theta_{\ul L}: \mathbb{RP}^n \to \mathbb{R}$.
The lift of $\theta_{\ul L}$ to $\theta_L:L \to \mathbb{R}$ gives a $\mathbb{Z}/2\mathbb{Z}$ invariant grading of $L$ and hence $Ind(p,q)=n$ for any $(p,q) \in R_L$.
Converesely, if $\theta_L$ is any grading on $L$, then we must have $\theta_L(p)=\theta_L(q)$ for any $(p,q)\in R$
because $\theta_L(q)-\theta_L(p) \in \mathbb{Z}$ for any $(p,q)\in R$ and $\theta_L(q)-\theta_L(p)$ varies continuously with respect to  $(p,q)$.
\end{eg}

\subsection{Local computation for surgery at a point}

The grading of Lagrangian surgery in the local model was considered by Seidel \cite{SeGraded} already, and we include an account for completeness.
Let $H_{\gamma}$ be a Lagrangian handle.
We equip $\mathbb{C}^{n}$ with the standard complex volume form $\Omega=dz_1 \wedge \dots \wedge dz_{n}$.

\begin{lemma}[\cite{SeGraded}]\label{l:indexPoint}
Let $\mathbb{R}^n$ and $i\mathbb{R}^n$ be equipped with gradings $\theta_r$ and $\theta_i$, respectively.
Then, there is a grading $\theta_H$ on $H_{\gamma}$ and a unique integer $m$ such that
$\theta_H$ can be patched with $\theta_r+m$ and $\theta_i$ to give a grading on $\R^n \#_0 i\R^n$.
If $Ind((\mathbb{R}^n|_0,\theta_r),(i\mathbb{R}^n|_0,\theta_i))=1$, we have $m=0$.
\end{lemma}

\begin{proof}
As shown in Example \ref{lemma: flow handle=Lagrangian handle}, $H_{\gamma}=H_{\nu}$ for some flow handle $H_{\nu}$.
Since $H_{\nu}$ is obtained by Hamiltonian flow of $i\R^n$,
$H_{\nu}$ is canonically graded by $\theta_i$ using the Hamiltonian isotopy.
We call this grading $\theta_H$ and continuously extend it on $Cl(H_{\nu})$.
Since $\R^n \cap Cl(H_{\nu})$ has one grading induced from $\theta_r$ and one induced from $\theta_H$,
$\theta_H|_{\R^n \cap Cl(H_{\nu})}-\theta_r|_{\R^n \cap Cl(H_{\nu})}$ is a locally constant integer-valued function.
If $\R^n \cap Cl(H_{\nu})$ is connected, then
there is a unique integer $m$ such that
$\theta_H|_{\R^n \cap Cl(H_{\nu})}=\theta_r|_{\R^n \cap Cl(H_{\nu})}+m$.
If $\R^n \cap Cl(H_{\nu})$ is not connected, then $n=1$ and one can check directly that the same conclusion holds.
As a result, this $m$ is the unique integer such that $\theta_H$ can be patched with $\theta_r+m$ and $\theta_i$ to give a grading on $\R^n \#_0 i\R^n$.
In what follows, we want to show that $m=0$ if $Ind((\mathbb{R}^n|_0,\theta_r),(i\mathbb{R}^n|_0,\theta_i))=1$.

Pick a point $x=(x_1,\dots,x_n) \in S^{n-1}$.
Let $c(s)=\gamma(s)x \in H_{\gamma}$ and denote the image curve as $Im(c)$, where $\gamma$ is an admissible curve (See Definition \ref{d:admissible}).
The Lagrangian plane $\Lambda_{s}$ at $c(s)$ is spanned by $\{\gamma'(s)x\} \cup \{\gamma(s)v_j\}_{j=2}^{n}$,
where $v_j \in T_{x}S^{n-1}$ forms an orthonormal basis.
(See also the proof of Lemma \ref{lemma: Lagrangian hanlde is Lagrangian}).
Therefore, we have
$$Det^2_{\Omega}(\Lambda_{s})=e^{i2(\arg(\gamma'(s))+(n-1)\arg(\gamma(s)))}$$
for all $s$.
There is a unique continuous function $\theta_{c}: Im(c) \to \mathbb{R}$ such that

$\bullet$ $\theta_{c}(c(s))=n-1$ for $s < 0$,

$\bullet$ $\theta_{c}(c(s))=\frac{n}{2}$ for $s > \epsilon$, and

$\bullet$ $e^{2 \pi i \theta_{c}(c(s))}=Det^2_{\Omega}(\Lambda_{s})$ for all $s$

%Since $SO(n)$ acts naturally on $H_{\gamma}$, we can extend $\theta_c$ to a $SO(n)$-invariant grading $\theta_H$ on $H_{\gamma}$.
Therefore, we have $\theta_c-\theta_H|_{Im(c)} \in \Z$ and $\theta_c$ describes the change of Lagrangian planes from $\R^n$ to $i\R^n$ along the handle.
By comparing with the Example \ref{e:IndexRnNRk} (for $k=0$), we can see that if the graded Lagrangians $\mathbb{R}^n$ and
$i\mathbb{R}^n$ inside $\mathbb{C}^n$ intersect at the origin of index $1$, then $m=0$.
This finishes the proof.
\end{proof}

\begin{corr}[\cite{SeGraded}]\label{c:ptIndex}
Let $\iota_i: L_i \to (M,\omega)$ for $i=1,2$ be two graded Lagrangian immersions with grading $\theta_1$ and $\theta_2$, respectively, intersecting transversally at a point $p$.
If $Ind((L_1|_p,\theta_1),(L_2|_p,\theta_2))=1$, then $\iota: L_1\#_pL_2 \to (M,\omega)$ can be equipped with a grading $\theta_{12}$ extending $\theta_1$ and $\theta_2$.
In this case, we call $L_1\#_pL_2$ together with its grading as a surgery from graded $L_1$ to $L_2$.
\end{corr}

\subsection{Local computation for surgery along clean intersection}

This subsection discuss the grading for Lagrangian surgery along clean intersection.
We start with ordinary clean surgery (See subsubsection \ref{s:ordinaryClean}).

\begin{lemma}\label{l:cleanIndex1}
Let $L_1,N^*_D \subset T^*L_1$ be equipped with gradings $\theta_r$ and $\theta_i$, respectively.
For any $\lambda$-admissible function $\nu$ such that $\lambda< r(D)$,
there is a grading $\theta_H$ on $H^D_{\nu}$ and a unique integer $m$
such that $\theta_H$ can be patched with $\theta_i,\theta_r+m$ to become a grading on $L_1\#_D^{\nu}N^*_D$.

Moreover, $m=0$ if and only if $Ind((L_1|_D,\theta_r),(N^*_D|_D,\theta_i))=dim(D)+1$.
\end{lemma}

Immediately from Lemma \ref{l:cleanIndex1}, we have

\begin{corr}
Let $L_1,L_2 \subset (M,\omega)$ be graded Lagrangians cleanly intersecting at $D$.
We can perform a graded surgery $L_1\#_DL_2$ from $L_1$ to $L_2$ along $D$ if and only if $Ind(L_1|_D,L_2|_D)=dim(D)+1$.
\end{corr}

\begin{proof}[Proof of Lemma \ref{l:cleanIndex1}]
The first statement of the lemma follows as in the first paragraph of the proof of Lemma \ref{l:indexPoint}.
Therefore, we just need to prove that $m=0$ if and only if $Ind((L_1|_D,\theta_r),(N^*_D|_D,\theta_i))=dim(D)+1$.
Let $dim(D)=k$.

Pick a Darboux chart such that in local coordinates $N^*_D$ is represented by points of the form
$(q,p)=(q_b,0,0,p_f) =(q_1,\dots,q_k,0,\dots,0,p_{k+1},\dots,p_n)$, where the first $0$ in
$(q_b,0,0,p_f)$ are the last $n-k$ $q_i$ coordinates and the second $0$ are the first $k$ $p_i$ coordinates.
We also require $(q_1,\dots,q_k,tq_{k+1},\dots,tq_n)$ are normalized geodesics on $L_1$ as $t$ varies, for any $q_1,\dots,q_n$ such that $\sum_{j=k+1}^n q_j^2=1$.
As a result, the handle $H^{D}_{\nu}$ in local coordinates is given by (here, we suppose that the surgery is supported in a sufficiently small region relative to the Darboux chart)
$$\{ \phi^{\vp}_{\nu(\| p_f \|)}(q_b,0,0,p_f)=((q_b,\nu(\| p_f \|)\frac{p_f }{\|p_f \|},0,p_f) | q_b \in \R^k, p_f \in \R^{n-k} \}$$

We consider the standard complex volume form $\Omega=dz_1 \wedge \dots \wedge dz_{n}$ in the chart.
Let $e_{\pi_2} \in S^{n-k-1} \subset \R^{n-k}$ be a vector in the unit sphere of last $n-k$ $p_i$ coordinates.
Let
$$c(r)=(0,\nu(\| re_{\pi_2} \|)\frac{re_{\pi_2} }{\|re_{\pi_2} \|},0,re_{\pi_2})=(0,\nu( r )e_{\pi_2},0,re_{\pi_2})$$
be a smooth curve on $H^{D}_{\nu}$ for $r \in (0,\epsilon]$.
We define $c(0)=\lim_{r \to 0^+} c(r)$.

We want to understand how the Lagrangian planes change from $L_1$ to $N^*_D$ along the handle and
it suffices to look at how the Lagrangian planes change along $c(r)$.
The Lagrangian plane $\Lambda_{r}$ at $c(r)$ is spanned by
$$\{(e_j,0,0,0)\}_{j=1}^k \cup \{(0,\nu'(r)e_{\pi_2},0,e_{\pi_2})\}
\cup \{(0,\nu(r)\frac{e^{\perp}_j }{r},0,e^{\perp}_j)\}_{j=2}^{n-k}$$
where $e_j \in \R^k$ are coordinate vectors and $e^{\perp}_j$ form an orthonormal basis for orthogonal complement of $e_{\pi_2}$ in $\R^{n-k}$.

Then we have
$$Det^2_{\Omega}(\Lambda_{r})=e^{i2(\arg(\nu'(r)-\sqrt{-1})+(n-k-1)\arg(\frac{\nu(r)}{r}-\sqrt{-1}))}$$
for all $r$.
Here, the convention we use is still, $z_i=q_i-\sqrt{-1}p_i$.
Observe that $\nu'(\epsilon)=\frac{\nu(\epsilon)}{\epsilon}=0$.
When $r$ goes to $0$, $\nu'(r)$ decreases monotonically to $-\infty$.
Similarly, $\frac{\nu(r)}{r}$ increases monotonically to infinity because $r$ goes to zero and $\nu$ is bounded and positive.

In particular, $\arg(\nu'(r)-\sqrt{-1})$ increases from $\pi$ to $\frac{3\pi}{2}$ as $r$ increases and
$\arg(\frac{\nu(r)}{r}-\sqrt{-1}))$ decreases from $2\pi$ to $\frac{3\pi}{2}$ as $r$ increases.
Therefore, there is a unique continuous function $\theta_{c}: Im(c) \to \mathbb{R}$ such that

$\bullet$ $\theta_{c}(c(r))=n-k-1$ for $r= 0$,

$\bullet$ $\theta_{c}(c(r))=\frac{n-k}{2}$ for $r = \epsilon$, and

$\bullet$ $e^{2 \pi i \theta_{c}(c(r))}=Det^2_{\Omega}(\Lambda_{r})$ for all $r \in [0,\epsilon]$.

By Example \ref{e:IndexRnNRk}, we have
$Ind((\R^n|_{\R^k},n-k-1),(N^*(\R^k)|_{\R^k},\frac{n-k}{2}))=k+1$.
Hence, $m=0$ if and only if $Ind((L_1|_D,\theta_r),(N^*_D|_D,\theta_i))=k+1$.
% the $\theta_H$ can be patched nicely with $\theta_r$
%and $\theta_i$ so the result follows.
\end{proof}

For the $E_2$-flow surgery, we use the setting in subsection \ref{s:E2surgery} and we have

\begin{lemma}\label{l:cleanIndex2}
Suppose $D\subset L=K_1 \times K_2$ is a smooth submanifold of dimension $k$ which is transversal to $\{p\} \times K_2$
for all $p \in K_1$.
Let $L,N^*_D \subset T^*L$ be equipped with gradings $\theta_r$ and $\theta_i$, respectively.
For any $\lambda$-admissible function $\nu$ such that $\lambda< r^{E_2}(D)$,
there is a grading $\theta_H$ on $H^{D,E_2}_{\nu}$ and a unique integer $m$ such that $\theta_H$ can be patched with $\theta_r+m$ and $\theta_i$
to become a grading on $L\#_{D,E_2}^{\nu}N^*_D$.

Moreover, we have $m=0$ if and only if $Ind((L|_D,\theta_r),(N^*_D|_D,\theta_i))=dim(D)+1$.
\end{lemma}

\begin{corr}
Let $L_1=K_1 \times K_2,L_2 \subset (M,\omega)$ be graded Lagrangians cleanly intersecting at $D$.
Suppose $D$ is transversal to $\{p\} \times K_2$ for all $p \in K_1$.
Then we can perform a graded $E_2$-flow surgery $L_1\#_{D,E_2}L_2$ from $L_1$ to $L_2$ along $D$ if and only if $Ind(L_1|_D,L_2|_D)=dim(D)+1$.
\end{corr}

\begin{proof}[Proof of Lemma \ref{l:cleanIndex2}]

As explained before (cf. Lemma \ref{l:indexPoint} and Lemma \ref{l:cleanIndex1}), we just need to show that $m=0$ if and only if $Ind((L|_D,\theta_r),(N^*_D|_D,\theta_i))=dim(D)+1$. Again denote $k=dim(D)$.

Pick a chart compatible with the product structure on $L$ and
define $q_b=(q_1,\dots,q_{k}) \in L_1$ and $q_f=(q_{k+1},\dots,q_n) \in L_2$.
We also want that $(q_b,tq_f)$ is a geodesic with velocity one as $t$ varies, for any $q_b,q_f$ such that $|q_f|=1$.
We can also assume the origin belongs to $D$
and denote a basis of the tangent space of $D$ at origin $T_0D$ as
$\{w^1,\dots,w^k\}$ and $w^j=w^j_b+w^j_f$, where $w^j_b$ and $w^j_f$ are the $q_b$ and $q_f$ components of $w^j$, respectively.
Since $D$ is transversal to the second factor,
we can assume $w^j_b$ are the unit coordinate vectors in the $q_b$-plane for $1 \le j \le k$.
Moreover, there is a function $q^D_f(q_b)$ of $q_b$ near origin such that $(q_b,q^D_f(q_b)) \in D$.

 This chart gives a corresponding Darboux chart on $T^*L$ and
 we define $p^D_b$ as a function of $q_b,p_f$ near origin such that $(q_b,q^D_f(q_b),p^D_b(q_b,p_f),p_f) \in N^*_D$.
 Note that $p^D_b(\cdot,\cdot)$ is linear on the second factor.
Near the origin (close enough to origin such that $q^D_f(q_b)$ is well-defined), the handle $H^{D,E_2}_{\nu}$ in local coordinates is given by
$$\{ \phi^{\vp_{\pi}}_{\nu(\| p_b \|)}(q_b,q^D_f(q_b),p^D_b(q_b,p_f),p_f)=
(q_b,q^D_f(q_b)+\nu(\| p_f \|)\frac{p_f }{\|p_f \|},p^D_b(q_b,p_f),p_f) | q_b \in \R^k, p_f \in \R^{n-k} \}$$

We consider the standard complex volume form $\Omega=dz_1 \wedge \dots \wedge dz_{n}$ in the chart.
Let $e_{\pi_2} \in S^{n-k-1} \subset \R^{n-k}$ be a vector in the unit sphere in the $p_f$ coordinates.
Let
\begin{eqnarray*}
c(r)&=& \phi^{\vp_{\pi}}_{\nu(\| re_{\pi_2} \|)}(0,0,p^D_b(0,re_{\pi_2}),re_{\pi_2})\\
&=&(0,\nu(\| re_{\pi_2} \|)\frac{re_{\pi_2} }{\|re_{\pi_2} \|},p^D_b(0,re_{\pi_2}),re_{\pi_2})\\
&=&(0,\nu( r )e_{\pi_2},p^D_b(0,re_{\pi_2}),re_{\pi_2})
\end{eqnarray*}
 be a smooth curve in $H^{D,E_2}_{\nu}$ for $r \in (0,\epsilon]$.
We define $c(0)=\lim_{r \to 0^+} c(r)$.

The Lagrangian plane $\Lambda_{r}$ of $H^{D,E_2}_{\nu}$ at $c(r)$ is spanned by
$$\{(w^j_b,w^j_f,\kappa(r,w^j),0)\}_{j=1}^k \cup \{(0,\nu'(r)e_{\pi_2},p^D_b(0,e_{\pi_2}),e_{\pi_2})\}
\cup \{(0,\frac{\nu(r)}{r}e^{\perp}_j,p^D_b(0,e^{\perp}_j),e^{\perp}_j)\}_{j=2}^{n-k}$$
where  $\kappa(r,w^j)=\partial_{q_j} p^D_b(0,re_{\pi_2})=r (\partial_{q_j} p^D_b(0,e_{\pi_2}))$ is linear in $r$ and
$e^{\perp}_j$ form an orthonormal basis for orthogonal complement of $e_{\pi_2}$ in $\R^{n-k}$.
We note that $(0,\nu'(r)e_{\pi_2},p^D_b(0,e_{\pi_2}),e_{\pi_2})=c'(r)$ and the computation uses the fact that
$p^D_b(\cdot,\cdot)$ is linear on the second factor.

Let $\kappa_j(r,w^j)$ be the coefficient of $w^j_b$-component of $\kappa(r,w^j)$ (Here, we identify the $q_b$-plane and the $p_b$-plane).
Notice that
$$Det^2_{\Omega}(\Lambda_{r})=e^{i2(\sum\limits_{j=1}^k\arg(1-\kappa_j(r,w^j)\sqrt{-1}) +\arg(\nu'(r)-\sqrt{-1})+(n-k-1)\arg(\frac{\nu(r)}{r}-\sqrt{-1}))}$$
for all $r$ (Here, we use the fact that $w^j_b$ are unit coordinates vectors and we use the convention $z_i=q_i-\sqrt{-1}p_i$).
Let $K(r)=\sum\limits_{j=1}^k\arg(1-\kappa_j(r,w^j)\sqrt{-1})$.

Similar to Lemma \ref{l:cleanIndex1}, $\arg(\nu'(r)-\sqrt{-1})$ increases from $\pi$ to $\frac{3\pi}{2}$ as $r$ increases and
$\arg(\frac{\nu(r)}{r}-\sqrt{-1}))$ decreases from $2\pi$ to $\frac{3\pi}{2}$ as $r$ increases.
Therefore, there is a unique continuous function $\theta_{c}: Im(c) \to \mathbb{R}$ such that
\begin{itemize}
\item $\theta_{c}(c(r))=n-k-1+\frac{K(0)}{\pi}=n-k-1$ for $r= 0$,
\item $\theta_{c}(c(r))=\frac{n-k}{2}+\frac{K(\epsilon)}{\pi}$ for $r = \epsilon$, and
\item $e^{2 \pi i \theta_{c}(c(r))}=Det^2_{\Omega}(\Lambda_{r})$ for all $r \in [0,\epsilon]$.
\end{itemize}

%In the chart, $\Lambda_{0}$ can be identified as the Lagrangian tangent plane of $L$ at origin.
On the other hand, we can lift a path of Lagrangian plane $\Lambda^N_r$ of $N^*_D$ over the path $c_2(r)=(0,0,p^D_b(0,re_{\pi_2}),re_{\pi_2})$
connecting the origin and $c(\epsilon)$.
The Lagrangian plane $\Lambda^N_r$ is spanned by
$$\{(w^j_b,w^j_f,\kappa(r,w^j),0)\}_{j=1}^k \cup \{(0,0,p^D_b(0,e_{\pi_2}),e_{\pi_2})\}
\cup \{(0,0,p^D_b(0,e^{\perp}_j),e^{\perp}_j)\}_{j=2}^{n-k}$$
Therefore, the grading of $N^*_D$ at origin is the grading of $N^*_D$ at $c(\epsilon)$ subtracted by $\frac{K(\epsilon)}{\pi}$.
If we extend $\theta_c$ continuously over $Im(c_2)$ (note: $Im(c) \cap Im(c_2)=\{c(\epsilon)\}$), then $\theta_c(c_2(0))=\frac{n-k}{2}$.

By an analogous calculation as in Example \ref{e:IndexRnNRk}, we have
$$Ind((L_1|_{D},n-k-1),
(N^*_D|_{D},\frac{n-k}{2}))=k+1$$
and by comparing it with $\theta_c$, the result follows.

\end{proof}

The following is a family version whose proof is similar.

\begin{corr}\label{c:familyE2Index}
Let $\mathcal{L}_0,\mathcal{L}_1 \subset (M^{2n},\w)$ as in Lemma \ref{l:familyE2flow} and let the dimension of $\mathcal{D}$ be $k$.
Assume $\mathcal{L}_0,\mathcal{L}_1$ are graded with grading $\theta_r$ and $\theta_i$.
Then $Ind((\mathcal{L}_1|_{\mathcal{D}},\theta_r),(N^*\mathcal{D}|_{\mathcal{D}},\theta_i))=k+1$
if and only if $\mathcal{L}_0\#_{\mathcal{D},E_2}^{\nu} \mathcal{L}_1$
has a grading such that the grading restricted to $\mathcal{L}_0,\mathcal{L}_1$ coincide with $\theta_r$ and $\theta_i$, respectively.
\end{corr}

\subsection{Diagonal in product}
We recall from \cite{WWQuiltedFloer} how to associate the canonical grading to the diagonal in $M \times M^-$.

For a standard symplectic vector space $(\R^{2n},\w_{std})$ and its $N$-fold Maslov cover $Lag^N(\R^{2n},\Lambda_0)$ based at a graded Lagrangian plane $\Lambda_0$,
we can associate a $N$-fold Maslov cover $Lag^N(\R^{2n,-} \times \R^{2n},\Lambda_0^- \times \Lambda_0)$.
In particular, $\Lambda_0^- \times \Lambda_0$ is canonically graded.
For any Lagrangian plane $\Lambda \subset \R^{2n}$ and a path $\gamma$ from $\Lambda$ to $\Lambda_0$,
the induced path $\gamma^- \times \gamma$ from $\Lambda^- \times \Lambda$ to $\Lambda_0^- \times \Lambda_0$ gives an identification between
$Lag^N(\R^{2n,-} \times \R^{2n},\Lambda_0^- \times \Lambda_0)$ and
$Lag^N(\R^{2n,-} \times \R^{2n},\Lambda^- \times \Lambda)$, independent from the choice of $\gamma$.
This gives a canonical grading on $\Lambda^- \times \Lambda$.

To give a canonical grading to the diagonal $\Delta \subset \R^{2n,-} \times \R^{2n}$, it suffices to give once and for all an identification
between $Lag^N(\R^{2n,-} \times \R^{2n},\Lambda^- \times \Lambda)$ and $Lag^N(\R^{2n,-} \times \R^{2n},\Delta)$.
This is given by concatenation of two paths
$$(e^{Jt}\Lambda^- \times \Lambda)_{t \in [0, \frac{\pi}{2}]},\quad (\{(tx+Jy,x+tJy)| x,y\in \Lambda \})_{t \in [0,1]} $$
where $J$ is an $\w_{std}$-compatible complex structure on $\R^{2n}$.
This canonical grading induces a canonical grading on $\Delta_M \subset M^- \times M$ for any symplectic manifold $M$.

In the following lemma, we consider our symplectic manifold being $M=\C^{n,-}$ and compute
the index between a product Lagrangian with the diagonal $\Delta_{M}$.

\begin{lemma}[c.f. Section $3$ of \cite{WWQuiltedFloer}]\label{l: grading of diagonal}
For any graded Lagrangian subspace $\Lambda \subset \C^{n,-}$, we have
$$Ind(\Lambda^- \times \Lambda|_{\Delta_{\Lambda}},\Delta_{\C^{n,-}}|_{\Delta_{\Lambda}})=n$$
where $\Lambda^- \times \Lambda$ and $\Delta_{\C^{n,-}}$ are equipped with their canonical gradings in $\C^{n} \times \C^{n,-}$.
\end{lemma}

\begin{proof}
It suffices to consider $\Lambda=\R^n \subset \C^{n,-}$ and $J=-J_{std}=-\sqrt{-1}$.
Let $z_i=x_i+\sqrt{-1}y_i$ be the coordinates of $\C^{n}$ and $w_i=u_i+\sqrt{-1}v_i$ be the coordinates of $\C^{n,-}$.
We consider the standard complex volume form $\Omega=dz_1 \wedge \dots \wedge dz_n \wedge d\bar{w_1} \wedge \dots \wedge d\bar{w_n}$ on $\C^{n} \times \C^{n,-}$ and
equip $\Lambda^- \times \Lambda$ with grading $0$.
We have $Det^2(e^{Jt}\Lambda^- \times \Lambda)=e^{-i2nt}$, which induces a grading of $-\frac{n}{2}$ on $e^{J\frac{\pi}{2}}\Lambda^- \times \Lambda$.
We also have $Det^2(\{(tx+Jy,x+tJy)| x,y\in \Lambda \})=e^{-in\pi}$ for all $t$ so the canonical grading on $\Delta$ is $-\frac{n}{2}$.

To calculate $Angle(\Lambda^- \times \Lambda,\Delta)$, we observe that $(\Lambda^- \times \Lambda) \cap \Delta=Span\{(\partial_{x_i}+\partial_{u_i})\}_{i=1}^n$.
We can use $\Lambda_t=(\Lambda^- \times \Lambda) \cap \Delta+Span\{(t(\partial_{y_i}+\partial_{v_i})+(1-t)(-\partial_{x_i}+\partial_{u_i}))\}$ from $\Lambda^- \times \Lambda $ to
$\Delta$ for the calculation of $Angle(\Lambda^- \times \Lambda,\Delta)$.
As a result, we have $2Angle(\Lambda^- \times \Lambda,\Delta)=\frac{n}{2}$ and hence
$$Ind(\Lambda^- \times \Lambda|_{\Delta_{\Lambda}},\Delta_{\C^{n,-}}|_{\Delta_{\Lambda}})=2n+(-\frac{n}{2})-0-\frac{n}{2}=n$$

\end{proof}

\begin{corr}\label{c:GradedSurgProduct}
 Let $L$ be a Lagrangian in $M$.
 With the canonical gradings of $L\times L \subset M \times M^-$ and $\Delta \subset M \times M^-$,
 one can perform graded clean surgery to obtain $(L\times L)[1] \#_{\Delta_L,E_2} \Delta $.

\end{corr}

\begin{proof}
 This is a direct consequence of Lemma \ref{l:cleanIndex2} and Lemma \ref{l: grading of diagonal}.
\end{proof}

\begin{corr}\label{c: graded idenity for graph of Sn}
 There is a graded clean surgery identity $(S^n\times S^n)[1] \#_{\Delta_{S^n},E_2} \Delta = Graph(\tau_{S^n}^{-1})$.
\end{corr}

\begin{proof}
A direct consequence of Corollary \ref{c:admissibleToDehnProduct} and Corollary \ref{c:GradedSurgProduct}.
\end{proof}

\begin{lemma}\label{l: graded identity for graph of CPm/2}
 There is a graded clean surgery identity
 $$\mathbb{CP}^{\frac{m}{2}}\times \mathbb{CP}^{\frac{m}{2}} \#_{D^{op},E_2}
 ((\mathbb{CP}^{\frac{m}{2}}\times \mathbb{CP}^{\frac{m}{2}})[1] \#_{\Delta_{\mathbb{CP}^{\frac{m}{2}}},E_2} \Delta) = Graph(\tau_{\mathbb{CP}^{\frac{m}{2}}}^{-1})$$.
\end{lemma}

\begin{proof}
 By Corollary \ref{c:GradedSurgProduct}, we can obtain a graded Lagrangian
 $L=(\mathbb{CP}^{\frac{m}{2}}\times \mathbb{CP}^{\frac{m}{2}})[1] \#_{\Delta_{\mathbb{CP}^{\frac{m}{2}}},E_2} \Delta$.
 As explained in the proof of Lemma \ref{l:adjust} and Lemma \ref{l:CPnCob},
 $L$ is Hamiltonian isotopic to a Lagrangian $Q$ cleanly intersecting with
 $\mathbb{CP}^{\frac{m}{2}}\times \mathbb{CP}^{\frac{m}{2}}$ along $D^{op}$ such that $Q$ coincide with the graph of
 a Morse-Bott function with maximum at $D^{op}$ near $D^{op}$.
 Therefore, we have $Ind(\mathbb{CP}^{\frac{m}{2}}\times \mathbb{CP}^{\frac{m}{2}}|_{D^{op}}, Q|_{D^{op}})=2m-1$.
 Here the first term $2m$ follows by Corollary \ref{c:MorseBott} and the second term $-1$
 comes from the grading shift of the first factor of $L$.
 %and the last term comes from the grading shift of $\mathbb{CP}^{\frac{m}{2}}\times \mathbb{CP}^{\frac{m}{2}}[-1]$.
 Since $D^{op}$ is of dimension $2m-2$, we get the result by applying Lemma \ref{l:cleanIndex2}.
 \end{proof}

The cases for $\mathbb{RP}^n$ and $\mathbb{HP}^n$ can be computed analogously.

\begin{lemma}\label{l:RpnHpnSurgIdentity}
 There are also graded clean surgery identities
 $$\mathbb{RP}^{n}\times \mathbb{RP}^{n}[1] \#_{D^{op},E_2}
 ((\mathbb{RP}^{n}\times \mathbb{RP}^{n})[1] \#_{\Delta_{\mathbb{RP}^{n}},E_2} \Delta) = Graph(\tau_{\mathbb{RP}^{n}}^{-1})$$
 and
 $$\mathbb{HP}^{n}\times \mathbb{HP}^{n}[-2] \#_{D^{op},E_2}
 ((\mathbb{HP}^{n}\times \mathbb{HP}^{n})[1] \#_{\Delta_{\mathbb{HP}^{n}},E_2} \Delta) = Graph(\tau_{\mathbb{HP}^{n}}^{-1})$$
 where $D^{op}$ are defined similar to Lemma \ref{l: graded identity for graph of CPm/2}.
\end{lemma}

For family Dehn twist, we have (See Corollary \ref{l:fiberedTwist})

\begin{lemma}\label{l:familyGrade}
There are graded clean surgery identities
$$\wt C_S[1] \#_{\eD,E_2} \Delta = Graph(\tau_{C_S}^{-1})$$
$$\wt C_R[1]\#_{\eD^{op},E_2} \wt C_R[1]\#_{\eD,E_2} \Delta=Graph(\tau_{C_R}^{-1})$$
$$\wt C_C\#_{\eD^{op},E_2} \wt C_C[1]\#_{\eD,E_2} \Delta=Graph(\tau_{C_C}^{-1})$$
$$\wt C_H[-2]\#_{\eD^{op},E_2} \wt C_H[1]\#_{\eD,E_2} \Delta=Graph(\tau_{C_H}^{-1})$$
where $C_S$ (resp. $C_R,C_C,C_H$) is a spherically (resp. real projectively, complex projectively, quaternionic projectively) coisotropic submanifold.

\end{lemma}

 \subsection{Primitive function on an exact Lagrangian under surgeries}

\begin{eg}\label{e:primitiveGraph}
Let $h:L \to \R$ be a function and $Graph(dh)$ be the graph of $dh$.
With respect to the canonical one form, $Graph(dh)$ is an exact Lagrangian with primitive $h$.
\end{eg}

The following lemma
shows that for exact Lagrangian obtained from ordinary clean surgery, the value of primitive function {\bf decreases along handle}.

\begin{lemma}\label{l:localEnergyClean}
Let $L_1,D,N^*_D,H^{D}_{\nu},c(r)$ be as in (the proof) of Lemma \ref{l:cleanIndex1}.
For the canonical one form $\alpha \in \Omega^1(T^*L_1)$, we have $\int_{c} \alpha < 0$.

In particular, if $L_1\#_{D}^{\nu}N^*_D$ is an exact Lagrangian with primitive $f$ with respect to $\alpha$ (cf. Lemma \ref{l:exactSimCob}),
then $f(c(0))>f(c(\epsilon))$.
\end{lemma}

\begin{proof}
Using the notation from the proof of Lemma \ref{l:cleanIndex1},
we have $c(r)=(0,\nu(r)e_{\pi_2},0,re_{\pi_2})$
and $c'(r)=(0,\nu'(r)e_{\pi_2},0,e_{\pi_2})$.
As a result,
$$\alpha|_{c(r)}(c'(r))= \la re_{\pi_2}, \nu'(r)e_{\pi_2} \ra=r\nu'(r) <0$$
for $0 < r < \epsilon$ since
$\nu'<0$. The result follows.
\end{proof}

 The case for $E_2$-flow clean surgery is similar.

\begin{lemma}\label{l:localEnergyClean2}
Let $L,D,N^*_D,H^{D,E_2}_{\nu},c(r)$ be as in (the proof) of Lemma \ref{l:cleanIndex2}.
For the canonical one form $\alpha \in \Omega^1(T^*L)$, we have $\int_{c} \alpha < 0$.

In particular, if $L\#_{D,E_2}^{\nu}N^*_D$ is an exact Lagrangian with primitive $f$ with respect to $\alpha$ (cf. Lemma \ref{l:exactSimCob}),
then $f(c(0))>f(c(\epsilon_c))$.
\end{lemma}

If $\iota_L$ is an exact Lagrangian immersion with primitive $f_L$ (ie. $f_L:L \to \mathbb{R}$ is such that $df_L=\iota_L^*\alpha$),
then we define the {\bf energy of branch jump} $E: R_L \to \mathbb{R}$
$$E(p,q)=-f_L(p)+f_L(q)$$
which is independent of choice of primitive $f_L$.

\begin{eg}
Using the setup in Example \ref{Example: double cover}, then $E:\ul L \to \mathbb{R}$ is identically
zero because $E$ is a locally constant function on $R$ and $E(p,q)=-E(q,p)$.  This applies to any double covers.
\end{eg}

\section{Review of Lagrangian Floer theory, Lagrangian cobordisms and quilted Floer theory}\label{s:review}

We first fix conventions for Lagrangian Floer theory in the rest of the paper, which follows that of \cite{Seidelbook}.  Note that this is different from the homology convention of \cite{BC13}.

Let $L_0, L_1\subset (M,\w)$ be a pair of transversally intersecting Lagrangians.  For a generic one-parameter family of $\w$-compatible almost complex structure $\mathbf{J}=J_t$, let
\beq\label{e:Floer}\begin{aligned}\cM(p_-,p_+)=\{&u: \R \times [0,1] \to M:\\
&u_s(s,t)+J_t(u(s,t))u_t(s,t)=0,\\
& u(s,0) \in L_0\text{ and }u(s,1) \in L_1\\
&\lim_{s\rightarrow +\infty}u(s,t)=p_+,\\
&\lim_{s\rightarrow -\infty}u(s,t)=p_-\}.\end{aligned}\eeq

Then the Floer cochain complex $CF^*(L_0,L_1)$ is generated by $L_0 \cap L_1$ and equipped with a differential by counting rigid elements from $\cM(p_-,p_+)$, i.e.

$$dp_+=\sum_{p_-\in L_0\cap L_1} \#\cM^0(p_-, p_+)p_-$$
The higher operations are defined analogously by counting holomorphic polygons as in \cite{Seidelbook}.  We refer thereof for the definition of the Fukaya category and will not repeat it here.

\bdf\label{d:cob} Let $L_i, L_j'\subset (M,\w)$, $1\le i\le k$, $1\le j\le k'$ be a collection of Lagrangian submanifolds.  A Lagrangian cobordism $V$ from $(L_1,\dots,L_k)$ to $(L_{k'}',\dots,L_1')$ is an embedded Lagrangian submanifold in $M\times \C$ so that the following condition hold.

$\bullet$ There is a compact set $K \subset \mathbb{C}$ such that $V-(M \times K)=(\sqcup_{i=1}^k L_i \times \gamma_i) \sqcup (\sqcup_{j=1}^{k'} L_j' \times \gamma_j')$ called the {\bf ends}, where $\gamma_i=(-\infty,x_i)\times\{a_i\}$ and $\gamma_j'=(x_j',\infty) \times \{b_j'\}$ for some $x_i,a_i,x_j',b_j'$ such that $a_1<\dots<a_k$ and $b_1' < \dots < b_{k'}'$.

\edf

For the grading we always choose the quadratic complex volume form on $M \times \mathbb{C}$
to be the quadratic complex volume form on $M$ times the standard one on $\mathbb{C}$.
When $V$ is graded, the restriction induces a grading on each end.
On an end, say $L_i \times \gamma_i$, we denote the induced grading as $\theta_{ii}$.
Since an end is a product Lagrangian, we can associate a grading $\theta_i$ to $L_i$ by requiring
$\theta_i(p)=\theta_{ii}(p \times z)$ for all $p \in L_i$ and $z \in \gamma_i$.  The same rule applies to $L_j' \times \gamma_j'$.
We use this grading convention between a cobordism and its fiber Lagrangians over its ends throughout.

\begin{figure}[h]
\includegraphics[scale=0.6]{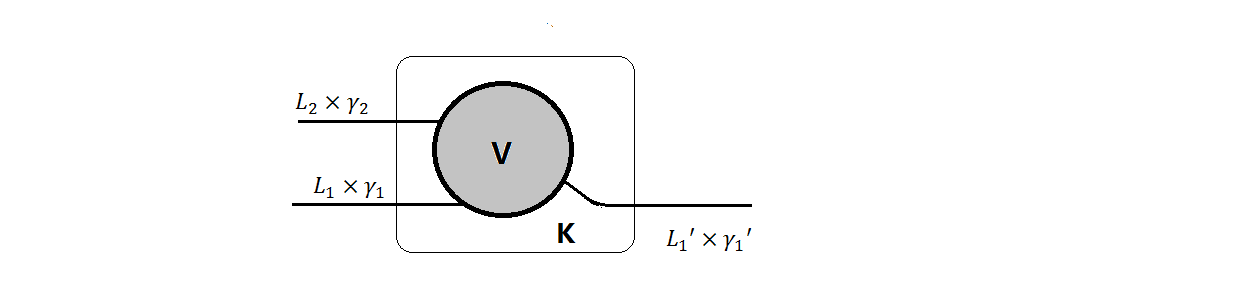}

\caption{Projection of a Lagrangian cobordism.}
\label{fig: sample cobordism}
\end{figure}

The main result we will utilize from Biran-Cornea's Lagrangian cobordism formalism reads:

\begin{thm}[\cite{BC2}]\label{thm: Cone from Lagrangian cobordism}
If there exists a graded monotone (or exact) Lagrangian cobordism from monotone (or exact) Lagrangians
$(L_1[k-1],L_2[k-2],\dots,L_k)$ to $(L'_{k'}[k'-1],L'_{k'-1}[k'-2],\dots,L_1')$, then there is an isomorphism between iterated cones in $\cD^{\pi}\cF uk(M)$,
$$Cone(L_1\rightarrow L_2\rightarrow\cdots\rightarrow L_k)\cong Cone(L'_{k'}\rightarrow L'_{k'-1}\rightarrow\cdots\rightarrow L'_{1})$$
Here $Cone(L_1\rightarrow L_2\rightarrow\cdots\rightarrow L_k)= Cone(\cdots Cone(Cone(L_1\rightarrow L_2)\rightarrow L_3)\rightarrow\cdots\rightarrow L_k)$.
\end{thm}

Note that, $CF^*(K, Cone(L_1\rightarrow L_2\rightarrow\cdots\rightarrow L_k))=CF^*(K,L_1[k-1]) \oplus \dots \oplus CF^*(K,L_k)$ as a graded vector space for
any graded Lagrangian $K$ transversally intersecting $L_i$.
It explains the seemingly weird grading shift of the Lagrangians $L_i,L_j'$ for the cobordism.

We dedicate the rest of this section to quilted Floer theory developed in \cite{WWQuiltedFloer}\cite{WWFunctor}\cite{WWComposition}\cite{MWWFunctor}.

\begin{defn}
 Given a sequence of symplectic manifolds $M_0,\dots,M_{r+1}$, a generalized Lagrangian correspondence $\underline{L}=(L_{01},\dots,L_{r(r+1)})$ is a sequence such that $L_{i(i+1)} \subset M_i^- \times M_{i+1}$ are embedded Lagrangian submanifolds for all $i$.
A cyclic generalized Lagrangian correspondence is one such that $M_0=M_{r+1}$.
\end{defn}

For a Lagrangian correspondence $L_{01} \subset M_0^- \times M_1$, $L_{01}^t \subset M_1^- \times M_0$ is defined to be $L_{01}^t=\{(x,y)| (y,x) \in L_{01}\}$.
Given two Lagrangian correspondences $L_{01} \subset M_0^- \times M_1$ and $L_{12} \subset M_1^- \times M_2$, the geometric composition is defined as
\beq\label{e:composition}L_{01} \circ L_{12}=\{(x,z)| \exists y \text{ such that } (x,y) \in L_{01} \text{ and } (y,z) \in L_{12}\}\eeq

For the composition to work nicely, we require that:
\begin{itemize}
\item the projection $\pi_{02}: L_{01} \times_{M_1} L_{12} \to L_{01} \circ L_{12}$ is an embedding, where $L_{01} \times_{M_1} L_{12}$ is the fiber product and $\pi_{02}$ is the projection forgetting $M_1$ factor
\item $L_{01}\times L_{12}$ intersects $M_0^-\times \Delta\times M_2$ transversally in $M_0^-\times M_1\times M_1^-\times M_2$.
\end{itemize}

In this case, the composition $L_{01} \circ L_{12}$ is called {\it embedded}.
 One is referred to Section \ref{s:familyTwist} for a non-trivial example of Lagrangian correspondence and comopsition coming from coisotropic embeddings.

For a cyclic generalized Lagrangian correspondence $\underline{L}$, the \textbf{quilted Floer cohomology} is defined to be
$$HF^*(\underline{L})=HF^*(L_{01} \times L_{23} \dots L_{(r-1)r},L_{12} \times L_{34} \dots L_{r(r+1)})$$
in $M_0^- \times M_1 \times \dots \times M_{r-1} \times M_r$ if $r$ is odd, and
$$HF^*(\underline{L})=HF^*(L_{01} \times L_{23} \dots L_{r(r+1)},L_{12} \times L_{34} \dots L_{(r-1)r} \times \Delta_{M_0})$$
in $M_0^- \times M_1 \times \dots \times  M_r^- \times M_{r+1}$ if $r$ is even.

It is worth pointing out that for the quilted Floer cohomology to be well-defined, $\underline L$ needs to satisfy a stronger monotonicity condition (\cite[Definitnion 4.1.2(b)]{WWQuiltedFloer}).  A sufficient condition for this stronger monotonicity to hold for $\underline L=(L_0, L_{01}, L_1)$ is when $\pi_1(L_{01})=1$ (\cite[Lemma 4.1.3]{WWQuiltedFloer}).
We refer readers to \cite{WWQuiltedFloer} for further details on monotonicity, as well as orientation, grading, exactness, and so forth for a generalized Lagrangian correspondence.  The following theorems summarize main properties that will concern us.

\begin{thm}[Theorem 5.2.6 of \cite{WWQuiltedFloer}]\label{thm: quilted Kunneth formula}
For a monotone (or exact) cyclic generalized Lagrangian correspondence $\underline{L}$ such that
\begin{itemize}
\item $M_i$ are monotone with the same monotonicity constant,
\item the minimal Maslov numbers of $L_{i(i+1)}$ are at least three,
\item $M_0=M_{r+1}$ is a point,
\item $L_{i(i+1)}=L_i \times L_{i+1}$ for Lagrangians $L_i \subset M_i$ and $L_{i+1} \subset M_{i+1}$ for some $1 \le i < r$
\end{itemize}
then there is a canonical isomorphism
$$HF^*(\underline{L})=HF^*(L_{01},L_{12},\dots,L_{(i-1)i},L_i) \otimes HF^*(L_{i+1},L_{(i+1)(i+2)},\dots,L_{r(r+1)})$$
with coefficients in a field.
\end{thm}

\begin{thm}[Theorem 5.4.1 of \cite{WWQuiltedFloer} and Theorem 1, 2 of \cite{LL13}]\label{thm: quilted geometric composition equality}
For a cyclic generalized Lagrangian correspondence $\underline{L}$ such that
\begin{itemize}
\item $M_i$ are monotone with the same monotonicity constant,
\item the minimal Maslov numbers of $L_{i(i+1)}$ are at least three,
\item $\underline{L}$ is monotone, relatively spin and graded in the sense of Section 4.3 of \cite{WWQuiltedFloer}, and
\item $L_{(i-1)i} \circ L_{i(i+1)}$ is embedded in the sense above
\end{itemize}
then there is a canonical isomorphism
$$HF^*(\underline{L})=HF^*(L_{01},L_{12},\dots,L_{(i-1)i} \circ L_{i(i+1)},\dots,L_{r(r+1)})$$
where the orientation and grading on the right are induced by those on $\underline{L}$.
\end{thm}

\begin{rmk}
In \cite{LL13}, Theorem \ref{thm: quilted geometric composition equality}  was extended to greater generality than stated here, which should be useful for extending our results to negatively monotone cases.
\end{rmk}

For a symplectomorphism $\phi \in Symp(M)$, the fixed point Floer cohomology can be defined as
$$HF^*(\phi)=HF^*(\Delta,Graph(\phi))=HF^*(Graph(\phi^{-1}),\Delta)$$ where the Lagrangian Floer cohomologies take place in $M \times M^-$.

\begin{rmk}
We follow the convention in \cite{WWQuiltedFloer}, where $HF^*(\phi)=HF^*(Graph(\phi),\Delta)$ in $M^- \times M$.
Therefore, we have $HF^*(\phi)=HF^*(\Delta,Graph(\phi))$ in $M \times M^-$.
\end{rmk}

An $A_\infty$ version of quilted Floer theory was also developed in \cite{MWWFunctor}.  This defines an $A_\infty$ structure for all generalized Lagrangian correspondence from $pt$ to $M$, denoted as $\fuk^\#(M)$.  Any Lagrangian correspondence $L^\flat$ from $M$ to $N$ defines an $A_\infty$ functor from $\fuk^\#(M)$ to $\fuk^\#(N)$.  In particular, there is an $A_\infty$ functor

\beq\label{e:MWW}\Phi: \fuk(M^-\times M)\to fun(\fuk^\#(M),\fuk^\#(M))\eeq
 which takes the graph of symplectomorphism $\phi$ to the action by $\phi$.  We also refer the reader to \cite[Section 5]{AS10} for a list of axioms for Mau-Wehrheim-Woodward's functor.

Now we restrict our concerns to the subcategory generated by two types of Lagrangians: product Lagrangians and graph of symplectomorphisms.  Then the functor $\Phi$ reduces to an $A_\infty$ functor with target $fun(Tw\fuk(M),Tw\fuk(M))$.  This is due to Abouzaid-Smith \cite{AS10} for the case of product Lagrangians, and the simple nature of geometric compositions with graph of symplectomorphisms.  We will denote this subcategory as $\wt\fuk(M\times M)$, and the subcategory generated solely by product Lagrangians as $\fuk(M\times M)^\times$.  What is also shown in \cite{AS10} was that

\bthm\label{t:prodLagFF}$\Phi$ is a fully faithful functor when restricted to $\fuk(M\times M)^\times$.\ethm

Alternatively, one may apply a close relative of $\Phi$ defined by Ganatra  \cite[section 9.3, 9.4]{Gan13}.  This is an $A_\infty$ functor $\sG:\wt \fuk(M\times M)\rightarrow Bimod(\fuk(M))$ (the original definition only include the identity symplectomorphism but this is an easy adaption).

Bimodules which are images of product Lagrangians has the form of $\sY_L^l\otimes_k\sY_{L'}^r$.   As in the case of functors, Ganatra showed a Yoneda-type proposition that $\sG$ is fully faithful when restricted to $\fuk(M\times M)^\times$, hence has a quasi-inverse $\sG^*$ from its image $Bimod(Fuk(M))^{\times}$.  Hence we have the following diagram which commutes up to quasi-isomorphism.

\beq\label{e:bimod-fun}\xymatrix{ && \fuk(M\times M)^\times\ar@{^{(}->}[d] && \\
Bimod(\fuk(M))\ar@{.>}[rru]^{\sG^*}&& \wt\fuk(M\times M)\ar[ll]^(.45){\sG}\ar[rr]_(.4){\Phi} && fun(\fuk(M),\fuk(M))\ar@{.>}[llu]_{\Phi^*}
}\eeq

%In this language, %\eqref{e:functorCone}\eqref{e:functorCone2} can easily  %be rewritten as a cone of bimodules which will turn out %useful.

%\beq\eeq

\section{Proof of long exact sequences}\label{s:proofLES}

We construct Lagrangian cobordisms associated to the surgery identities and deduce the long exact sequences in this section.
Throughout the whole section, we assume all Lagrangians in $M$ are $\Z$ or $\Z/N$-graded.

\blem\label{l:cob} Let $L=L_1\#_{D} L_2, L_1\#_{D,E_2} L_2$ or $\cL_1\#_{\cD,E_2} \cL_2$ as surgeries of graded Lagrangians.
Then there is a graded Lagrangian cobordism $V$ from $L_1$ and $L_2$ (or $\cL_1$ and $\cL_2$) to $L$.

\elem

\begin{proof}
We give the proof for $L=L_1\#_{D,E_2} L_2$ and the proof for $L_1\#_{D} L_2$ and $\cL_1\#_{\cD,E_2} \cL_2$ are similar.
It suffices to consider $M=T^*L_1$ and $L_2=N^*_D$ is the conormal bundle of $D$ in $L_1$.
As usual, we assume a product metric on $L_1=K_1 \times K_2$ is chosen and $D \pitchfork (\{p\} \times K_2)$ for all $p \in K_1$ so that the $E_2$-flow clean surgery can be performed.

First note that $L_1 \times \R$ intersects cleanly with $L_2 \times i\R$ at $D \times \{0\}$.
Let the grading of $L_1$ be $\theta_i$.
We give a grading $\theta_{1r}$ to $L_1 \times \R$ by requiring $\theta_{1r}(p,z)=\theta_1(p)$ for all $p \in L_1$ and $z \in \R$.
On the other hand, we equip $L_2 \times i\R$ with grading $\theta_{2i}$ such that $\theta_{2i}(p,z)=\theta_1(p)-\frac{1}{2}$ for all $p \in L_1$ and $z \in \R$.
Then we have $Ind(L_1 \times \R|_{D \times \{0\}},L_2 \times i\R|_{D \times \{0\}})=Ind(L_1|_D,L_2|_D)+Ind(\R|_0,i\R|_0)=Ind(L_1|_D,L_2|_D)$.
Moreover, we also have $Ind(L_1|_D,L_2|_D)=dim(D)+1$ by the assumption that
graded $E_2$-flow surgery from $L_1$ to $L_2$ can be performed and Lemma \ref{l:cleanIndex2} .

Pick the standard metric on $\mathbb{R}$.
By Lemma \ref{l:cleanIndex2}, we can perform the graded Lagrangian surgery from $L_1 \times \mathbb{R}$ to $L_2 \times i\mathbb{R}$ resolving the clean intersection
by a $(E_2 \oplus \R)$-flow handle $H^{D,E_2 \oplus \R}_{\nu}$, where we canonically identify $T^*(L_1 \times \R)$ as a $E_1 \oplus E_2 \oplus \R$ bundle over $L_1 \times \R$.
We note that $E_2 \oplus \R$-flow is well-defined to give a smooth Lagrangian manifold because we stayed inside the injectivity radius (Lemma \ref{l:flowGlueCleanE2}).
Hence we have a graded embedded Lagrangian cobordism with four ends.

Let $\pi: M \times \mathbb{C} \to \mathbb{C}$ be the projection on second factor and $\pi_H=\pi|_{H^{D,E_2 \oplus \R}_{\nu}}$.
We define $S_+=\{(x,y)\in \mathbb{R}^2| y \ge x\}$ and $W= \pi_H^{-1}(S_+)$.
A direct check shows that $W$ is a smooth manifold with boundary $\pi_H^{-1}(0)=L$.
Let $W_0=W \cap \pi^{-1}([-3\epsilon,0]\times[0,3\epsilon])$.
It has three boundary components, namely $L_1 \times \{(-3\epsilon,0)\}$,
$L_2 \times (0,3\epsilon)$ and $L \times \{(0,0)\}$, while $L \times \{(0,0)\}$
is the only boundary component that is not cylindrical.
One then applies a trick due to Biran-Cornea (see Section 6 of \cite{BC13}).  This yields a Hamiltonian perturbation $\varphi$ supported on $\pi^{-1}([-\epsilon,\epsilon]\times [-\epsilon, \epsilon])$, so that $\varphi(W)$ has all three cylindrical ends.  By extending $\varphi(\pi^{-1}_H(0))$ to infinity and bending the cylindrical end corresponding to $L_2$ to the left, we get the desired Lagrangian cobordism $V$.

Finally, by the identification of gradings from ends to fiber Lagrangians, we conclude that it is a cobordism from
$L_1$ and $L_2$ to $L$.
\end{proof}

\begin{figure}[h]
\centering
\includegraphics[scale=1.5]{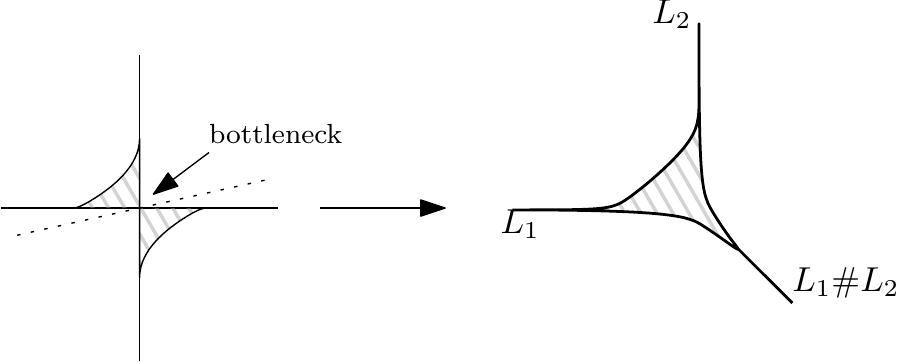}
\caption{Construction of a simple cobordism.}
\label{fig:SimCob}
\end{figure}

We call a cobordism obtained by Lemma \ref{l:cob} a {\bf simple cobordism }.
When $D$ is a single point, it reduces to the usual Lagrangian surgery and Lemma \ref{l:cob}
was discussed in Section 6 of \cite{BC13} in detail.

\begin{lemma}\label{l:exactSimCob}
Let $V$ be a simple cobordism from $L_1,L_2$ to $L$ and $D$ is connected.
If $L_1$ and $L_2$ are exact Lagrangians, then $L$ is exact and $V$ is also exact.
\end{lemma}

\begin{proof}
We give the proof for $L=L_1\#_{D,E_2} L_2$.
Without loss of generality, we can assume $M=T^*L_1$, $L_2=N^*_D$ is the conormal bundle.  We first assume $codim_{L_i}(D)\ge 2$.

Since the $E_2$-flow handle $H_{\nu}^{D,E_2}$ is obtained by a Hamiltonian flow of $N^*_D \backslash D$, it is immediate that $H_{\nu}^{D,E_2}$ is an exact Lagrangian because $\mathcal{L}_{X^{\nu}} \alpha=dK$ for a function $K$.
Let $f_1,f_2$ and $f_H$ be a primitive of $\alpha$ restricted on $L_1$, $L_2$ and $H_{\nu}^{D,E_2}$, respectively.
Since we assume $D$ is of codimension two or higher, $(f_i-f_H)|_{L_i \cap \overline{H_{\nu}^{D,E_2}}}$ are locally constants and hence constants for $i=1,2$, where $ \overline{H_{\nu}^{D,E_2}}$ denotes the closure of the handle.
By possibly adding a constant to $f_1$ and $f_2$, we can assume $f_1$, $f_2$ and $f_H$ are chosen such that they match together to give a primitive on $L$.

Now we drop the codimension assumption and only assume $codim_{L_i}(D)\ge 1$.  We recall that in the proof of Lemma \ref{l:cob}, the first step for
 constructing $V$ is to resolve $L_1 \times \R$ and $L_2 \times i\R$ along $D \times \{(0,0)\}$, which has now $codim_{L_i\times\R}(D)\ge2$.  This process preserves exactness by what we just proved. Then we cut the cobordism into a half, do Hamiltonian perturbation near $L \times \{(0,0)\}$ and extend the cylindrical end.
All of these steps preserve the exactness of the Lagrangian and hence $V$ is exact.  The restriction of $V$ to the fiber over $\{(0,0)\}$ is precisely $L$, proving the exactness of the surgery.
\end{proof}

\begin{lemma}\label{l: preserving monotonicity in simple cobordism}
Let $V$ be a simple cobordism from $L_1,L_2$ to $L$.
If $L_1$ and $L_2$ are monotone Lagrangians such that either

\begin{enumerate}[(1)]
\item $\pi_1(L_1,D)=1$ or $\pi_1(L_2,D)=1$, or
\item the image of $\pi_1(L_i)$ to $\pi_1(M)$ is torsion for either $i=1,2$
\end{enumerate}
then $L$ is monotone and $V$ is also monotone.
\end{lemma}

\begin{proof}
Again we give the proof for $L=L_1\#_{D,E_2} L_2$ and
 we first assume that $codim_{L_i}(D)\ge2$.  For convenience we decompose $L=\mathring L_1\cup \mr L_2$.  Here $\mr L_2$ is the closure of the image of $L_2\backslash D$ under the $E_2$-flow defining the surgery, and $\mr L_1$ is the closure of the complement of $\mr L_2$.

In case (1) it suffices to prove the lemma when $\pi_1(L_2,D)=1$, since the slight asymmetry of $L_1$ and $L_2$ will be irrelevant.
First note that $\pi_1(U(D), U(D)\backslash D)=\pi_1(N^*_D,N^*_D \backslash D)=1$ by our assumption on $D$, where $U(D)$ is a tubular neighborhood of $D$ in $L_2$.
Since the flow handle $H_{\nu}^{D,E_2}$ is obtained by applying an $E_2$-flow on $N^*_D \backslash D$, any path in $\mr L_2$ with ends at $H_{\nu}^{D,E_2}$ can be homotoped to a path in $H_{\nu}^{D,E_2}$, while $H_{\nu}^{D,E_2}$ in turn retracts to its boundary component that lies on $\mr L_1$.

The upshot is, we can find for any element in $\pi_2(M,L)$ a representative $u: \mathbb{D}^2 \to M$ with boundary completely lie in $\mr L_1$.
Since $L_1$ is monotone, it finishes the proof for $L$.

Case (2) is similar.  Without loss of generality, assume the image of $\pi_1(L_2) \to \pi_1(M)$ is torsion. Take again any disk $u:\bD^2\rightarrow M$ with boundary on $L$, and assume $\partial u$ intersects $\partial \mr L_2$ transversally.  For any segment $I\subset \partial u$ contained in $\mr L_2$ satisfying $\partial I\subset \partial\mr L_2$, one connects the two endpoints of $\partial I$ by $I'\subset \partial\mr L_2$ (the relevant boundary is connected due to the assumption of connectedness and codimension of $D$). By assumption, we can take a disk $v:\bD^2\rightarrow M$ with $\partial v=m[I\cup I']$ for some integer $m$.  Then one may decompose $mu$ so that $m[u]=[mu-v]+[v]$, so that $\partial v\subset \mr L_2$.  By performing such a cutting iteratively, one may assume $\partial(mu-v)\subset\mr L_1$.  Since $\partial v$ retracts to $L_2\cap \mr L_2$, the monotonicity follows from that of $L_1$ and $L_2$ with such a decomposition.

Now in either case the monotonicity of $V$ is argued in a similar way as Lemma \ref{l:exactSimCob} because all processes involved preserve monotonicity.  The restriction to the fiber over the origin again removes the assumption of $codim_{L_i}D\ge 2$.
\end{proof}

\begin{thm}[\cite{Se03}\cite{WWfamily}\cite{BCIII}]\label{thm:long exact sequence 1}
Let $(M,\omega)$ be a monotone (or exact) symplectic manifold and $S^n$ ($n >1$) an embedded Lagrangian sphere.
For monotone (or exact) Lagrangians $L_1$ and $L_2$, there is a long exact sequence
$$ \dots \to HF^*(S^n,L_2) \otimes HF^*(L_1,S^n) \to HF^*(L_1,L_2) \to HF^*(L_1,\tau_{S^n}(L_2)) \to \dots$$
\end{thm}

\begin{proof}
By Lemma \ref{l:adjust}, Corollary \ref{c:GradedSurgProduct} and Lemma \ref{l:cob}, there is a Lagrangian cobordism $V$ from $S^n \times S^n[1]$ and the diagonal $\Delta$ to $Graph(\tau_{S^n}^{-1})$ in $M \times M^-$, where $M^-=(M,-\omega)$.  By Lemma \ref{l:exactSimCob}, \ref{l: preserving monotonicity in simple cobordism}, the monotonicity (exactness) of $(M,\w)$ implies the same property for $S^n\times S^n[1]$, $\Delta\subset M\times M^-$ and the corresponding cobordism $V$.

In either case, $Graph(\tau_{S^n}^{-1})$ is a cone from $S^n \times S^n$ to $\Delta$ in the Fukaya category of $M \times M^-$ by Theorem \ref{thm: Cone from Lagrangian cobordism}.
In particular, we have a long exact sequence
$$ \dots \to HF^*( L_1 \times L_2,S^n \times S^n) \to HF^*(L_1 \times L_2, \Delta) \to HF^*(L_1 \times L_2,Graph(\tau_{S^n}^{-1})) \to \dots$$
In the language of Lagrangian correspondence, we have $HF^*(L_1 \times L_2,S^n \times S^n)=HF^*(L_1,S^n \times S^n,L_2)=HF^*(L_1,S^n) \otimes HF^*(S^n,L_2)$ by Theorem \ref{thm: quilted Kunneth formula},
where $(L_1,S^n \times S^n,L_2)$ is a generalized Lagrangian correspondence in $\{pt\} \times M \times M \times \{pt\}$.
Similarly, we have $HF^*(L_1 \times L_2, \Delta)=HF^*(L_1, \Delta, L_2)=HF^*(L_1,\Delta \circ L_2)=HF^*(L_1,L_2)$ by Theorem \ref{thm: quilted geometric composition equality}.
Finally, we also have $HF^*(L_1 \times L_2,Graph(\tau_{S^n}^{-1}))=HF^*(L_1,Graph(\tau_{S^n}^{-1}) \circ L_2)=HF^*(L_1,\tau_{S^n}(L_2))$, by Theorem \ref{thm: quilted geometric composition equality} again.

We remark that although the results in \cite{WWQuiltedFloer} require a stronger monotonicity assumption on the generalized Lagrangian correspondence, the isomorphisms we need are classical (e.g. it can be proved by hand-crafted correspondence of relevant moduli spaces) and require only monotonicity assumptions on the Lagrangians.
\end{proof}

\bcor[\cite{Se4dim}\cite{WWfamily}]\label{c:Seideltwist-fixed} In the same situation as Lemma \ref{thm:long exact sequence 1}, $f\in Symp(M)$, then

\beq\label{e:exactSeqFix}\dots \to HF^*(\tau \circ f) \to HF^*(f) \to HF^*(f(S^n),S^n) \to \dots\eeq

\ecor

\bpf The exact sequence follows from applying the cohomological functor $HF^*(-,Graph(f))$ to the cone given by the cobordism.

\epf

The above result is predicted by Seidel \cite[Remark 2.11]{Se4dim} in a slightly different form from here.  This is solely due to the cohomological convention we took.
In the following theorem, we assume all involved symplectic manifolds and Lagrangians have the same monotonicity constant with minimal Maslov number at least three.

\begin{thm}[\cite{WWfamily}]\label{thm:long exact sequence 2}
Let $C$ be a spherically fibered coisotropic manifold over the base $(B,\w_B)$ in $(M,\omega)$.
Given Lagrangians $L_1$ and $L_2$ and assume the following monotonicity conditions:

 \begin{enumerate}[(i)]
 \item the generalized Lagrangian correspondence $(L_1,C^t,C,L_2)$ is monotone in the sense of \cite{WWQuiltedFloer} and,
 \item the simple cobordism corresponding to the surgery in Corollary \ref{l:fiberedTwist} is monotone.

     \end{enumerate}
 Then there is a long exact sequence
$$ \dots \to HF^*(L_1 \times C, C^t \times L_2) \to HF^*(L_1,L_2) \to HF^*(L_1,\tau_{C}(L_2)) \to \dots$$

In particular if the spherical fiber of $C$ has dimension $>1$ or $\pi_1(M)$ is torsion, (ii) is automatic.
\end{thm}

\begin{proof}
The proof is analogous to Theorem \ref{thm:long exact sequence 1} with Lemma \ref{l:adjust} replaced by Lemma \ref{l:fiberedTwist}. Here we give a sketch.
First, $(L_1,C^t,C,L_2)$ being monotone implies $\widetilde{C}=C^t \circ C$ being monotone (See Remark 5.2.3 of \cite{WWQuiltedFloer}).
The Lagrangian cobordism in Lemma \ref{l:fiberedTwist} is monotone by Lemma \ref{l: preserving monotonicity in simple cobordism}.  It is not hard to verify $\pi_1(\wt C, \wt C\cap \Delta)=1$ when $codim_M C\ge2$.
Hence, Theorem \ref{thm: Cone from Lagrangian cobordism} applies either in this case or when $\pi_1(M)$ being torsion, and we obtain long exact sequence
$$ \dots \to HF^*( L_1 \times L_2,\widetilde{C}) \to HF^*(L_1 \times L_2, \Delta) \to HF^*(L_1 \times L_2,Graph(\tau_{C}^{-1})) \to \dots $$
With our assumption on monotonicity of $(L_1,C^t,C,L_2)$, we apply Theorem \ref{thm: quilted geometric composition equality} to obtain the desired result.
\end{proof}

There is a similar result on the fixed point version of family Dehn twist, and we will not state it explicitly here.

\section{Immersed Lagrangian cobordism}\label{s:immersedCob}

In this section, we provide a long exact sequence associated to the Dehn twist along $\mathbb{CP}^{\frac{m}{2}}$ using immersed objects.  We note the readers that this is an independent approach to a special case for $\CP^n$-objects, but it yields more information on connecting maps (see Corollary \ref{c:CpnCpnSphereMap} \ref{c:immersedLES}).
We achieve this by considering a reasonable immersed class of Lagrangian cobordisms.

\subsection{Review on immersed Lagrangian Floer theory}\label{s:immersedFloer}

We collect facts from immersed Lagrangian Floer theory that will be used later on in this section.  Our expositions follows \cite{AB14}, but one may find a much more general treatment in \cite{AJ10}.

We will restrict our attention to a very limited generality of the theory. From now on, any Lagrangian immersions $\iota:L\lar M$ have
\begin{itemize}
\item clean self-intersections which form a totally geodesic submanifold in $M$ for some metric (but not a covering)
\item a $\Z$-grading
\end{itemize}

When $M$ is non-compact, we assume it has cylindrical ends, and all almost complex structures in consideration should be compactible with these ends.  We also impose the following positivity condition holds for all Lagrangian immersions under consideration following \cite{AB14}:\\\\
{\bf (A)} $\text{If }E(p,q) >0 \text{ for some } (p,q)\in R_L \text{, then } Ind(p,q) \ge 3$.\\

Suppose $\iota_i:L_i\rightarrow M$, $i=0,1$ are a pair of Lagrangian immersions which intersect transversally at smooth points.  We define $\ov L_i=\iota_i(L_i)$ and $\ov R_{L_i}:=\iota_i(R_{L_i})$ to be the image of $L_i$ and the immersed points, respectively (see Definition \ref{d:branchJ}), and require a genericity assumption
$\ov R_{L_0}\cap \ov L_1 = \ov R_{L_1} \cap \ov L_0=\emptyset$.
%We also assume that there is a Riemannian metric %$g_{fix}$ of $M$ such that $\ov L_i$ are totally %geodesic.
We will define the immersed Floer cohomology $HF^*(\iota_0,\iota_1)$ in this setting.

\begin{comment}
\begin{rmk}\label{r:notCovering}
The restriction of Lagrangian immersion being not a covering is only for the ease of expository.
In fact, one can identify a Lagrangian covering as the image Lagrangian equipped with a non-trivial local system (See e.g. \cite{AB14}).
\end{rmk}

\begin{rmk}\label{r:gfix}
The existence of the metric $g_{fix}$ is guaranteed when the self-intersections of $\ov L_i$ are transversal
but it is not necessarily true in general. We impose this assumption for any Lagrangian immersion throughout the section.
For the application in the current paper, any Lagrangian immersion we consider have transversal self-intersections.
\end{rmk}

\begin{rmk}
Any almost complex structure of $M$ in this section is chosen to be compatible with the cylindrical end outside a large compact set such that
maximum principle applies to avoid any pseudo-holomorphic curve of interest escaping to the infinite end of $M$.
\end{rmk}

\end{comment}

\noindent\ul{\it\large Floer cochain group and differentials.}\\

The Floer cochain group is the $\Z$-graded abelian group defined as
$$CF^*(\iota_0,\iota_1):=\bigoplus_{x_i\in\iota_0(L_0)\cap\iota_1(L_1)}\bk\la x_i\ra$$

The differentials $d:CF^*(\iota_0,\iota_1)\rightarrow CF^{*+1}(\iota_0,\iota_1)$ is defined using exactly the same set of equations \eqref{e:Floer} as the embedded case, with one extra condition on the boundary

\begin{itemize}
\item $u(s,i)\in\iota(L_i)$ can be lifted continuously to curves on $L_i$ for $i=0,1$.
\end{itemize}

Equivalently, the Floer strip is not allowed to have any branch jumps on the boundary.  Orientations of these moduli spaces are handled analogously as in the embedded case, see \cite{AB14}\cite{AJ10}.\\

\noindent\ul{\it\large $d^2=0$.}\\

As usual, to prove the above-defined $d$ indeed gives a differential of the Floer complex, one needs to examine the moduli space $\cM(p_-,p_+)$ when $Ind(p_-)-Ind(p_+)=2$.  In the immersed case, the compactification of $\cM(p_-,p_+)$ is more complicated, because we have a new type of bubbling as follows.

Take any Lagrangian immersion $\iota:L\rightarrow M$.  Let $\mathbb{D}=\{z \in \mathbb{C} | \|z\| \le 1\}$ and $T \subset \partial \mathbb{D}$ be a set of discrete points, which is called a set of \textbf{branch jump marked points on disks}, ordered counterclockwisely as $(t_0,\dots,t_{|T|-1})$ such that $t_0$ is an incoming marked point and the others are outgoing.
We denote $T^-$ and $T^+$ as the set of incoming and outgoing marked points, respectively.
Let $\alpha_{|T|}: \{0,1,\dots,|T|-1 \} \to R_{L}$ be a function to specify the immersed points where branch jumps occur.
We require that $T \neq \emptyset$ and also fix a (domain independent) compatible almost complex structure $J$.

\begin{defn}\label{d:markedDisk}
The moduli space of \textbf{$\alpha_{|T|}$-marked disks} is defined to be
$$\widetilde{\mathcal{M}}(\alpha_{|T|})=\widetilde{\mathcal{M}}(\alpha_{|T|},L;J)$$
which consists of four tuples $\underline{u}=(u,T,\alpha_{|T|},l_T)$
\begin{itemize}
\item a continuous map $u: \mathbb{D} \to M $ which is smooth on $\mathbb{D} \backslash T$,
\item $u(\partial \mathbb{D}) \in \overline{L}$,
\item $\partial_su(s,t)+J(u(s,t))(\partial_tu(s,t))=0$ for all $(s,t)\in \mathbb{D} \backslash T$,
\item the energy $E(u)=\frac{1}{2}\int_{\mathbb{D} \backslash T}|du|^2$ is finite
\item $l_T$ is a continuous map $l_{T}: \partial\mathbb{D} \backslash T \to L$, which is called a lift, such that
$\iota_{L} \circ l_{T} = u|_{\partial\mathbb{D}-T}$,
\item for the $j$-th element $t_j \in T$, we have
$\alpha_{|T|}(j)=(\lim_{s \to t_j^-} l_{T}(s),\lim_{s \to t_j^+} l_{T}(s))$ if $t_j \in T^+$
(resp. $\alpha_{|T|}(j)=(\lim_{s \to t_j^+} l_{T}(s),\lim_{s \to t_j^-} l_{T}(s))$ if $t_j \in T^-$), where $s \to t_j^-$ stands for $s$ approaching to $t_j^-$ along a counterclockwise direction, and similarly for $s \to t_j^+$.
\end{itemize}

\end{defn}

We do not address the regularity
for $\widetilde{\mathcal{M}}(\alpha_{|T|})$.
In any case, there is a $Aut(\mathbb{D})$-action on $\widetilde{\mathcal{M}}(\alpha_{|T|})$ and we define
$$\mathcal{M}(\alpha_{|T|})=\widetilde{\mathcal{M}}(\alpha_{|T|})/Aut(\mathbb{D}).$$

Its compactification $\overline{\mathcal{M}}(\alpha_{|T|})$ consists of stable bubble trees by standard Gromov compactness, where nodes all locate at $R_L$.  We will referto such marked disks and their compactifications as \textbf{fishtails} following \cite{Ab08}.

\begin{figure}[h]
\centering
\includegraphics[scale=1.5]{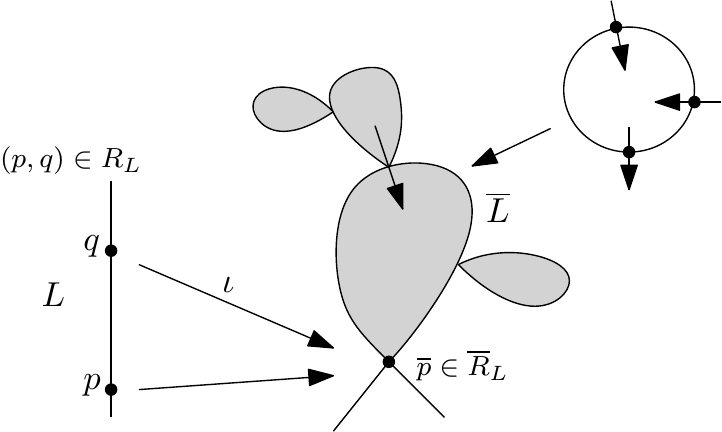}
\caption{A fishtail-type bubble which consists of a tree of fishtails.  The disk of main component is shown.}
\label{fig:fishtail}
\end{figure}
\begin{comment}
Since we have non-regular curves, we have to pay attention to the indices of these curves.
Let $TR$ be a tree with a single root with $v(TR)+1$ many vertices.
We order the vertices from $0$ to $v(TR)$ where the $0$th vertex is the root.
Let the valence of the $i$th vertex be $a_i$.
A disc bubble tree modeled on $TR$ consists of $v(TR)+1$ many marked discs $(v_0,v_1,\dots,v_{v(TR)})$ such that
$v_i$ ($i \neq 0$) has $a_i$ many marked points and $v_0$ has $a_0+1$ many marked points.
If the $i$th vertex is adjacent to the $j$th vertex and the $i$th vertex is closer to the root, then
the branch jump type of the incoming marked point of $v_j$ is the same as the branch jump type of the corresponding outgoing marked points
of $v_i$.
We define $\overline{\mathcal{M}(\alpha,\mathbf{J})}$ to be the set of all disc bubble tree modeled on any single-rooted tree such that the incoming marked point of $v_0$ has branch jump type $\alpha$.
We usually denote an element in $\overline{\mathcal{M}(\alpha,\mathbf{J})}$ as $\overline{v}$ and call $\alpha$ the branch jump type of $\overline{v}$. We define the index of $\overline{v}$ as the sum of indices of $v_i$ for $i=0,\dots,v(TR)$, which equals to the index of $\alpha$.

\end{comment}

When fishtails appear as part of the compactification of $\cM(p_-,p_+)$, the strip decomposes accordingly, where the main component is described by the following \textbf{moduli space of marked strips}.

Given a set $T$ along with a partition $T=T_0 \cup T_1$, where $T_i \subset \R \times \{i\}$ for $i=0,1$, we call $T$ a set of {\bf branch jump marked points on strips}.  All such marked points are required to be \textit{outgoing} marked points.
The points of $T_0$ are ordered increasingly in $\R$-coordinate and the points of $T_1$ are ordered decreasingly.
The {\bf type} of $T$ is the pair of integers $\bt_T=(|T_0|,|T_1|)$.

For a fixed type $\bt$ of branched jumps we also associate to the marked strip two functions $\alpha_{i}: \{1,\dots,|T_i| \} \to R_i$ for $i=0,1$.
We define $\alpha=(\alpha_{0},\alpha_{1})$ similar to the disk case, which specifies the immersed points where the branch jumps occur.

We first define the moduli space of marked strips with fixed Lagrangian label.
For a one parameter family of compatible almost complex structure $\mathbf{J}=(J_t)_{t \in [0,1]}$ and two intersection points $p_{\pm} \in \iota_0(L_0)\cap\iota_1(L_1)$, we consider the following.

\begin{defn}\label{d:markedStrip}
The moduli space of $\alpha$-marked strip with fixed Lagrangian label from $p_-$ to $p_+$ is defined to be
$$\widetilde{\mathcal{M}}(p_{-},p_+,\alpha;\mathbf{J})$$
whose elements consist of four tuples $\underline{u}=(u,T,\alpha,l_T)$
\begin{itemize}
\item $T$ is a set of branch jump marked points with $\bt_T=\bt$ ($\bt$ is part of the data in $\alpha$).
\item a continuous map $u: \R \times [0,1] \to M $ which is smooth on $\R \times [0,1] \backslash T$,
\item $u(s,i) \in \iota_i(L_i)$, $i=0,1$,
\item $\lim_{t \to \pm \infty} u(s,t)=p_{\pm}$ uniformly in $t$,
\item $\partial_su(s,t)+J_t(u(s,t))\partial_tu(s,t)=0$ for all $(s,t)\in \R \times [0,1] \backslash T$,
\item the energy $E(u)=\frac{1}{2}\int_{\R \times [0,1] \backslash T}|du|^2$ is finite
\item $l_T=(l_{T,0},l_{T,1})$ is a pair of continuous maps $l_{T,i}: \R \times \{i\} \backslash T_i \to L_i$ for $i=0,1$, which are called lifts, such that
$\iota_{L_i} \circ l_{T,i} = u|_{\R \times \{i\} \backslash T_i}$,
\item for the $j$-th element $t_j \in T_0$ (resp. $t_j \in T_1$), we have
$\alpha_{0}(j)=(\lim_{s \to t_j^-} l_{T,0}(s),\lim_{s \to t_j^+} l_{T,0}(s))$
(resp. $\alpha_{1}(j)=(\lim_{s \to t_j^+} l_{T,1}(s),\lim_{s \to t_j^-} l_{T,1}(s))$)
\end{itemize}

\end{defn}

If every element is regular, this moduli space $\widetilde{\mathcal{M}}(p_{-},p_+,\alpha)$ is a smooth manifold with dimension $Ind(p_{-})-Ind(p_+)-\sum\limits_{j=1}^{|T_0|} Ind(\alpha_{0}(j))- \sum\limits_{j=1}^{|T_1|} Ind(\alpha_{1}(j))+|T|$.
There is an $\R$-action acting on  $\widetilde{\mathcal{M}}(p_{-},p_+,\alpha)$ by translation.
We define
$$ \mathcal{M}(p_{-},p_+,\alpha)=\widetilde{\mathcal{M}}(p_{-},p_+,\alpha)/ \R$$
which is a smooth manifold when $\widetilde{\mathcal{M}}(p_{-},p_+,\alpha)$ is.  When $T=\emptyset$ this is the usual moduli of holomorphic strips.

An element in the boundary of the compactification $\ov\cM(p_-,p_+)$ then consists of (an equivalence class of)
$(u_0,\dots,u_n,\overline{v^0_{1}},\dots,\overline{v^0_{k_0}},\dots,\overline{v^{n}_1},\dots,\overline{v^n_{k_n}})$
where $(p_0=p_-,p_1,\dots,p_n,p_{n+1}=p_+)$ is a sequence of intersection points in $\iota_0 \cap \iota_1$, $u_i\in\mathcal{M}(p_i,p_{i+1},\alpha^i)$ and $\bar v^i_j\in\overline{\mathcal{M}}(\alpha_{|T^i|})$ are fishtails rooted on $\partial u_i$.
When $Ind(p_-)-Ind(p_+)=2$ and under Assumption {\bf (A)}, \cite[Lemma 6.2]{AB14} shows that fishtails can be excluded.
This concludes that $d^2=0$.

\bigskip
\noindent\ul{\it\large Continuation maps and independence of auxiliary data.}\\

To prove the well-definedness of our Floer cohomology, we need to address the independence of auxiliary data.  To define continuation maps, we need a mild modification for $\alpha$-marked strips (Definition \ref{d:markedStrip}) to moving boundary conditions.

Suppose we have a Lagrangian immersion $\iota_1:L_1\rightarrow M$, and a family of Lagrangian immersion $\iota^s_0:L_0\rightarrow M$, so that $\iota^s_0=\phi^{\fH}_s\circ\iota_0$ for some Hamiltonian $\{\fH_s\}_{s\in[0,1]}$.  Assume that $\iota^{s}_{0}\pitchfork\iota_1$
for $s=0,1$ and $Im(\iota^0_0) \cap Im(\iota^1_0) \cap Im(\iota_1)=\emptyset$. These intersection conditions allow us to not use any Hamiltonian perturbation.
For convenience, we smoothly extend the family $\iota^s_0$ such that $\iota^s_0(L_0)=\iota^0_0(L_0)$ when $s\le0$ and $\iota^s_0(L_0)=\iota^1_0(L_0)$ when $s\ge1$.

For a smooth $((s,t)\in [0,1] \times [0,1])$-family of $\w_M \oplus \w_{\mathbb{C}}$-compatible almost complex structure $\mathbf{J}_s=(J_{s,t})$ such that $J_{s,t}=J_{0,t}$ for $s$ near $0$ and $J_{s,t}=J_{1,t}$ for $s$ near $1$, two intersection points $p_{-} \in \iota^1_0(L_0) \cap \iota_1(L_1)$ and $p_{+} \in \iota^0_0(L_0) \cap \iota_1(L_1)$, a type $\bt=(|T_1|,|T_2|)$ of a set of branch jump marked points on strips, an associated pair of functions
$\alpha=(\alpha_{0},\alpha_{1})$, we are interested in curves in the following moduli.

\begin{defn}\label{d:movMarkStrip}
The moduli space of $\alpha$-marked strip with moving Lagrangian label from $p_-$ to $p_+$ is defined to be
$$\widetilde{\mathcal{M}}^m(p_{-},p_+,\alpha)=\widetilde{\mathcal{M}}^m(p_{-},p_+,\alpha;\mathbf{J}_s)$$
which consists of four tuples $\underline{u}=(u,T,\alpha,l_T)$
\begin{itemize}
\item T is a set of branch jump marked points with $\bt_T=\bt$
\item a continuous map $u: \R \times [0,1] \to M$ which is smooth on $\R \times [0,1] \backslash T$,
\item $u(s,0) \in \iota^s_0(L_0)$, $u(s,1) \in \iota_1(L_1)$,
\item $\lim_{t \to \pm \infty} u(s,t)=p_{\pm}$ uniformly in $t$,
\item $\partial_su(s,t)+J_{s,t}(u(s,t))\partial_tu(s,t)=0$ for all $(s,t)\in \R \times [0,1] \backslash T$,
\item the energy $E(u)=\frac{1}{2}\int_{\R \times [0,1] \backslash T}|du|^2$ is finite
\item $l_T=(l_{T,0},l_{T,1})$ is a pair of continuous maps $l_{T,i}: \R \times \{i\} \backslash T_i \to L_i$ for $i=0,1$, which are called lifts, such that
$(\iota^s_0) \circ l_{T,0} = u|_{\R \times \{0\}-T_0}$ and $\iota_1 \circ l_{T,1} = u|_{\R \times \{1\} \backslash T_1}$,
\item for the $j$-th element $t_j \in T_0$ (resp. $t_j \in T_1$), we have
$\alpha_{0}(j)=(\lim_{s \to t_j^-} l_{T,0}(s),\lim_{s \to t_j^+} l_{T,0}(s))$
(resp. $\alpha_{1}(j)=(\lim_{s \to t_j^+} l_{T,1}(s),\lim_{s \to t_j^-} l_{T,1}(s))$)
\end{itemize}

\end{defn}

If every element is regular, this moduli space $\widetilde{\mathcal{M}}^m(p_{-},p_+,\alpha)$ is a smooth manifold with dimension $Ind(p_{-})-Ind(p_+)-\sum\limits_{j=1}^{|T_0|} Ind(\alpha_{0}(j))- \sum\limits_{j=1}^{|T_1|} Ind(\alpha_{1}(j))+|T|$.
There is no $\R$-action acting on  $\widetilde{\mathcal{M}}^m(p_{-},p_+,\alpha)$ by translation so
we define
$$ \mathcal{M}^m(p_{-},p_+,\alpha)=\widetilde{\mathcal{M}}^m(p_{-},p_+,\alpha)$$

As in the embedded case, we define continuation maps by considering a dimension $0$ moduli space $\cM^m(p_-,p_+)$ with type $\bt=\emptyset$ and prove such moduli problem gives a chain map between $CF(\iota_0^0,\iota_1)$ and $CF(\iota_0^1,\iota_1)$.  To this end, we examine the boundary of dimension $1$ moduli of the same type.
Again in the immersed case, the new issue is the possible existence of fishtails bubblings, and the rest of the boundary components contributes to the chain map.  With assumption {\bf(A)}, Alston-Bao \cite[Lemma 6.2]{AB14} again eliminated the fishtail bubblings similar to the proof of $d^2=0$.  This shows the continuation map defines a chain map between $CF(\iota_0^0,\iota_1)$ and $CF(\iota_0^1,\iota_1)$.

The last step is to show that such continuation maps give isomorphisms on cohomology.  This again follows from the original argument of Floer: we consider a family of continuation data, namely a $(r,s)$-parameter family of Lagrangian immersions $(\iota_0^{r,s})_{r,s \in [0,1]}$ (induced by $(r,s)$-parameter of Hamiltonian) interpolating $\iota_0^{0,s}$ and $\iota_0^{1,s}$ together with a $(r,s,t)$-parameter family $(J_{r,s,t})_{r,s,t \in [0,1]}$ interpolating $J_{0,s,t}$ and $J_{1,s,t}$, to construct a chain homotopy between the two continuation maps, by considering the boundary of a moduli space $\cM(p_-,p_+;\fJ_{r,s})$ defined analogously to Definition \ref{d:movMarkStrip}, where $ind(p_-)=ind(p_+)-1$.  We then again examine the boundary of a one dimensional moduli space of such shape, which gives the desired chain homotopy identity.  The key is again to use Assumption {\bf(A)} to exclude fishtail bubbles by a dimension count argument, and the rest of the proof is routine which we will not reproduce here.

\subsection{Immersed Lagrangian cobordism theory}\label{s:immersedCob}

\subsubsection{Immersed Lagrangian cobordisms and testing Lagrangians}\label{s:ILFreview}

In this subsubsection, we extend the Lagranian cobordism theory in \cite{BC13} to the immersed Lagrangian Floer setting (\cite{AB14}).
We assume $(M,\w)$ is exact, $2c_1(M)=0$ and any compact oriented exact Lagrangian immersion is equipped with integer grading.
%We do not allow clean immersions to be covering and we %assume the existence of $g_{fix}$ as in Remark %\ref{r:notCovering} and \ref{r:gfix}.

We first introduce the notion of immersed cobordisms.

\bdf\label{d:conical}  Let $\iota:L\rightarrow (M,\w)$ be an immersion with clean self-intersection.  A \textbf{pinched Lagrangian associated to $\iota$} is a Lagrangian immersion with clean self-intersection $\wt\iota: \wt L \rightarrow M\times\C$ so that

\begin{itemize}
\item there is an open embedding $I_L:L \times \R \to \wt L$ so that its closure is also embedded.  We then identify $L \times \R$ as a subset of $ \wt L$,
\item for some $t_0\in\R$ and $z_0\in\C$, $\wt\iota|_{\wt L \backslash (L \times \{t_0\})}$ is an embedding, and $\wt\iota (x,t_0)=(\iota(x),z_0)$ for all $x \in L$,
\item there is a neighborhood $U$ of $z_0$ such that $ (\pi_{\mathbb{C}} \circ \wt\iota)^{-1}(U) \subset L \times \R$ and
the projection $Im(\pi_{\mathbb{C}}\circ\wt\iota) \cap U$ is a double cone region (pinched region) bounded by two straight line segments on $\C$.
\end{itemize}

\noindent In this case, $\wt\iota(L\times\{t_0\})$ is called the {\bf Lagrangian bottleneck}, $z_0$ is called the { \bf bottleneck on $\mathbb{C}$}
and $M \times \{z_0\}$ is called the {\bf bottleneck}.
When it is clear from the context, we use bottleneck to refer to $\wt\iota(L\times\{t_0\})$, $z_0$ or $M \times \{z_0\}$.
\edf

Most cases we will use pinched Lagrangians with domain $\wt L =L \times \R$, but later on we need the full flexibility of pinched Lagrangians (see Lemma \ref{l:CpnCpnImmersedSphere}).
We remark that if $\iota$ is an embedding and $\gamma : \R \to \mathbb{C}$ is a simple smooth curve,
$\iota \times \gamma$ is also considered a pinched Lagrangian because we allow the two straight lines on $\C$ forming the ``double cone" region to be the same line.

\begin{figure}[h]
\centering
\includegraphics[scale=1]{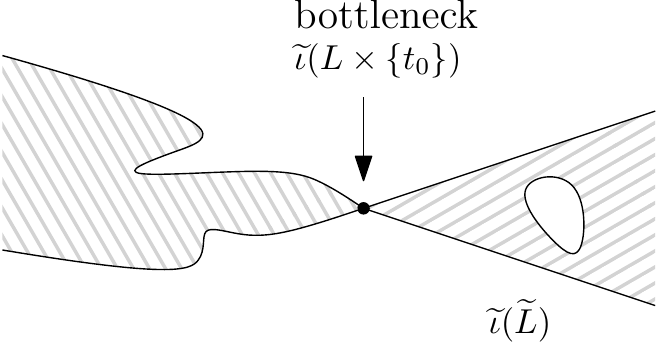}
\caption{A pinched Lagrangian.}
\label{fig:conical}
\end{figure}

\begin{figure}[h]
\centering
\includegraphics[scale=1.5]{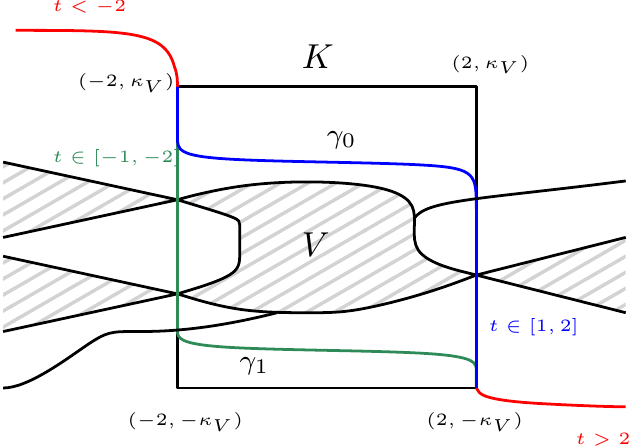}
\caption{A picture of immersed cobordism and horizontal isotopy. $\gamma_{0}$ consists of the red line and the blue line, while $\gamma_{1}$ consists of the red line and the green line. They illustrate the ends of an admissible
 horizontal isotopy from $\iota_{N,0}$ to $\iota_{N,1}$.}
\label{fig:testingCurve}
\end{figure}

\begin{defn}\label{d:ImmCobwBottleneck}
For Lagrangian immersions $\iota_{L_i}: L_i \to (M,\omega)$ and
$\iota_{L_j'}: L_j' \to (M,\omega)$,
a Lagrangian immersion $\iota_V: V \to (M \times \mathbb{C},\w_M \oplus \w_{\mathbb{C}})$ is an {\bf immersed Lagrangian cobordism with bottlenecks} from $(\iota_{L_1},\dots,\iota_{L_k})$ to $(\iota_{L_{k'}'},\dots,\iota_{L_1'})$
if there is $K'=[-1,1]\times (-\kappa_V,\kappa_V) \subset \mathbb{C}$ such that $\iota_V|_{V \backslash \iota_V^{-1}(M \times K')}$ is a disjoint union of pinched Lagrangians
    $$\wt\iota_i: \wt L_i=L_i\times(-\infty,-1)\rightarrow M\times\C,$$
     $$\wt\iota_j': \wt L'_j=L_j'\times(1,\infty)  \rightarrow M\times\C,$$

%We identify $L_i\times(-\infty,-1) \subset \wt{L_i}$ and %$L_j' \times(1,\infty) \subset \wt{L'_j}$ as in the %definition of pinched Lagrangian.
We assume the Lagrangian bottleneck of $\wt\iota_i$ (resp. $\wt\iota_j'$) being $\wt\iota_{i} (L_i \times \{-2\})$
(resp. $\wt\iota_{j}' (L_j' \times \{2\})$), whose bottleneck on $\mathbb{C}$ being $z_i$ (resp. $z_j'$).
 We also require $Re(z_i)=-2$ (resp. $Re(z_j')=2$) and $Im(z_{i_1}) < Im(z_{i_2})$ for $i_1 < i_2$ (resp. $Im(z_{j_1}') < Im(z_{j_2}')$ for $j_1 < j_2$).
 Finally, we require $Im(\wt\iota_i) \cup Im(\wt\iota_j') \subset \R \times (-\kappa_V,\kappa_V)$.
 We call $\wt\iota_i$ the \textbf{negative ends} and $\wt\iota_j'$ the \textbf{positive ends},
and $K=(-2,2)\times (-\kappa_V,\kappa_V)$ is called the {\bf body}.

\end{defn}

\begin{defn}\label{d:testingCurve}
A {\bf testing curve} $\gamma:\R \to \mathbb{C}$ is a smoothly embedded curve satisfying the following conditions:

 \begin{enumerate}[(1)]
 \item The projections to axes $\pi_x\circ\gamma(t)$ is non-decreasing and $\pi_y\circ\gamma(t)$ is non-increasing% in $x$-coordinate %and such that

 \item
$$ \pi_x\circ\gamma(t)=  \left\{
\begin{aligned}[l]
-2 &\text{ when } t\in[-2,-1],\\
2 & \text{ when } t\in[1,2].   \\
          \end{aligned} \right. $$
 \end{enumerate}
 and is strictly increasing otherwise.
$\pi_y\circ\gamma(t)$ is strictly decreasing when $t\in[-2,-1]$ or $[1,2]$, and is constant when $|t|>R$ for some $R\in\R$.

A \textbf{testing Lagrangian} is a product of Lagrangian immersion $\iota_N:N\rightarrow M$ with some testing curve $\gamma$, denoted as $\iota_{N,\gamma}:N\times\R\rightarrow M\times\C$.

\end{defn}

\bdf\label{d:horizontalIsot}
Let $\{ \gamma_s \}$ be a Hamiltonian isotopy (not necessarily compactly supported) from $\gamma_0$ to $\gamma_1$ within the class of testing curves in $\C$, requiring the image of $\gamma_s$ overlaps outside a sufficiently large compact set for all $s\in[0,1]$.  Then the product of such a family with $\iota_N:N\rightarrow M$, denoted as $\iota_{N,s}:N\times\gamma_s\rightarrow M\times\C$, is called a \textbf{horizontal isotopy}.

\edf

We alert the readers that the definition of horizontal isotopy here is different from the one in \cite{BC13} although they are similar in spirit.

\begin{defn}\label{d:admissibleLagrangian}
Let $\iota_V$ be an exact immersed Lagrangian cobordism with bottlenecks from $(\iota_{L_1},\dots,\iota_{L_k})$ to $(\iota_{L_{k'}'},\dots,\iota_{L_{1}'})$ with body $K=(-2,2) \times (-\kappa_V,\kappa_V)$.
An {\bf admissible testing curve} with respect to $\iota_V$ is a testing curve $\gamma$ such that the following holds:

\begin{itemize}
\item $\gamma$ passes through $(-2,\kappa_V)$ and $(2,-\kappa_V)$,
\item if $\gamma(t_0)$ lies on a bottleneck of $\iota_V$ on $\mathbb{C}$, then $t_0\in(-2,-1)$ or $t_0\in(1,2)$ (equivalently, $\gamma'(t)$ is pointing downwards).
\end{itemize}

For any exact Lagrangian immersion with clean self-intersection $\iota_N: N \to M$, if $\iota_{N,\gamma}=\iota_N \times \gamma$ intersects $\iota_{V}$ transversally at smooth points, then $\iota_{N,\gamma}$
is called an {\bf admissible testing Lagrangian}.

Lastly, an \textbf{admissible horizontal isotopy} $\iota_{N,s}:N\times\gamma_s\rightarrow M\times\C$ is a horizontal isotopy such that $\iota_{N,0}$ and $\iota_{N,1}$ are both admissible testing Lagrangians, which satisfies $Im(\iota_{N,0})\cap Im(\iota_{N,1})\cap \ov V=\emptyset$.

\end{defn}

\subsection*{Main Statement}

Let $\iota_{V}$ be a graded immersed Lagrangian cobordism with bottlenecks in $M \times \mathbb{C}$ and $K=(-2,2)\times (-\kappa_{V},\kappa_{V})$ be its body.
Let $\iota_{N,\gamma}=\iota_N \times \gamma$ be an admissible testing Lagrangian.

A {\bf semi-admissible almost complex structure} of $M \times \mathbb{C}$ is an
$\omega_M \oplus \omega_{\mathbb{C}}$-compatible almost complex structure $J$ such that it
is product-like ($J=J_{M} \oplus J_{\mathbb{C}}$) outside $M \times K$.
A semi-admissible $J$ is {\bf admissible} if there is an open set $K'$ with $Cl(K') \subset K$
such that $J$ is product-like outside $M \times K'$ (note that this \textit{includes a neighborhood of} $\partial K$).
An admissible (resp. semi-admissible) {\bf Floer data} $\mathbf{J}=\{J_t\}_{t \in [0,1]}$
is a choice of a smooth family of admissible (resp. semi-admissible) almost complex structure.
We remark on the various admissibility conditions.
The admissibility of the testing Lagrangian and horizontal isotopy implies that they are in a generic position without perturbation,
hence we do not include a Hamiltonian in the Floer data because we focus on the
Floer cohomology between $\iota_{N\times\gamma}$ and $\iota_{V}$ in this paper.
Admissibility of $\mathbf{J}$ is convenient for actual computations (see Theorem \ref{t:ImmersedFloerLES}),
while the semi-admissible Floer data is more flexible on the technical level.

The main result of this subsection is the following.

\begin{thm}\label{t:ImmersedFloerAd}
Let $\iota_{V}$ be an exact immersed Lagrangian cobordism with bottlenecks from $(\iota_{L_1},\dots,\iota_{L_k})$ to $(\iota_{L_{k'}'},\dots,\iota_{L_1'})$ such that all $\iota_{L_j}$, $\iota_{L_j'}$ satisfy the Assumption {\bf (A)}.
Let $\iota_{N,\gamma}$ be an admissible testing Lagrangian such that $\iota_N$ satisfies the Assumption {\bf (A)}.
Then for generic choice of admissible Floer data $\mathbf{J}$, the immersed Lagrangian Floer cohomology
$HF(\iota_{N,\gamma},\iota_{V})=H(CF(\iota_{N,\gamma},\iota_{V},\mathbf{J}),\partial)$
is well defined, independent of the choice of regular $\mathbf{J}$ and is invariant under admissible horizontal isotopy of $\iota_{N,\gamma}$.
\end{thm}

The actual proof is carried out by passing to \textit{semi-admissible} Floer data: admissible data will be crucial for comparing the Floer cohomology between cobordisms with those on the bottlenecks, it does not provide enough genericity on a technical level.

\begin{thm}\label{t:ImmersedFloer}
Under the assumption of Theorem \ref{t:ImmersedFloerAd},
the conclusion of Theorem \ref{t:ImmersedFloerAd} holds for generic choice of semi-admissible Floer data $\mathbf{J}$.
\end{thm}

From now on, a Floer data is a semi-admissible Floer data and any (family of) compatible almost complex structure
is taken in the space of semi-admissible almost complex structure, unless specified otherwise.\\

\noindent{\bf Note on a convention:} For ease of notations, we sometimes use $V_0$ in place of $N\times\gamma$,
$V_{0,s}$ in place of $N \times \gamma_s$ and $V_1$ in place of $V$ in the rest of the section.
\bigbreak

The well-definedness involves several steps. The first of which is to address the regularity of all rigid Floer strips. Secondly, we have to address the compactness of the moduli spaces of rigid Floer strips. These two steps guarantee that the differential $\partial$ is a well-defined map. Thirdly, we have to address the regularity as well as the Gromov compactification of one dimensional moduli of Floer strips to prove $\partial \circ \partial=0$. The independence of $\mathbf{J}$ uses the standard cobordism argument.
Finally, the invariance under admissible horizontal isotopy results from the regularity and Gromov compactification consideration of Floer strips with moving boundary condition.

We highlight several ingredients that are needed in our setting.
Since we require $\mathbf{J}$ to be product-like at the bottlenecks, there can be non-generic pseudo-holomorpic polygons on bottlenecks no matter how generic $\mathbf{J}$ is.
However we show that Floer strips are regular, which essential follows from a similar consideration in \cite{BC2}.
For compactness, we need the bottlenecks and the treatment in \cite{BC2} to argue that the curves we interested are inside a fixed compact set.
In the meanwhile, since our Lagrangians are immersed, bubbling at immersed points could possibly appear in the Gromov compactification in a priori.
If it happened, the argument would break down, say the differential might not square to zero.
This will be eliminated by the Assumption {\bf (A)} and an index computation modified from \cite{AB14}.

\subsubsection{Curves of interests and their compactness}\label{s:compactness}

Following the general theme of immersed Lagrangian Floer theory, what we need to examine is the compactness and regularity of following moduli spaces of holomorphic curves as well as excluding the fishtails bubbling.

\begin{defn}\label{d:curveInt} We call curves in the following moduli spaces the \textbf{curves of interests}:
\begin{itemize}
\item $\widetilde{\mathcal{M}}(p_{-},p_+,\alpha,\mathbf{J})$, $p_\pm\in \iota_{V_0}\cap\iota_{V_1}$,
\item    $\widetilde{\mathcal{M}}^m(p_{-},p_+,\alpha,\mathbf{J}_s)$, $p_- \in \iota_{N,1}\cap \iota_V$, $p_+ \in \iota_{N,0}\cap\iota_V$ for an admissible horizontal isotopy $\iota_{N,s}$,
\item $\widetilde{\mathcal{M}}(\alpha_{|T|},V_0; J)$ and $\widetilde{\mathcal{M}}(\alpha_{|T|},V_1; J)$, which appears in the boundary of the above two moduli spaces.
\end{itemize}
\end{defn}

The difficulty of compactness lies in that, immersed Lagrangian cobordisms with bottlenecks are non-compact Lagrangian submanifolds of $M\times\C$ with \textit{non-cylindrical ends} thus the usual compactness theorems of holomorphic curves does not apply.  Instead, we follow the idea of Biran and Cornea of projections to the $\C$-factor.  We first recall the following Proposition in \cite{BC2}, which is an application of open mapping theorem.

\begin{prop}[Proposition 3.3.1 of \cite{BC2}]\label{p: open mapping theorem}
Let ${\Sigma}$ be a connected Riemann surface (possibly with boundary) and ${u}: {\Sigma} \to \mathbb{C}$ be a holomorphic map.
If  ${u}({\Sigma}) \cap U \neq \emptyset$ for some connected open subset $U$ of $\mathbb{C}$ such that
\begin{itemize}
\item ${u}(\partial {\Sigma}) \cap U =\emptyset$, and
\item $ (\overline{{u}({\Sigma})} - {u}({\Sigma})) \cap U =\emptyset$
\end{itemize}
then $U \subset {u}({\Sigma})$.
\end{prop}

The bottleneck feature of $\overline V$ has the following consequence.

\begin{lemma}\label{l:compactness}
For any curve of interest $u:\Sigma \to M \times \mathbb{C}$, we have $Im(u) \subset M \times Cl(K)$.  Moreover, one of the following is true:
\begin{itemize}
\item $Im(\pi_\C\circ u)\cap K\neq\emptyset$.
\item $Im(\pi_\C\circ u)$ is a point of bottleneck.
\end{itemize}
\end{lemma}

\bpf First of all, note that the boundary punctures of $u$ are mapped to either an immersed point or an intersection $\iota_V\cap\iota_{N,\gamma}$. Therefore, the images of $\pi_\C\circ u$ must be precompact hence bounded.

It is clear that $\pi_\C\circ u\subset Cl(K)\cup\pi_\C(V)$. Otherwise, say if $u$ is a marked strip, $\pi_\C\circ u$ is either an unbounded region by Proposition \ref{p: open mapping theorem}, which is a contradiction; or is a constant map to a point outside $Cl(K)$, violating boundary conditions of curves of interests. On the other hand, say if $u$ is a fishtail, it could have boundary on $\iota_{N,\gamma}$ and $\pi_\C\circ u$ maps to a point on $\gamma \backslash Cl(K)$
but such curves cannot appear on the boundary component of  $\widetilde{\mathcal{M}}(p_{-},p_+,\alpha,\mathbf{J})$, or  $\widetilde{\mathcal{M}}^m(p_{-},p_+,\alpha,\mathbf{J}_s)$.

We now argue that, if $Im(u) \cap (M\times K)\neq\emptyset$ then $\pi_\C\circ u(\Sigma)\cap(\pi_\C(V) \backslash Cl(K))=\emptyset$.  Assume the contrary, let $z_1,z_2\in\Sigma$ with $\pi_\C\circ u(z_1)\in\pi_\C(V)\backslash Cl(K)$ and $\pi_\C\circ u(z_2)\in K$.  Take any path joining $z_1$ and $z_2$ in $\Sigma$ and stays inside interior of $\Sigma$ except possibly at endpoints, there is a point $z\in\Sigma$ such that $u(z)$ lies on a bottleneck (the path has to be projected within the shaded area in Figure \ref{fig:testingCurve} when it is outside $K$ by what we proved).  The open mapping theorem applied to $\pi_\C\circ u(z)$ immediately yields a contradiction to $\pi_\C\circ u\subset Cl(K)\cup\pi_\C(V)$.

Lastly, if $Im(u) \cap (M\times K)=\emptyset$, we want to show that $Im(u)$ is on a bottleneck.
Again assume the contrary.  We first examine the case when $u$ is a marked strip with fixed boundary conditions.
By the boundary conditions, the two asymptotes of $u$ get mapped to a bottleneck, say $M\times \{z_0\}$.
Moreover, the image of $\R\times \{0\}$ projects to $\gamma$ and hence project constantly to $z_0$,
which contradicts the open mapping theorem on $\C$ (with reflection principle) if $\pi_\C\circ u$ is not constant.
Hence $Im(u)$ is on a bottleneck if $u$ is a marked strip with fixed boundary conditions and $Im(u) \cap (M\times K)=\emptyset$.
The case is similar for the moving boundary case.
Finally, fishtails that completely lie outside $M \times Cl(K)$ might exist, but will never appear on the boundary of the moduli of the first two kinds of curves of interests: before bubbling off such a fishtail, the marked strips have to stay completely outside $M\times K$, while we just proved such strips have to lie on a bottleneck, leading to a contradiction to Gromov's compactness.

\epf

\subsubsection{Regularity}\label{s:regularity}

We address the regularity by the following Lemma.
\begin{lemma}\label{l:regularity}
For generic choice of semi-admissible Floer data $\mathbf{J}$ (resp. $\mathbf{J}_s$ for curves with moving boundary condition), the following curves of interest are regular
\begin{itemize}
\item $\alpha$-marked Floer strips $u$ with fixed (resp. moving) boundary condition such that $Im(\pi_2 \circ u) \cap K \neq \emptyset$, and
\item marked Floer strips without branch jumps.
%such that $Ind(p_-)-Ind(p_+)=1$
\end{itemize}
\end{lemma}
Note that marked discs and marked Floer strips with at least one branch jump marked point that completely lie on the one of the bottlenecks are not necessarily regular for generic semi-admissible data, but this will not concern us.
%An analogous result for admissible $\mathbf{J}$ is proved %in Lemma \ref{l:regularityAd} but the result is weaker.

\begin{proof}
If $u$ is a $\alpha$-marked Floer strip (possibly with moving boundary condition) such that $Im(\pi_2 \circ u) \cap K \neq \emptyset$, then it is regular for generic $\mathbf{J}$ (or $\mathbf{J}_s$) because we impose no constraint on $\mathbf{J}$ (or $\mathbf{J}_s$) in $M \times K$ (See [\cite{McDuff_Salamon_2004}]).
We remark that $\overline{V_0}$ and $\overline{V_1}$ being transversally intersecting at smooth points and $\overline{V_{0,0}}\cap \overline{V_{0,1}} \cap \overline{V_1} =\emptyset$ in the definition of admissible horizontal isotopy are needed for us to find the generic $\mathbf{J}$.

Now, we assume $u$ is a Floer strip without branch jumps with $Im(u) \cap M \times K =\emptyset$.
By Lemma \ref{l:compactness} together with the boundary conditions, $Im(u)$ must lie on a bottleneck with fixed Lagrangian boundary condition.

We want to argue that although $\mathbf{J}$ is not a product near the bottleneck (it is a product at the bottleneck), when studying the regularity of these $u$, we can still assume the Cauchy Riemann operator splits and it suffices to study the surjectivity of the operators in the two factors individually.
The following Lemma explains why the operator splits when $\mathbf{J}$ is domain-independent.
The situation for domain dependent $\mathbf{J}$ does not impose additional difficulties by identifying $\mathbf{J}$-holomorphic map as $\widetilde{\mathbf{J}}$-holomorphic section of $\Sigma \times M \to \Sigma$, where $\widetilde{\mathbf{J}}$ is a domain-independent vertical almost complex structure induced by $\mathbf{J}$ (See Chapter 8 of \cite{McDuff_Salamon_2004}, or Section $(8h),(8i)$ of \cite{Seidelbook}).

\begin{lemma}\label{l:operatorSplits}
Let $J$ be an $\w \oplus \w_{\mathbb{C}}$-compatible almost complex structure such that $J|_{M \times (\mathbb{C} \backslash K)}=J^v\oplus J^h$ is a product in $M \times (\mathbb{C} \backslash K)$.
If $u:\Sigma=\R \times [0,1] \to M \times \mathbb{C}$ is $J$-holomorphic with boundary conditions $u(s,0)\in \overline{V_0}$, $u(s,1)\in \overline{V_1}$ and asymptotic to $p_-,p_+ \in \overline{V_0} \cap \overline{V_1}$ such that $Im(u)$ lies on a bottleneck, then the linearized operator
$D_u:W^{k,p}(\Sigma,u^*(T(M\times \mathbb{C}))) \to W^{k-1}(\Sigma,\Lambda^{0,1}\otimes_Ju^*T(M \times \mathbb{C}))$
splits as a direct sum of two linearized Cauchy-Riemann operators $D_u^v:W^{k,p}(\Sigma,u^*(TM)) \to W^{k-1}(\Sigma,\Lambda^{0,1}\otimes_{J^v}u^*TM)$
and $D_u^h:W^{k,p}(\Sigma,u^*(T\mathbb{C})) \to W^{k-1}(\Sigma,\Lambda^{0,1}\otimes_{J^h}u^*T\mathbb{C})$.

Here $D^v_u$ is the Cauchy-Riemann operator when $u$ is considered as a $J^v$-holomorphic curve on the fiber, and $D_u^h$ is a linearized operator of a constant solution for a Cauchy Riemann problem with moving boundary condition.

\end{lemma}

\begin{proof}%[Proof of Lemma \ref{l:operatorSplits}]
Let $l_i:\R \times \{i\} \to V_{i}$ be such that $\iota_{V_{i}} \circ l_i=u|_{\R \times \{i\}}$ for $i=0,1$.
Without loss of generality, assume the bottleneck is $M \times \{z_1\}$, and the corresponding Lagrangian immersion on this bottleneck is $\iota_{L_1}$.
The linearized Fredholm operator is given by
$$D_u \xi =\eta ds-J(u) \eta dt$$
 where $\eta=\frac{1}{2}(\partial_s \xi +J(u)\partial_t \xi +\partial_{\xi}J(u) \partial_t u)$ and
 $\xi \in W^{k,p}(\Sigma,u^*(T(M\times \mathbb{C})))$ satisfies $\xi(s,i) \in l_i^*TV_{i}$.

Notice that we have a canonical splitting $TV|_{\iota_{V}^{-1}(M\times\{z_1\})}=T^v_1 \oplus T^h_1$.
Here, the vertical splitting is given by $T^v_1 = \{v \in TV| d(\iota_{V})(v) \in T(M\times\{z_1\})\}$ which is rank $n$ and the horizontal rank $1$ splitting is given by
$T^h_1= \{ v \in TV|_{\iota_{V_1}^{-1}(M\times\{z_1\})}| d(\pi_M \circ \iota_{V})(v)=0   \}  $.
We note that $T^h_1$ being well-defined and of rank $1$ uses the fact that $\iota_{V}$ is Lagrangian.
There is also a canonical splitting of $T(N\times\gamma)$ as $T^v_0 \oplus T^h_0$.
Therefore, the domain of $D_u$ splits.
If we write $u=(u^v,u^h)=(\pi_M \circ u, \pi_\C \circ u)$, then
the first summand $Dom^v(D_u)$ of domain of $D_u$ is $\{\xi^v \in W^{k,p}(\Sigma,u^*(TM))| \xi^v(s,i) \in l_i^*T^v_i, i=0,1\}$ and the
second summand $Dom^h(D_u)$ is defined similarly.
Obviously, the target of $D_u$ also canonically splits into vertical $Tar^v(D_u)$ and horizontal $Tar^h(D_u)$ components because $J$ splits in the image of $u$.

We define $D^v_u: Dom^v(D_u) \to Tar^v(D_u)$ to be
$$D^v_u \xi^v =\eta^v ds-J^v(u^v) \eta^v dt$$
 where $\eta^v=\frac{1}{2}(\partial_s \xi^v +J^v(u^v)\partial_t \xi^v +\partial_{\xi^v}J^v(u^v) \partial_t u^v)$
 and similarly for $D^h_u: Dom^h(D_u) \to Tar^h(D_u)$.
We claim that for any $\xi^v \in Dom^v(D_u)$ and $\xi^h \in Dom^h(D_u)$, we have
$$D_u^v\xi^v+D_u^h\xi^h=D_u (\xi^v+\xi^h)$$
The only less obvious equality is $\partial_{\xi^v}J^v(u^v) \partial_t u^v+\partial_{\xi^h}J^h(u^h) \partial_t u^h=\partial_{\xi^v+\xi^h}J(u) \partial_t u$.
This involves the the first order derivative of $J$.
Since $J$ is product-like over $M \times (\mathbb{C} \backslash K)$, the first order derivative of $J$ can be computed in this product-like region (The key point is that this set is closed and includes $M\times\{z_1\}$).
Therefore, we have $D_u=D_u^v\oplus D_u^h$.

Finally, we want to explicitly identify geometric meaning of $D_u^v$ and $D_u^h$ as linearized operators.
The first operator $D_u^v$ is exactly the linearized operator obtained by regarding $u^v: \Sigma \to (M,\w,J^v)$ as a $J^v$-holomorphic map with fixed Lagrangian boundary conditions on $N$ and $L_1$.
The second operator $D_u^h$ is the linearized operator of the constant $J^h$-holomorphic map $u^h: \Sigma \to (\mathbb{C},\w_{\mathbb{C}},J^h)$
with moving Lagrangian boundary conditions $u^h(s,0)\in \gamma$ and $u^h(s,1) \in \gamma^l_s$, where $\gamma^l_s$ is a family of lines such that $T_{z_1}\gamma^l_s=d(\pi_{\C} \circ \iota_{V})T^h_1(u(s,1))$ and $T^h_1(u(s,1))$ is the horizontal component of the canonical splitting of $TV$ at $l_1(s)$.  See Figure \ref{fig:FredholmSplit} for the $u^h$ part.

\end{proof}

\begin{figure}[h]
\centering

\includegraphics[scale=1.28]{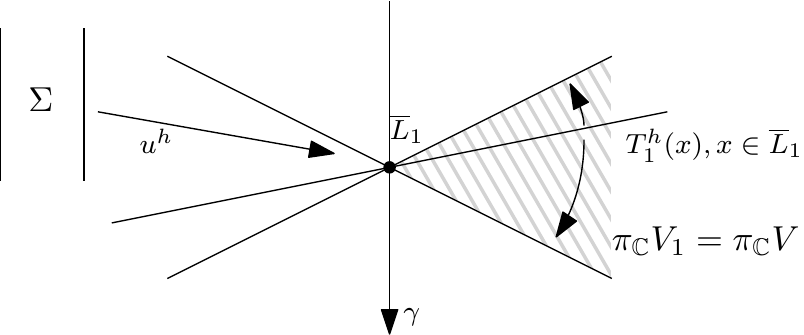}

\caption{The moving boundary problem for $u^h$.}
\label{fig:FredholmSplit}
\end{figure}

Next, we want to examine the Fredholm indices of $D^v_u$ and $D^h_u$ coming from the splitting.

\begin{lemma}\label{l:indexUnchange}
Let $\widetilde{p}_-,\widetilde{p}_+ \in \overline{N} \cap \overline{L_i}$ (resp. $\widetilde{p}_-,\widetilde{p}_+ \in \overline{N} \cap \overline{L_j'}$).
Let $p_-=\widetilde{p}_- \times z_i,p_+=\widetilde{p}_+ \times z_i$ (resp. $p_-=\widetilde{p}_- \times z_j',p_+=\widetilde{p}_+ \times z_j'$ ) be the corresponding points on the bottleneck $M \times z_i$ (resp. $M \times z_j'$).
Then we have $Ind(p_-)-Ind(p_+)=Ind(\widetilde{p}_-)-Ind(\widetilde{p}_+)$.

Moreover, any Floer strip without branch jump marked points from $p_-$ to $p_+$ on the bottleneck is regular for generic choice of $\mathbf{J}$ that is product like in $M \times (\mathbb{C} \backslash K)$.
\end{lemma}

\begin{proof}%[Proof of Lemma \ref{l:indexUnchange}]

We assume the bottleneck is $M \times \{z_1\}$ and the corresponding Lagrangian immersion is $\iota_{L_1}$.
As mentioned in the proof of Lemma \ref{l:operatorSplits}, we have a canonical splitting of $T_{p_\pm} (N\times\gamma)
=T_{\wt p_\pm}(\ov N)\oplus T_{z_1}\gamma := T_0^v(p_\pm) \oplus T_0^h(p_\pm)$ and $T_{p_\pm} V=T^v_1(p_\pm)\oplus T^h_1(p_\pm)$.
Here, we have a canonical identification %$T^v_1(p)=T_{\widetilde{p}}\overline{N}$,
 $T^v_1(p_{\pm})=T_{\widetilde{p}_{\pm}}\overline L_1$.
We also have,  for all $x \in \overline L_1=\overline{V_1} \cap M\times\{z_1\}$, that $T^h_1(x)$ is transversal to  $T_{z_1}\gamma=T_{\pi_\C(p_{\pm})}\gamma$, by viewing them as two lines in $\mathbb{C}$.

Hence the index calculation splits as
\beq\label{e:splitInd}Ind(T_{p_{\pm}}V_0,T_{p_{\pm}}V_1)=Ind(T_{\widetilde{p}_{\pm}}N,T_{\widetilde{p}_{\pm}}L_1)
+Ind(T_{z_1}\gamma,T^h_1(p_{\pm}))\eeq
By assumption, $\overline L_1$ is connected and $T_{\pi_\C(p_{\pm})}\gamma=T_{z_1}\gamma$ is transversal to $T^h_1(x)$ for all $x \in \overline L_1$. This implies that
\beq\label{e:equalInd}Ind(T_{z_1}\gamma,T^h_1(p_-))=Ind(T_{z_1}\gamma,T^h_1(p_+)).\eeq
As a result,
\begin{eqnarray*}
&&Ind(p_-)-Ind(p_+) \\
&=&Ind(T_{p_-}V_0,T_{p_-}V_1)-Ind(T_{p_+}V_0,T_{p_+}V_1) \\
&\stackrel{\eqref{e:splitInd},\eqref{e:equalInd}}{=}&Ind(T_{\widetilde{p}_-}N,T_{\widetilde{p}_-}L_1)-Ind(T_{\widetilde{p}_+}N,
T_{\widetilde{p}_+}L_1) \\
&=&Ind(\widetilde{p}_-)-Ind(\widetilde{p}_+)
\end{eqnarray*}
The first and last equality are from definitions.

Finally, by Lemma \ref{l:operatorSplits}, we can split the linearized operator $D_u$ into the vertical component $D^v_u$ and the horizontal component $D^h_u$.
The operator $D^v_u$ is surjective for generic choice of family of $\w$ compatible almost complex structure in $(M,\w)$.
The operator $D^h_u$ can be identified as the linearized operator of constant map from strip to $\mathbb{C}$ with moving boundary condition and of index zero. This is surjective by automatic regularity (See Section $4$ of \cite{SeLefschetzI} and compare to Section $4.3$ of \cite{BC2}).
\end{proof}

This finishes the proof of Lemma \ref{l:regularity}.
\end{proof}

The next lemma addresses the regularity for admissible $\mathbf{J}$ in view of Theorem \ref{t:ImmersedFloerAd}.
%after we establish Theorem \ref{t:ImmersedFloer}.

\begin{lemma}\label{l:regularityAd}
For generic choice of admissible Floer data $\mathbf{J}$, the following curves of interest are regular
\begin{itemize}
\item $\alpha$-marked Floer strips $u$ with fixed boundary condition such that $Im(\pi_2 \circ u) \cap K \neq \emptyset$, and
\item marked Floer strips without branch jumps with fixed boundary condition.
\end{itemize}
\end{lemma}
We do not address regularity for curves with moving boundary condition for admissible $\mathbf{J}_s$.

\begin{proof}
We choose an open set $K'$ with $Cl(K') \subset K$ such that
$\overline{V_0} \cap \overline{V_1} \cap (M \times K) \subset M \times K'$
and the connected components of $\overline{V_1} \backslash (M \times K')$ are in bijection with pinched Lagrangian ends of $\overline{V_1}$.

We first establish the following lemma about a marked Floer strip $u$
having both asymptotic points on the bottlenecks.

\begin{lemma}\label{l:reflection}
Let $\mathbf{J}$ be an admissible Floer data.
If both asymptotic points of a marked Floer strip $u$ with fixed boundary condition are on the same bottleneck,
then $u$ completely lies on the bottleneck.

If, instead, the two asymptotic points lie on different bottlenecks, then $Im(\pi_2 \circ u) \cap K' \neq \emptyset$.
\end{lemma}

\begin{proof}
First consider the case that the two asymptotic points of $u$ lie on the same bottleneck, say $M \times \{z_0\}$.
Let $\mathbf{J}$ be product-like away from $M \times K'$

Since $u(\R \times \{0\}) \subset \overline{V_0}$ and the two asymptotes project to $z_0$,
we have that $\pi_2 \circ u(\R \times \{0\})=z_0$ as in Lemma \ref{l:compactness}.
Moreover, $\pi_2 \circ u$ is holomorphic over $\mathbb{C} \backslash K'$ apart from branch jump marked points
(ie. $\pi_2 \circ u|_{(\pi_2 \circ u)^{-1}(\mathbb{C} \backslash K')}$ is holomorphic).
In particular, $\pi_2 \circ u$ is holomorphic near $\R \times \{0\}$.
By reflection principle, we can extend the domain of $\pi_2 \circ u|_{(\pi_2 \circ u)^{-1}(\mathbb{C} \backslash K')}$ across $\R \times \{0\}$, but this extended holomorphic map sends $\R\times\{0\}$ to $z_0$.  This concludes that $\pi_2 \circ u|_{(\pi_2 \circ u)^{-1}(\mathbb{C} \backslash K')}$, and hence $\pi_2 \circ u$, is a constant.

The second half of the lemma is obvious because the two different asymptotes
guarantee that $Im(u)$ intersects at least two different connected components of $\overline{V_1} \backslash (M \times K')$ and hence
the boundary condition on $u$ implies $u(\R \times \{1\}) \cap (M \times K') \neq \emptyset$, where the almost complex structure has full genericity.
\end{proof}

Now, we continue the proof of Lemma \ref{l:regularityAd}.

If a marked strip $u$ intersects $M \times K$, then by Lemma \ref{l:reflection}, the two asymptotes of $u$
are either on different bottlenecks or one of them is not on a bottleneck.
In either case, $u$ must intersect $M \times K'$, and hence regular for generic admissible $\mathbf{J}$.

We now address the regularity of a marked strip $u$ that has no branch jump.
It suffices to consider $u$ that does not intersect $M \times K$, in other words, $u$ that completely lie on a bottleneck.
The regularity of these $u$ for generic admissible $\mathbf{J}$ can be argued as in the proof of Lemma \ref{l:regularity}.
\end{proof}

\subsubsection{Positivity assumption on cobordism}

\begin{lemma}\label{l:preservingPositivityAssumption}
The Lagrangian immersions $\iota_{V}$ and $\iota_{N,\gamma}$ satisfy Assumption {\bf (A)} if $\iota_N, \iota_{L_i},\iota_{L_j'}$
and all immersed points of $\iota_{V}$ in $M \times K$ satisfy Assumption {\bf (A)}.
\end{lemma}

\begin{proof}

We start by computing the energy on the cobordisms.  For any immersed point $(p_-,p_+)\in R_{N\times\gamma}$,
we write it as  $(\wt p_-, \wt p_+) \times z$ for the corresponding $(\wt p_-, \wt p_+)\in R_N$ and $z \in \C$.
By taking a split primitive one form $\alpha\oplus\alpha_\C$ and a curve $\wt l \subset N$ with ends at $\wt p_-$ and $\wt p_+$, one has

\beq\label{e:energy}E(\widetilde p_-,\widetilde p_+)= \int_{\wt l} \iota_N^*\alpha
= \int_{l} \iota_{V_1}^*(\alpha+\alpha_{\mathbb{C}})= E(p_-,p_+).\eeq

where $l= \wt l \times z$.
This computation also applies to immersed points on bottlenecks of $V$.

We now consider the index part.   For the $N\times\gamma$ case, a direct calculation gives $Ind(\widetilde p_-,\widetilde p_+)+1=Ind(p_-,p_+)$ which implies the conclusion for $N\times\gamma$.

For an immersed Lagrangian cobordism with bottleneck $\iota_V$, let $(\widetilde{p}_-,\widetilde{p}_+) \in R_{L_1}$ be a branch jump type  at the bottleneck $L_1\subset M\times\{z_1\}$, and denote $(p_-,p_+)=(\wt p_-,\wt p_+)\times\{z_1\}\subset R_V$.
Let $T_{p_{\pm}}=(d\iota_{V})(T_{p_{\pm}}V)$.
We can write $T_{p_{\pm}}=T^v_1(p_{\pm}) \oplus T^h_1(p_{\pm})$, where $T^v_1(p_{\pm}),T^h_1(p_{\pm})$ are the vertical part
and the horizontal part in $T_{p_\pm}V$ explained in Lemma \ref{l:indexUnchange}.
Since the quadratic complex volume form defining gradings splits, the grading of $T_{p_\pm}$ and $T^v_1(p_\pm)$ determines a grading on $T^h_1(p_{\pm})$, denoted as $\theta^h_{\pm}$.  We claim that $Ind((T^h_1(p_{-}),\theta^h_-),(T^h_1(p_{+}),\theta^h_+))=0$ or $1$.

Take a path $l(p_-,p_+)\subset L_1\times\{z_1\}$ with ends at $p_\pm$.  $l(p_-,p_+)$ induces a path of graded Lagrangian subspace as above from the grading of $V$, which decomposes similarly as before.  The horizontal component $l^h(p_-,p_+)$, as a path in $Gr_{lag}^\infty(\C)$, is bounded within the conic region on the plane specified by the bottleneck, which has to have index $0$ or $1$.  This index coincides with $Ind((T^h_1(p_{-}),\theta^h_-),(T^h_1(p_{+}),\theta^h_+))$.

Therefore, we have
\beq\label{e:index}Ind(T_{p_-},T_{p_+})=Ind(\widetilde{p}_-,\widetilde{p}_+)+\epsilon(p_-,p_+),\quad \epsilon(p_-,p_+)=0\text{ or }1.\eeq

Combining this with \eqref{e:energy} the result follows.
\end{proof}

\subsubsection{Proof of Theorem \ref{t:ImmersedFloer}}

We now may prove the well-definedness of $HF(\iota_{N,\gamma},\iota_V)$ and the invariance of admissible isotopy.  The usual proof sketched in Section \ref{s:ILFreview} mostly applies: we study the boundary moduli space of holomorphic strips with no branch jumps, $\widetilde{\mathcal{M}}(p_{-},p_+,\alpha=\emptyset,\mathbf{J})$, for index gap $Ind(p_-)-Ind(p_+)=2$; and $\widetilde{\mathcal{M}}^m(p_{-},p_+,\alpha,\mathbf{J}_s)$, $p_- \in\iota_{N,1}\cap\iota_V$, $p_+ \in\iota_{N,0}\cap\iota_V$ for an admissible isotopy $\iota_{N,s}$ and $Ind(p_-)-Ind(p_+)=1$ (for the chain homotopy, consider a family version of latter moduli space with zero index gap).  What we need to reconsider in our situation is the compactification for these moduli spaces, and the absence of fishtails.

The compactness was essentially proved in \ref{s:compactness}, because the only part to address is that the pinched ends are usually not cylindrical in an ordinary sense.  But we showed in Lemma \ref{l:compactness} that any holomorphic curve in the above moduli spaces has to project into $Cl(K)$ in the semi-admissible setting, hence reducing the problem to ordinary Gromov compactness theorem.

For the fishtails, we have the following lemma for generic semi-admissible data:

\begin{lemma}\label{l:nonregularControl}
Let $u$ be a stable curve on the boundary of a moduli space of the first two kinds among curves of interests (Defimition \ref{d:curveInt}).
Suppose $u$ consists of a single marked strip $u_0$ (possibly non-regular) with $k$ ($k \ge 1$) branch jumps,
together with fishtail bubble trees $\overline v_j$ rooted at the branch jumps of $u_0$.
Then the index of $u$ is at least three.

\end{lemma}

\begin{proof}
We first assume $u_0$ is non-regular.
By Lemma \ref{l:regularity}, a non-regular marked strip from $p_-$ to $p_+$ must lie on a bottleneck.
Since the almost complex structure is product like at the bottleneck and we choose it to be regular on fibers, this marked strip is regular when viewed as a strip in $(M,\w)$ (from $\wt p_-$ to $\wt p_+$).
Therefore, calculated on $M$,
$Ind(\wt p_-)-Ind(\wt p_+)-\sum_{i=1}^k Ind(\wt\alpha(i))+k-1 \ge 0$, where $\wt\alpha(i)$ are the branch jump marked points \cite[Corollary 5.6]{AB14}.  This implies

\begin{eqnarray*}
&& Ind(p_-)-Ind(p_+)\\
&=& Ind(\wt p_-)-Ind(\wt p_+)\quad (\text{Lemma} \ref{l:indexUnchange})\\
&\ge& \sum(Ind(\wt\alpha(i))-1)+1\\
&\ge& 3\end{eqnarray*}

The sum of indices of $u_0$ and $\overline v_j$ equals $Ind(p_-)-Ind(p_+)$ by definition, which is at least three in turn.
The proof when $u$ is regular is a similar index computation and easier.
\end{proof}

For the definition and invariance of immersed Floer theory involving cobordisms, we do not need to consider the moduli space with index gap $3$ or higher.  This implies we recover the absence of fishtails for all relevant cases, hence concluding Theorem \ref{t:ImmersedFloer}.

\begin{proof}[Proof of Theorem \ref{t:ImmersedFloerAd}]
We first make one observation.
From the proof for semi-admissible case,  Lemma \ref{l:regularityAd} already guarantees the Floer cohomology is well-defined (compare Lemma \ref{l:regularity}) because the well-definedness does not involve any moving boundary condition.
%Since an admissible Floer data is a semi-admissible Floer %data, the well-definedness of the Floer cohomology in %Theorem \ref{t:ImmersedFloerAd}
%is a direct consequence of Theorem \ref{t:ImmersedFloer}
%by Lemma \ref{l:regularityAd}.

To show that the Lagrangian Floer cohomology is not only well-defined but also invariant
under admissible horizontal isotopy, a direct attempt would be to address the regularity for marked strips with moving boundary conditions
for Lemma \ref{l:regularityAd} as in Lemma \ref{l:regularity}.
However, there is no obvious reason that it is true and we bypass this issue as follows.

If $\iota_{N,s}$ is an admissible horizontal isotopy from $\iota_{N,0}$ to $\iota_{N,1}$ and $\mathbf{J}_0,\mathbf{J}_1$ are
generic admissible Floer data that computes the Floer cohomology for $HF(\iota_{N,0},\iota_{V})$ and $HF(\iota_{N,1},\iota_{V})$,
 then there is a generic choice of {\it semi-admissible} $\mathbf{J}_s$ connecting $\mathbf{J}_0$ and $\mathbf{J}_1$ inducing a chain map (which is also a chain homotopy)
 from $CF(\iota_{N,0},\iota_{V})$ to $CF(\iota_{N,1},\iota_{V})$, by Theorem \ref{t:ImmersedFloer}.
The invariance under admissible horizontal isotopy is hence concluded.
\end{proof}

\subsubsection{Biran-Cornea's cobordism exact sequence in the immersed setting}\label{s:LES}

%We are now ready to recover Biran-Cornea's homological cone %theorem in the exact immersed setting.
\begin{thm}\label{t:ImmersedFloerLES}
Let $\iota_{V}$ be an exact immersed Lagrangian cobordism with bottlenecks from
$(\iota_{L_1},\dots,\iota_{L_k})$ to $(\iota_{L'_{k'}},\dots,\iota_{L'_{1}})$ such that
$\iota_{L_i}$ and all interior immersed points of $\iota(V)$ satisfy Assumption {\bf(A)} for all $i$,
then for any cleanly immersed Lagrangian $\iota_N$ which is not a covering and satisfies the positivity assumption,
we have
$$Cone (CF(\iota_N,\iota_{L_1}) \to \dots \to CF(\iota_N,\iota_{L_k})) \cong Cone(CF(\iota_N,\iota_{L'_{k'}})\rightarrow\cdots\rightarrow CF(\iota_N,\iota_{L'_{1}})) $$

\end{thm}

\begin{proof}
The main idea of the proof follows from \cite[Theorem 2.2.1]{BC13}.
By \cite{AB14}, $HF(\iota_N,\iota_{L_1})$ is invariant under Hamiltonian isotopy in $M$ as long as {\bf (A)} is satisfied.
Therefore, we can assume
$\iota_N$ intersect $\iota_{L_i} $ transversally at smooth points for all $i$, possibly after a Hamiltonian perturbation of $\iota_N$.  Similar restrictions applies also to $L_j'$.

Let $\widehat{\gamma}$ be an \textit{admissible} testing curve passing through only bottlenecks of $L_j'$ and $\widecheck\gamma$ passing through only bottlenecks of $L_i$ (eg.  $\widehat{\gamma}=\gamma_0$ and $\widecheck\gamma=\gamma_1$ in Figure \ref{fig:testingCurve}).
Then $\iota_{N,\widehat\gamma}$ and $\iota_{N,\widecheck\gamma}$
are admissible testing Lagrangian with clean self-intersection that are admissible horizontal isotopic.
Also , we have
$\iota_{N,\wh\gamma} \cap \overline{V}= \overline{N} \cap (\overline L_1' \cup \dots \cup \overline L_{k'}')$
and $\iota_{N,\wc\gamma} \cap \overline{V}= \overline{N} \cap (\overline{L_1} \cup \dots \cup \overline{L_k})$.

By Theorem \ref{t:ImmersedFloerAd}, for a generic {\textit admissible} Floer data
$(\mathbf{H}=0,\mathbf{J})$, the Floer cohomology
$CF(\iota_{N,\wh\gamma},\iota_{V})$ and $CF(\iota_{N, \widecheck{\gamma}},\iota_{V})$ are well-defined
and chain homotopic, which we will identify the two cochain complexes with an iterated mapping cone.  It suffices to focus on $\iota_{N,\wc\gamma}$ below.

Note that all intersections $\iota_{N,\wc\gamma}\cap\ov V$ are contained in the bottlenecks and we consider Floer strips between them. %$\overline{\hat{V}_N} \cap \overline{V_2}$ completely lies on the %bottleneck.
Note also that for any holomorphic strip $u$, from Lemma \ref{l:compactness} one has $\partial_s(\pi_\C\circ u)(s_0,0)$ points upward or vanishes in the complex plane, as long as $\pi_\C\circ u(s_0,0)\in\partial K$.

This simple fact has two consequences.  Firstly, if a strip $u$ without branch jumps does not have $\pi_\C\circ u(\R \times \{0\})$ being a constant,
the two asymptotes of $u$ are on different bottlenecks and $u$
contributes to the differential from the bottleneck of $L_{i_0}$ to $L_{i_1}$ for some $i_0<i_1$.
Consequently, if $\pi_\C\circ u(\R \times \{0\})$ is a constant, then $u$ completely lies on a bottleneck by Lemma \ref{l:reflection}.
and contributes to the differential from the bottleneck of $L_{i}$ to itself for some $i$.
This gives a filtration of $CF(\iota_{N,\wc\gamma},\iota_V)$.

Secondly, for two points $p_{\pm}\in \iota_{N,\wc\gamma}\cap \ov V$ that are on the bottleneck of $L_{i}$,
we have that $Ind(p_-)-Ind(p_+)$ is the same
no matter viewing it as an intersection point of $\iota_{N,\wc\gamma}$ and $\ov V$,
or as that of $\iota_N$ and  $\iota_{L_i}$ by Lemma \ref{l:indexUnchange}.
Moreover, there is an obvious bijection between Floer strips contributing to $CF(\iota_{N,\wc\gamma},\ov V)$ with both
asymptotes on $\iota_{L_i}$ and
Floer strips contributing to $CF(\iota_{N},\iota_{L_{i}})$.
Furthermore, these strips are all regular, as argued in Lemma \ref{l:indexUnchange}.
As a result, the differential $\check{\partial}$ of the chain complex $CF(\iota_{N,\wc\gamma},\ov V)$
is a lower triangular matrix with the diagonal entries being $\partial_i$, the differential of $CF(\iota_N,\iota_{L_i})$.
The result follows.

\end{proof}

\subsection{Immersed construction}\label{s:ImmersedConstruction}

In this subsection, we let $L_0=\mathbb{CP}^{\frac{m}{2}} \subset T^*L_0$ be equipped with a grading.
We want to examine the Dehn twist long exact sequence for the complex projective space
when a Lagrangian $L_1$ intersects $L_0$ transversally at a point.
We interpret the cone relation predicted by Huybrechts and Thomas geometrically using immersed Lagrangians cobordism.
Precisely, we want to show the existence of the following two immersed cobordisms

\begin{enumerate}[(1)]
\item from $L_0[-2]$ and $L_0$ to an immersed sphere $S_{\looparrowright}$,
\item from $S_{\looparrowright}$ and $L_1$ to $\tau_{L_0}(L_1)$.
\end{enumerate}

We start by constructing the immersed sphere $S_{\looparrowright}$ associated to a Lagrangian $\CP^{\frac{m}{2}}$ in a Weinstein neighborhood.
\begin{lemma}\label{l:immersedSphere}
Let $x_0\in L_0$ and $D=\{x \in L_0| dist(x_0,x)=\pi \}= \mathbb{CP}^{\frac{m}{2}-1}$.
Let $L_{-1}$ be the graph of differential of $h(\cdot)=-dist^2(\cdot,x_0)$ in $T^*L_0$ equipped with the induced grading from $L_0$.

Then $S_{\looparrowright}=L_{-1}[-1]\#_{D}L_0$ is a graded immersed Lagrangian sphere with $R_{S_\looparrowright}=\{(\qno,q_0),(q_0,\qno)\}$, where $\qno$ and $q_0$ are both $x_0$ as a point on $T^*L_0$, while $\qno$ and $q_0$ comes from the branch $L_{-1}$ and $L_0$, respectively.  Moreover,
\begin{itemize}
\item $Ind(q_{-1},q_0)=-1$ (resp. $Ind(q_0,\qno)=m+1$), and
\item $E(\qno,q_0)<0$ (resp. $E(q_0,\qno)>0$)
\end{itemize}
\end{lemma}

\begin{proof}
It is clear that $L_{-1}$ intersect with $L_0$ transversally at $x_0$ and cleanly at $D$.
The calculation of index is similar to Example \ref{e:MorseBottIndex}.
Precisely, since $x_0$ is maximum of $h(\cdot)$, we have $Ind(L_{0}|_{x_0},L_{-1}|_{x_0})=m$ and $Ind(L_{-1}|_{x_0},L_{0}|_{x_0})=0$.
On the other hand, $D$ is the minimum of $h(\cdot)$ and $D$ is of dimension $m-2$ so we have
$Ind(L_{0}|_D,L_{-1}|_D)=m-2$ and hence $Ind(L_{-1}|_D,L_{0}|_D)=m$.
Note that, $Ind(L_{-1}[-1]|_D,L_{0}|_D)=m-1$ so Lemma \ref{l:cleanIndex1} implies
that $S_{\looparrowright}$ is graded.
It is clear that it is an immersed sphere with the only immersed point at $x_0$.
The first bullet is also clear by the index computation before surgery, since the grading remains unchanged outside the surgery site.

For the energy, we consider the canonical one form $\alpha$ on $T^*L_0$.
Since $L_{-1}$ is the graph of $dh$, $h$ is a primitive function of $\alpha|_{L_{-1}}$. Note that $h$ attains its maximum at $x_0$ and minimum at $D$;
also by Lemma \ref{l:localEnergyClean}, the value of the primitive function of  $L_{-1}[-1]\#_{D}L_0$
decreases along the handle and the primitive function of $L_0$ is a constant.  Therefore, when we apply the gluing of primitive functions as in Lemma \ref{l:exactSimCob}, we have $E(\qno,q_0)<0$.
Finally, we have $E(q_0,\qno)=-E(\qno,q_0)>0$.
\end{proof}

We continue to use $q_0$ (resp. $q_{-1}$) to denote $x_0$ regarded as on the $L_0$ (resp. $L_{-1}$) branch of $S_{\looparrowright}$.

\begin{lemma}\label{l:SnCpn=Dehn}
Let $x_0 \in L_0$, $F_{x_0}$ be the cotangent fiber at $x_0$ and $S_{\looparrowright}$ as in Lemma \ref{l:immersedSphere}.
We equip a grading to $F_{x_0}$ such that $Ind(L_0|_{x_0},F_{x_0}|_{x_0})=0$.
Then there is a graded Hamiltonian isotopy from $S_{\looparrowright}[1]\#_{q_0}F_{x_0}$ to $\tau_{L_0}(F_{x_0})$.
\end{lemma}

\begin{proof}
By Lemma \ref{c:ptIndex}, we can perform a graded surgery $S_{\looparrowright}[1]\#^{\nu_\lambda}F_{x_0}$
for some small $\lambda$.  To understand $S_\looparrowright[1]\#^{\nu_\lambda} F_{x_0}$ we use Lemma \ref{l:rotLag} and the remark following it.  One considers a geodesic $\gamma(t)$ on $\CP^{m/2}$ and takes a smooth lift on $S_\looparrowright[1]\#^{\nu_\lambda} F_{x_0}$, starting from $F_{x_0}$.  This lift written as $(\gamma(t),f(t)\gamma'(t))$ clearly has no points with $f(t_0)=-f(2\pi-t_0)$ (in fact the only immersed point at $x_0$ was resolved by the surgery).  Therefore from Lemma \ref{l:rotLag}
 and Remark \ref{r:rotLag}, we concluded the proof.

\end{proof}

We consider a pinched perturbation of $\iota_{S_{\looparrowright}} \times \R$ which will be used to construct immersed cobordism with bottleneck associated to the surgery in Lemma \ref{l:SnCpn=Dehn}.
%where the immersed Lagrangian sphere %$\iota_{S_{\looparrowright}}:S^m \to (T^*L_0,\w)$ is the
%$S_{\looparrowright}$ above.
Let $q_0,\qno \in S^m$ be the points defined above.  We define a bottleneck perturbation $\iota_{S^\Delta}$ by the following auxiliary data.

\begin{itemize}
\item Let $\widetilde{g}_R: \mathbb{R} \to \mathbb{R}$ be a Morse function with a unique critical point which is a maximum at $x=-2$ that is $C^1$ small.
\item Let $\widetilde{g}_S: S^m \to \mathbb{R}$ be a non-negative function such that $\widetilde{g}_S=0$ outside the $\delta$ neighborhood of $\qno$
and $\widetilde{g}_S=1$ inside the $\delta/2$ neighborhood of $\qno$.
\item Consider a Weinstein neighborhood $N_1$ of $S^m \subset T^*S^m$ and $N_2$ of $\R\subset\C$.
Extend $\widetilde{g}_R$ and $\widetilde{g}_S$ to a function
$\widetilde{g}_S:N_1\rightarrow \mathbb{R}$ and $\widetilde{g}_R:N_2\to\mathbb{R}$, respectively, by pulling back from projection.

\item Extend the domain of $\iota_{S_{\looparrowright}}$ to $N_1$ to define symplectic immersion $\iota_{N_1}: N_1 \to T^*L_0$.
Also let $\iota_{N_2}: N_2 \to \mathbb{C}$ be the canonical symplectic embedding by the convention $z=q-ip$.

\end{itemize}

We then use the map
$\iota_{N_1} \times \iota_{N_2}: N_1 \times N_2 \to M \times \mathbb{C}$ and define a time-one flow of $S^n \times \mathbb{R}$ by the product Hamiltonian $\widetilde{g}_R \widetilde{g}_S$ be
$S^\Delta=\phi^{\widetilde{g}_R\widetilde{g}_S}(S^m \times \mathbb{R}) \subset N_1 \times N_2$ and
define $\iota_{S^\Delta}=\iota_{N_1} \times \iota_{N_2}|_{S^\Delta}$.

\begin{lemma}\label{l:PerturbationImmersedEnd}
If the $C^1$ norm of $\widetilde{g}_R$ is sufficiently small,
$\iota_{S^\Delta} $ is a pinched Lagrangian immersion with a single transversal immersed point.
One of the two branch jumps has index $m+1$ with positive energy,
while the other has index $0$ with negative energy.
\end{lemma}

\begin{proof}
In $N_1 \times N_2$, $S^\Delta$ equals the graph of $d(-\widetilde{g}_R\widetilde{g}_S)=-(\widetilde{g}_Rd\widetilde{g}_S+\widetilde{g}_Sd\widetilde{g}_R)$.
When $d\widetilde{g}_R(x) > 0$, or equivalently $x <-2$ , the fact that $\widetilde{g}_S \ge 0$
implies that $\iota_{S^\Delta}$ has non-positive $p$-coordinate in $N_2$ and hence non-negative $y$-coordinate in $(\mathbb{C},dx \wedge dy)$ factor.
In particular, $\phi^{\widetilde{g}_R\widetilde{g}_S}(\qno,x)$ has positive $y$ coordinate for any $x < -2$ while $\phi^{\widetilde{g}_R\widetilde{g}_S}(q_0,x)$ has zero $y$-coordinate
because $\widetilde{g}_S(\qno)>0$ and $\widetilde{g}_S(q_0)=0$.
This means $\iota_{S^\Delta}(\qno,x)
\neq \iota_{S^\Delta}(q_0,x)$ for all $x < -2$.
When the $C^0$ norm of $\widetilde{g}_R$ is sufficiently small, the vertical perturbation ($N_1$ direction) given by $-\widetilde{g}_Rd\widetilde{g}_S$ is insignificant
and we have $\iota_{S^\Delta}|_{\{x < -2\}}$ is a Lagrangian embedding.
The same is true when $x > -2$ (See Figure \ref{fig:perturbation}).
Finally, it is obvious that there is exactly one transversal immersed point of $\iota_{S^\Delta} $ when $x = -2$ since $g''_R(-2)\neq0$,
coming from the transversal immersed point of $\iota_{S_{\looparrowright}}$.

\begin{figure}[h]
\centering
\includegraphics[scale=1.2]{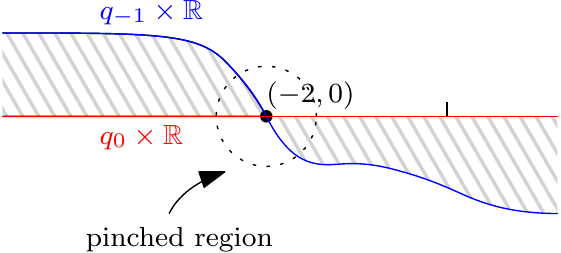}
\caption{A pinched perturbation $\iota_{S^\Delta}$.}
\label{fig:PinPerturbation}
\end{figure}

\begin{figure}[h]
\centering
\includegraphics[scale=1.2]{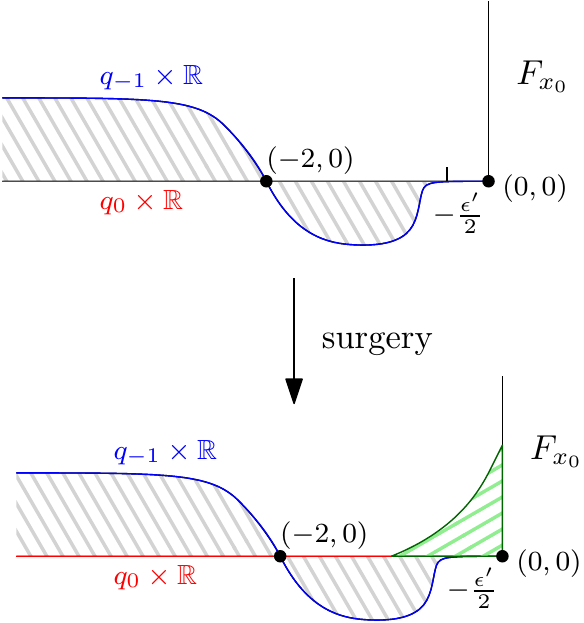}
\caption{Projection of half of $\iota_{S^\Delta}\# (F_{x_0} \times i\R)$: part of $q_0\times\R$ is removed to glue the green handle, which removes all immersed points from the cobordism.}
\label{fig:perturbation}
\end{figure}

To calculate the energy, we may assume $\widetilde{g}_R(-2)=0$ so that the immersed point is
$(\iota_{S_{\looparrowright}}(\qno),-2)=(\iota_{S_{\looparrowright}}(q_0),-2)$.
Consider the product immersion $(S^m \times \mathbb{R},\iota_{S_{\looparrowright}} \times Id)$ with the set of branch jump types
$$\{ ((\qno,t),(q_0,t)), ((q_0,t),(\qno,t)) | t \in \mathbb{R} \}$$
 where $E((\qno,t),(q_0,t))<0$. The $C^0$ control of $\widetilde{g}_R$ guarantee that $E((\qno,-2),(q_0,-2))$ remains negative after perturbation.

 For the index part, the Lagrangian tangent plane at $(\qno,-2)$ and $(q_0,-2)$ can be simultaneously decomposed into $M$ and $\C$ components.
 For the $\C$-component, the index is nothing but $Ind(\pi_\C(\phi^{\widetilde{g}_R\widetilde{g}_S}(\qno\times\R)),\pi_\C(q_0\times\R))=1$.
 Combining this with the computation on fiber in Lemma \ref{l:immersedSphere}, it follows that $Ind(\qno,q_0)=0$.  The case of $Ind(q_0,\qno)=m+1$ follows similarly.
\end{proof}

\begin{lemma}\label{l:SnFiberDehnCob}
Let $L_0, x_0$ be as above, and $L_1 \subset (M,\w)$ be a Lagrangian that intersects $L_0$ transversally $x_0$.
Then there is an immersed Lagrangian sphere $\iota_{S_{\looparrowright}}:S^m \to (M^{2m},\w)$ and an
immersed Lagrangian cobordism with bottleneck from $\iota_{S_{\looparrowright}}$ and $L_1$ to $\tau_{L_0}(L_1)$.
\end{lemma}

\begin{proof}

Without loss of generality, we can assume $M=T^*L_0$ and $L_1=F_{x_0}$.
We proceed in a way similar to Lemma \ref{l:cob} using $S_{\looparrowright}\#_{q_0}F_{x_0}$ obtained from Lemma \ref{l:SnCpn=Dehn}.

We first illustrate the idea by constructing a ``naive cobordism".
Consider $\iota_{S_{\looparrowright}} \times \R$ and $F_{x_0} \times i\R$ in $T^*L_0 \times \mathbb{C}$.
Then we perform Lagrangian surgery at $q_0 \times (0,0)$ supported in an $\epsilon'$-neighborhood of $q_0$.
The resulting Lagrangian  whose fiber at $(0,0)$ is $S_{\looparrowright}[1]\#F_{x_0}$ and hence embedded
(Precisely, the surgery is done on the branch $L_0 \times \R$ of $\iota_{S_{\looparrowright}} \times \R$).
We can cut this cobordism in half and do a Hamiltonian perturbation and extend the cylindrical end as in Lemma \ref{l:cob}.
This Lagrangian immersion has all its immersed points on the end $\iota_{S_{\looparrowright}} \times ((-\infty,-\epsilon') \times \{0\})$
because when $\epsilon$ is small, the handle added does not produce new immersed points.
However, it is not yet an immersed Lagrangian cobordism with bottleneck (the self-intersection is not clean).
To make it an immersed Lagrangian cobordism with bottleneck, we need to perturb this end by Lemma \ref{l:PerturbationImmersedEnd} before the surgery.

The actual construction of the cobordism, detouring a bit, starts from a perturbed copy of $\iota_{S_\looparrowright}\times\R$.  We choose a smooth $g_R$ so that $g'_R(x)>0$ for $x \in (-\infty,-2)$, $g'_R(x)<0$ for $x \in (-2,-\frac{\epsilon'}{2})$  and
$g_R(x)=0$ for $x\ge -\epsilon'/2$ and $g_S=\wt g_S$ as in Lemma \ref{l:PerturbationImmersedEnd}.
This defines a pinched Lagrangian $\iota_{S^\Delta}$ as in the first picture in Figure \ref{fig:perturbation}.  Notice that $g_R=0$ for $x\ge -\epsilon'/2$ and $g_S=0$ near $q_0$, which implies that $\iota_{S^\Delta}$ coincide with
$\iota_{S_{\looparrowright}} \times \R$ for $x \ge -\epsilon'/2$ near $q_0 \times \R$.
This allows us to perform Lagrangian surgery from $\iota_{S^\Delta}$ to $F_{x_0} \times i\R$
at $q_0 \times (0,0)$ whose fiber at $(0,0)$ being $S_{\looparrowright}[1]\#F_{x_0}$.

The resulting surgery coincides with the naive cobordism near $M\times(0,0)$, hence can be extended to an actual cobordism as above.  It is a cobordism with a single bottleneck at $x=-2$ by the fact that $S^\Delta$ is pinched.

\end{proof}

\begin{corr}\label{c:immersedSphereCone}
Let $L_0,L_1 \subset (M,\w)$ be as above.
For any clean immersed Lagrangian $\iota_N:N \to (M,\w)$ satisfying the Assumption {\bf (A)} which is not a covering,
we have the long exact sequence
$$\dots \to HF^*(\iota_N,\iota_{S_{\looparrowright}}) \to HF^*(\iota_N, L_1) \to HF^*(\iota_N,\tau_{L_0}(L_1)) \to \dots$$
\end{corr}

\begin{proof}
By Lemma \ref{l:immersedSphere}, $\iota_{S_{\looparrowright}}$ satisfies the Assumption {\bf (A)}.
Therefore, the result follows by Theorem \ref{t:ImmersedFloerLES} and Lemma \ref{l:SnFiberDehnCob}.

\end{proof}

We also want an immersed Lagrangian cobordism from $L_0[-2]$ and $L_0$ to $\iota_{S_{\looparrowright}}$.

\begin{lemma}\label{l:CpnCpnImmersedSphere}
There is a graded immersed Lagrangian cobordism with bottleneck from $L_0'[-2]$ and $L_0$ to $\iota_{S_{\looparrowright}}$.
\end{lemma}

\begin{proof}
Recall from the construction of a simple cobordism (Lemma \ref{l:cob}), we constructed a surgery of $L_0'[-2]\times\R$ and $L\times i\R$ using the geodesic flow on the product.  This surgery clearly has a bottleneck at $M\times (0,0)$ (see the left of Figure \ref{fig:SimCob}).  Therefore, by appropriate Hamiltonian isotopy on the $\C$-factor alone, one may adjust the resulting surgery Lagrangian submanifold into an actual immersed Lagrangian cobordism as desired (in this case the part of surgery Lagrangian in the fourth quadrant will be isotoped to the part outside $K$).
\end{proof}

\begin{corr}\label{c:CpnCpnSphere}
For any clean immersed Lagrangian $\iota_N:N \to (T^*L_0,\w)$ satisfying the Assumption {\bf (A)} which is not a covering,
we have the long exact sequence
$$\dots \to HF^*(\iota_N, L_0[-2]) \to HF^*(\iota_N,L_0) \to HF^*(\iota_N,\iota_{S_{\looparrowright}}) \to \dots$$
\end{corr}

\section{Computations of connecting maps}\label{s:connectingMaps}

While the theory of Lagrangian cobordism provides a convenient way of proving isomorphism of mapping cones, a general difficulty is to determine the connecting maps involved in these cones.
In the case of a simple cobordism such that $L_1 \pitchfork L_2=\{p\}$, one may adapt the analysis of \cite[Chapter 10]{FOOO_Book} to find the actual count of the connecting map.
In this section, we explain how to ``compute" some connecting maps through the following simple algebraic fact.

\begin{lemma}\label{l:rankone}
Given chain complexes $A,B$ over a field $\mathbb{K}$ and $c,c' \in hom^0(A,B)$ which are closed.
Assume that $0 \neq t \in \mathbb{K}$ and $[c]=t[c']$. Then $cone(c)$ is quasi-isomorphic to $cone(c')$.
\end{lemma}

\begin{proof}
This is a straightforward verification by sending $(a,b) \in A[1] \oplus B$ to $(a,tb+\eta(a))$, where $\eta$ is a chain homotopy between $c$ and $tc'$.
\end{proof}

Lemma \ref{l:rankone} can be upgraded to a categorical level, for example, using Yoneda lemma.
This means the quasi-isomorphism type of a non-trivial mapping cone is determined by the choice in $\mathbb{P}Hom^0(A,B)$.
Hence, it suffices to compute the connecting morphisms up to a rescaling factor
when only the quasi-isomorphism type of the cone is concerned.
In particular, when $rank(Hom^0(A,B))=1$, the cone between $A$ and $B$ can have only one quasi-isomorphism type that is not the direct sum.  The following perturbation lemma will be useful for excluding the direct sum.

\begin{figure}[h]
\centering
\includegraphics[scale=1]{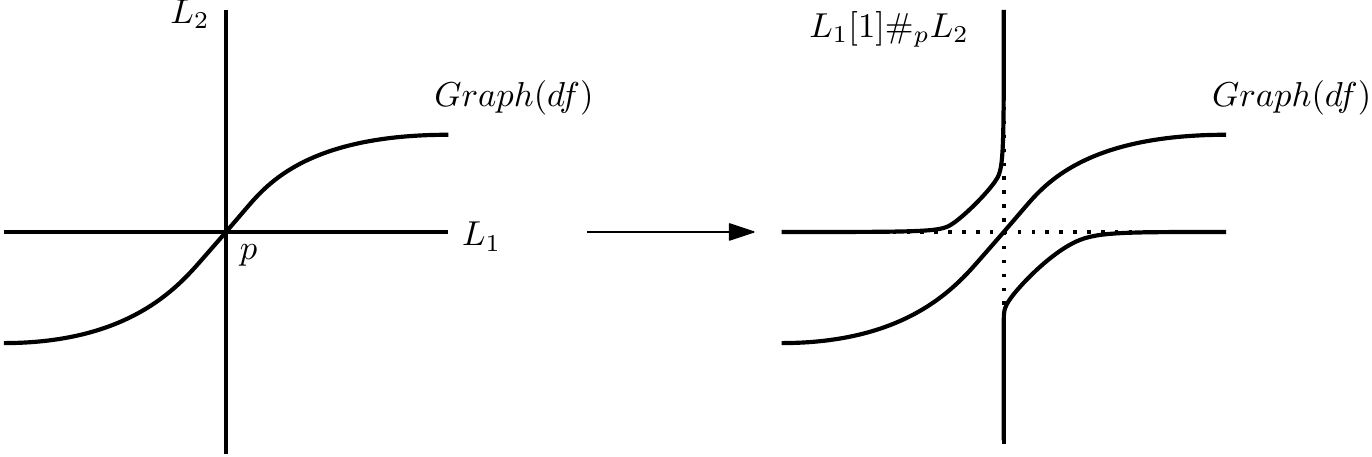}
\caption{Resolving the degree zero intersection by surgery}
\label{fig:throughPert}
\end{figure}

\blem\label{l:throughPert} Let $L_1$, $L_2\subset M$ be a pair of $\Z$-graded exact Lagrangian submanifolds.  Assume $L_1 \cap L_2=D$ with index $Ind(L_1|_D,L_2|_D)=\dim(D)=k$ and the intersection is clean.
Let $f: L_1 \to \R$ (resp. $f: L_2 \to \R$) be a Morse-Bott function which attains maximum (resp. minimum) at $D$.
Then the graph of $df$ as a perturbation $\wt L_1$ of $L_1$ (resp. $\wt L_2$ of $L_2$) in a Weinstein neighborhood satisfies

$$\wt L_1\pitchfork(L_1[1]\#_D L_2)=(\wt L_1\pitchfork L_1[1])\backslash\{D\},$$
and respectively,
$$\wt L_2\pitchfork(L_1[1]\#_D L_2)=(\wt L_2\pitchfork L_2)\backslash\{D\},$$

\noindent are correspondence of intersections preserving degrees.
\elem

\bpf  Pick a Weinstein neighborhood $W$ of $L_1$ such that $L_2$ can be identified as a conormal bundle (Proposition \ref{p:Poz99}).
Let $\wt L_1$ be the graph of $df$ and identify $\wt L_1$
as a Lagrangian in $W$ and hence in $M$.
%We can pick a chart of $L_1$ such that $D$ is identified as $\R^k$ in the chart $\gamma(t)=(x_1,\dots,x_k,tx_{k+1},\dots,tx_n)$ are normalized geodesic for any $(x_1,\dots,x_n)$ such that $\sum_{j=k+1}^n x_j^2=1$.
Pick a Darboux chart $U \subset W$ centered at a point $p \in D$ such that $L_1$ is identified with $\R^n$ and $L_2$ is identified with $N^*_{\R^k}$.
The $U$ can be chosen such that $f=c\sum\limits_{i=1}^{n-k} x_i^2$ in the Darboux chart $U$
for some small negative constant $c$ and hence $\R^k$ is the only critical submanifold of $f$ in $U$.

Let $L_3=L_1[1] \#_D L_2$, where the surgery takes place in $U$. %(Corollary \ref{c:SurgeryAtPoint}).
%Here we take the surgery description from Lemma \ref{lemma: Lagrangian hanlde is Lagrangian}. The handle $H_\gamma=(a(t)\overrightarrow{x},b(t)\overrightarrow{x})$ where $a(t) \ge 0$ and $b(t) \le 0$, while
We have $Graph(df)=\{(\overrightarrow x,2c\overrightarrow x)|\overrightarrow x \in \R^n\}$ in $T^*\R^n=U$.
On the other hand, the flow handle is given by $H^D=\{(exp(\overrightarrow v),\overrightarrow v)|\overrightarrow v \in N^*_{\R^k}\}$, where $exp$ denotes the exponential map.
Since $c<0$, one sees that the two Lagrangians do not intersect in this Darboux chart $U$ by checks on signs.
Since $p \in D$ is arbitrary, the flow handle does not intersect $Graph(df)$.

The perturbation $\wt L_2$ is constructed similarly, except $f$ is taken to have a critical minimum submanifold along $D$ on $L_2$.  We leave the details to the reader.

\epf

We exploit consequences of this simple fact.  In the rest of this section all Lagrangians will be assumed to be $\Z$-graded and exact.

\begin{corr}\label{c:SurgeryTriangle}(Surgery exact triangle)
Let $L_1,L_2$ be graded exact closed embedded Lagrangians.
Assume $L_1 \cap L_2=D$ is connected such that $Ind(L_1|_D,L_2|_D)=\dim(D)=k$ and the intersection is clean.
Let $L_3=L_1[1] \#_D L_2$.
Suppose also that there is a Morse-Bott function $f : L_1 \to \R$ (or $f : L_2 \to \R$) such that
$f$ attains local maximum (resp. minimum) exactly at $D$ (ie. no points other than $D$ attains a local maximum).
Then there is an exact triangle $$L_1 \xrightarrow{[D]} L_2 \to L_3 \to L_1[1] $$
\end{corr}

\begin{proof}
When $D$ is a point, the exact triangle is known to Fukaya-Oh-Ohta-Ono \cite{FOOO_Book} in its cohomological version, and is a direct consequence of Biran-Cornea's cobordism theory in the categorical version.
We focus on the derivation of the connecting map $c_s:L_1\rightarrow L_2$.

We assume $f : L_1 \to \R$ attains local maximum exactly as $D$. The case for a Morse-Bott function on $L_2$ is similar.
Since there is a Hamiltonian perturbation of $L_1$ such that $CF^0(L_1,L_2)$ is of rank one, and hence $HF^0(L_1,L_2)$ is at most rank one.
By Lemma \ref{l:rankone}, it suffices to show that the first connecting map is non-zero.

By Lemma \ref{l:throughPert} there is no degree zero element
in $CF(\wt L_1[1],L_3)$ (note: $Ind(\wt L_1|_D,L_1|_D)=n-Ind(L_1|_D,\wt L_1|_D)=0$ by Example \ref{e:MorseBottIndex}).
If the connecting map is zero, $HF^0(\wt L_1[1],L_3)=HF^0(\wt L_1[1],L_1[1])\oplus HF^0(\wt L_1[1],L_2)$, which is at least rank one, so we arrive at a contradiction.
\end{proof}

\begin{comment}
\brmk\label{r:cleanSurgeryES} There is an obvious generalization of Corollary \ref{c:SurgeryTriangle} when $L_1\pitchfork L_2=D$ is a connected clean intersection.  Assume the Maslov index of $D=L_1\cap L_2$ is $dim(D)$, an adaption of Lemma \ref{l:throughPert} to the clean case will deduce for $[D]$ the fundamental class of $D$ (see \cite{FOOO_Book}),

$$L_1 \xrightarrow{[D]} L_2 \to L_3 \to L_1[1]. $$

\ermk
\end{comment}

We now consider Seidel's long exact sequence, which is slightly more involved.
Consider the exact triangle in $\mathcal{F}uk(M \times M^-)$.
\beq\label{e:catCone}S^n \times S^n \xrightarrow{c_d} \Delta \to Graph(\tau_{S}^{-1}) \to S^n \times S^n[1]\eeq
Since $HF^0(S^n \times S^n, \Delta)=HF^0(S^n,S^n)=\mathbb{K}$, we may apply Lemma \ref{l:rankone}.  Considering morphisms from $L_1\times L_2$ yields
$$\dots \to HF^*(L_1 \times L_2, S^n \times S^n) \xrightarrow{c_d} HF^*(L_1 \times L_2, \Delta) \to HF^*(L_1 \times L_2, Graph(\tau_{S}^{-1})) \to \dots$$
The fact that the connecting map $c_d$ is non-zero can be verified by plugging $L_1 \times L_2=S^n \times S^n$ and do a simple rank computation.  Hence, the connecting map $c_d$ can be taken as the multiplication with any non-zero element of  $HF^0(S^n \times S^n, \Delta)$.

To show that $c_d$ can be taken as the evaluation map, we give a proof of the following lemma communicated to the second author by S. Mau, which is a simple instance of quilt unfolding.

\begin{lemma}
Under the isomorphism $HF^*(L_1 \times L_2 , \Delta)=HF^*(L_1,L_2)$
and $HF^*(L_1 \times L_2, K \times K)=HF^*(K,L_2)\otimes HF^*(L_1,K)$,
the homomorphism
\begin{eqnarray}\label{eqn:Seidel1}
HF^*(L_1 \times L_2, K \times K) \xrightarrow{\mu^2(\bar{e}_K,\cdot)} HF^*(L_1 \times L_2,\Delta)
\end{eqnarray}
is identified with the evaluation map $HF^*(K,L_2)\otimes HF^*(L_1,K) \to HF^*(L_1,L_2)$.
Here $\bar{e}_K \in HF^0(K \times K, \Delta)$ is the image of the fundamental class under the isomorphism
$HF^*(K,K) \to HF^*(K \times K, \Delta)$
\end{lemma}

\begin{proof}

The homomorphism \ref{eqn:Seidel1} is obtained by counting quilted surfaces in the shape of the left of Figure \ref{fig:Quiltedfirst}.
\begin{figure}[h]
\centering
\includegraphics[scale=0.9]{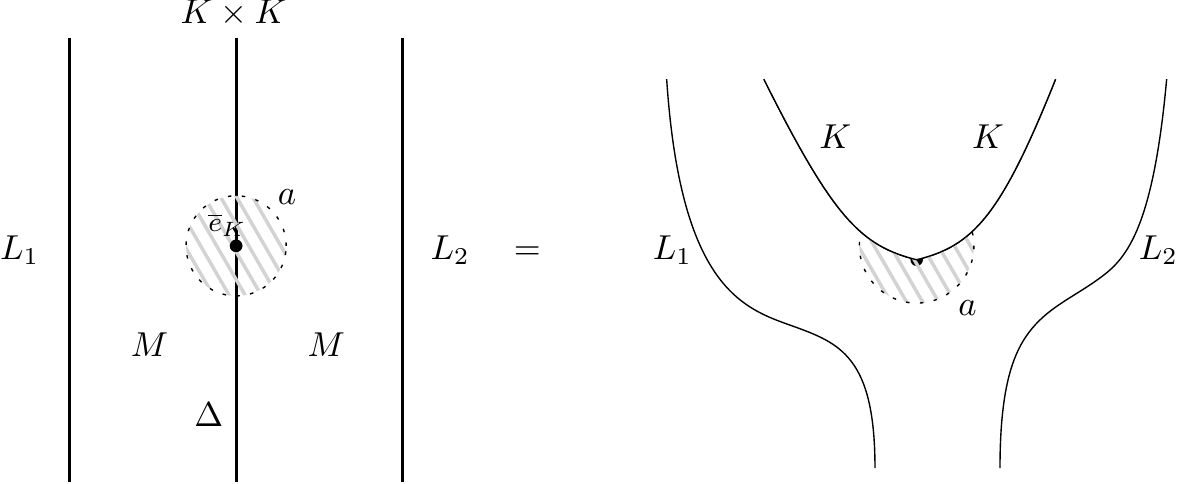}
\caption{Unfolding of $c_d$}
\label{fig:Quiltedfirst}
\end{figure}
Splitting the shaped region from the quilted surface, the rest of the surface is equivalent to counting
holomorphic curves shown in the right of Figure \ref{fig:Quiltedfirst}, where the evaluation at a point $a$ is constrained by the output from the shaded region.
The latter is given by rigid quilted cylinders shown on the left of Figure \ref{fig:Quilted3}, which is in turn equivalent to the
holomorphic section on the right.
\begin{figure}[h]
\centering
\includegraphics[scale=0.9]{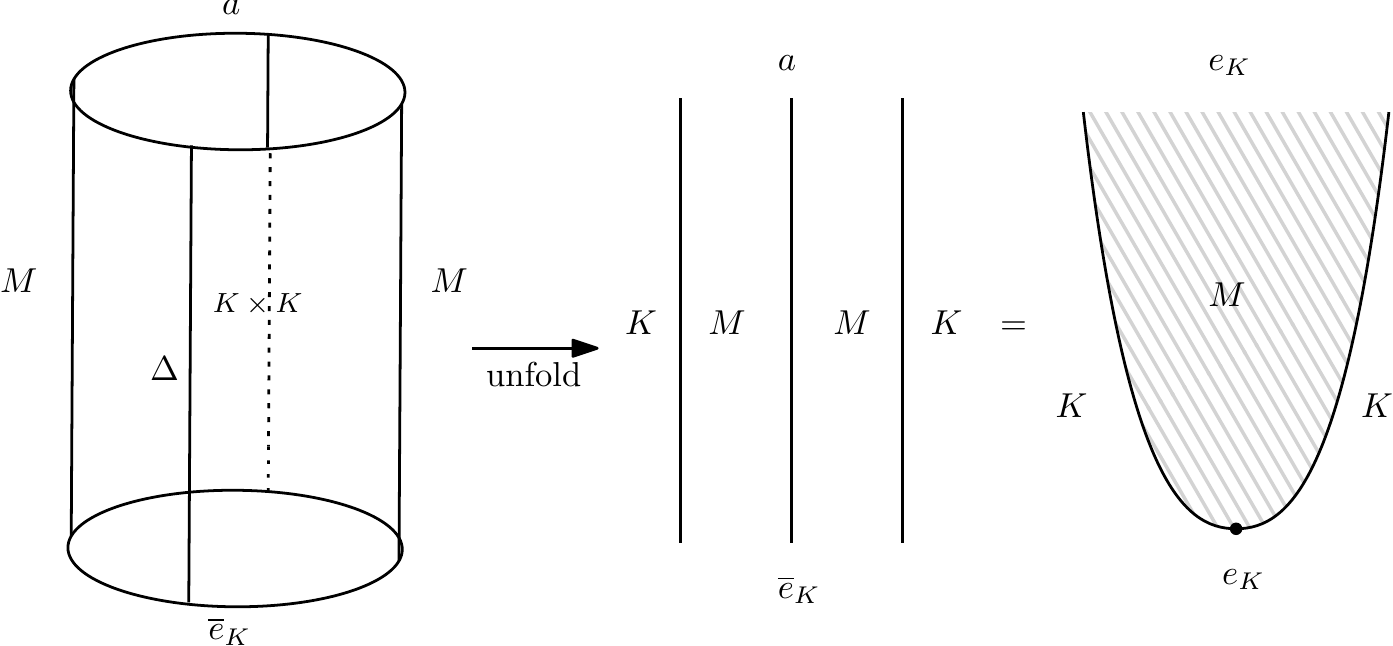}
\caption{Unfolding of the shaded region}
\label{fig:Quilted3}
\end{figure}
This shows $a=e_K$, meaning the constraint is only decorative.  A standard gluing thus concludes the lemma.

\end{proof}

A similar application of Lemma \ref{l:rankone} and quilt unfolding also gives an explicit description of $c_{f}:HF^*(f)\to HF^*(f(S^n),S^n)$ in \eqref{e:exactSeqFix} (this is the composition with $c$ in \eqref{e:catCone}).  Consider holomorphic half-strips $u: \R^{\ge0}\times [0,1]\to (M,J)$ for suitable choice of Floer data, with the following boundary conditions

$$
\left\{ \begin{aligned}
        &u(0,t)\in S^n \\
        &f(u(s,1))=u(s,0)\\
        &\lim_{s\to+\infty}u(s,t)=x\in Fix(f).
                          \end{aligned} \right.
                          $$

Rigid counting of these holomorphic half-strips defines a chain map $\CO_{f}:CF(f)\to CF(f(S^n),S^n)$.  To show that $[\CO_{f}]$ is identified with $c_{f}$, consider the unfolding shown in Figure \ref{fig:unfoldHF(f)} for $l=u(s,1)$.  Hence we showed that:

\bcor\label{c:fixedConn} $c_f=[\CO_f]$.\ecor

\begin{figure}[h]
\centering
\includegraphics[scale=0.9]{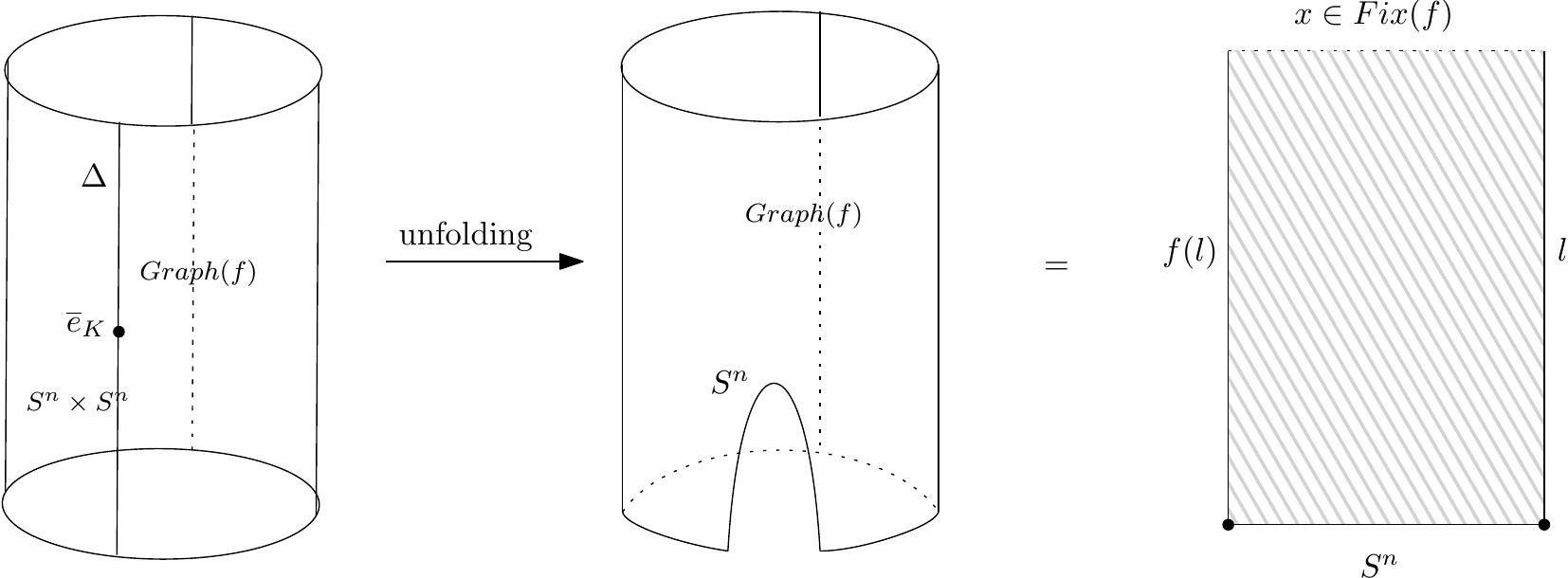}
\caption{The unfolding of $c_f$}
\label{fig:unfoldHF(f)}
\end{figure}

In the last two corollaries we consider the case of $\CP^n$-twists.  However, the reader should use caution here: they only hold in situations that we may upgrade Corollary \ref{c:CpnCpnSphere} and \ref{c:immersedSphereCone} to the categorical level.  While we believe this is true in general, we have not developed sufficient tools in the current paper to claim it as a theorem.  Nonetheless, we still include them here to make our discussions complete.

\begin{corr}\label{c:CpnCpnSphereMap}
Assuming Corollary \ref{c:CpnCpnSphere} can be upgraded to a categorical cone.
Let $S=\mathbb{CP}^{\frac{m}{2}}$ and $S_{\looparrowright}$ be the immersed sphere in Lemma \ref{l:immersedSphere}.
Let $0 \neq [h] \in HF^2(S,S)$.
Then there is a long exact sequence for any $N$
$$  \dots \to HF^*(N,S[-2]) \xrightarrow{\mu^2([h],\cdot)} HF^*(N,S) \to HF^*(N,S_{\looparrowright}) \to \dots $$
\end{corr}

\begin{proof}
Since $HF^0(S[-2],S)=\K\cdot [h]$, in view of Corollary \ref{c:CpnCpnSphere} and Lemma \ref{l:rankone},
it suffices to prove that the connecting map is non-zero.
In other words, we want to show that $HF^*(N,S[-2]) \oplus HF^*(N,S) \neq HF^*(N,S_{\looparrowright})$ for some $N$.
It follows directly by taking $N$ to be an appropriate perturbation of $S$, so that  $CF^*(N,S_\looparrowright)$ has rank equal to two.
\end{proof}

\begin{corr}\label{c:immersedLES}
Assume Corollary \ref{c:immersedSphereCone} can be upgraded to a categorical cone.
Let $S=\mathbb{CP}^{\frac{m}{2}}$, $L$ and $N$ be embedded exact Lagrangians in $M$.
Assume $S \pitchfork L =\{p\}$ and $Ind(S|_p,L|_p)=0$.
Let $S_{\looparrowright}$ be the immersed sphere in Lemma \ref{l:immersedSphere} constructed by $x_0=p$.

Denote the two generators as $q_0\in CF^0(S_\looparrowright,L)$ and $q_{-1}\in CF^{-1}(S_\looparrowright,L)$ which are both geometric point $\{p\}$.
Then there is a long exact sequence for any $N$
$$  \dots \to HF^*(N,S_{\looparrowright}) \xrightarrow{\mu^2([q_0],\cdot)} HF^*(N,L) \to HF^*(N,\tau_{S}(L)) \to \dots $$
\end{corr}

\begin{proof}
First of all, we want to show that $HF^*(L,S_\looparrowright) \neq 0$.
Note that, we have $HF^{m+2}(L,S[-2])=HF^m(L,S)=\K\cdot [p^\vee]$ and $HF^{m+2}(L,S)=0$.
By Corollary \ref{c:CpnCpnSphereMap}, we have
$$  \cdots\to HF^{m+1}(L,S_{\looparrowright})\to HF^{m+2}(L,S[-2]) \xrightarrow{\mu^2([h],\cdot)} HF^{m+2}(L,S) \to  \dots $$
which shows that $HF^{-1}(S_\looparrowright,L)=(HF^{m+1}(L,S_\looparrowright))^\vee =\K\cdot [q_{-1}]$.
In particular, $q_0$ and $q_{-1}$ are both cocycles since they are the only generators.
It also follows that

\begin{equation} \label{eq:1}
HF^*(S_\looparrowright,L)=\left\{ \begin{aligned}
         &\K,\quad *=-1,0 \\
                 &0,\quad \text{otherwise}
                          \end{aligned} \right.
                          \end{equation}

both $HF^0(S_{\looparrowright},L)$ and $HF^{-1}(S_{\looparrowright},L)$ have rank one.
From Lemma \ref{l:rankone} again the only thing to show is the non-vanishing of the connecting map $c_p\in Hom^0(S_\looparrowright,L)$.

Choose $N=\wt L$ to be the perturbation of $L$ as in Lemma \ref{l:throughPert}, we have a graded identification of intersection points $\wt L \cap  (\tau_{S}(L)\backslash\{q_{-1}\}) = (\wt L\cap L)\backslash\{p\} $, none of which has degree $\ge m$.  Note that our situation slightly differs from \ref{l:throughPert} since $S_\looparrowright$ has two branches of intersections at $p$ and the same proof there removes the intersection $q_0$ but $q_{-1}$ survives (Figure \ref{fig:throughPertCPn}).
We now have $Ind(\wt L|_{q_{-1}},\tau_{S}(L)|_{q_{-1}})=Ind(L|_{q_{-1}},L_{-1}|_{q_{-1}})=m$ and hence %=Ind(L|_{p},L_{-1}|_{p})
$CF^{m}(\wt L,\tau_{S}(L))=\K\cdot q_{-1}$ so the cohomology has at most rank one.
However, $HF^m(\wt L,S_{\looparrowright}[1])\oplus HF^m(\wt L,L)=HF^0(S_{\looparrowright}[1],L) \oplus HF^m(L,L)$ has rank two.
Therefore, one cannot have a short exact sequence in degree $n$ so the connecting map is non-vanishing.
\end{proof}

\begin{figure}[h]
\centering
\includegraphics[scale=1]{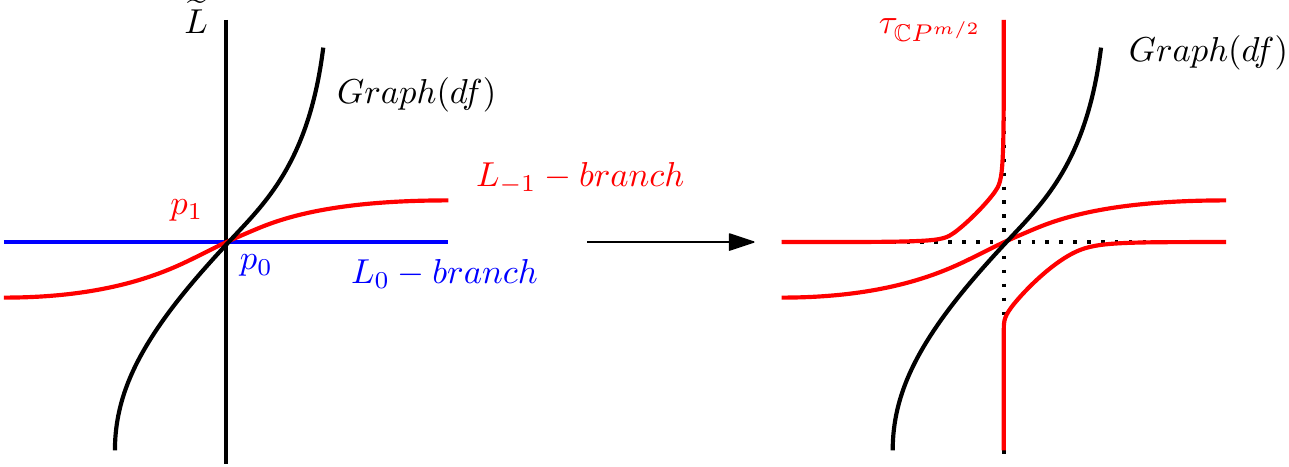}
\caption{Resolving the degree zero intersection by surgery with $S_\looparrowright$}
\label{fig:throughPertCPn}
\end{figure}

\section{Categorical point of view}

\subsection{$\CP^n$-twist and $\bP^n$-objects}

We recall the definition of $\bP^n$-objects from both $A$-side and $B$-side.

\bdf[\cite{HT06}, Definition 1.1]\label{d:bPn} Let $X$ be a smooth projective varieties.  An object $\cE\in D^b(X)$ is called a $\bP^n$-object if $\cE\otimes \w_X\cong\cE$ and $Ext^*(\cE,\cE)$ is isomorphic as a graded ring to $H^*(\bP^n,\C)$.
\edf

From here Huybrechts and Thomas defined an autoequivalence of $D^b(X)$. This is the Fourier-Mukai functor induced by the following iterated mapping cone in $D^b(X)$

\beq\label{e:HTcone} Cone(Cone(\cE^\vee\boxtimes\cE[-2]\rightarrow \cE^\vee\boxtimes\cE)\rightarrow \cO_\Delta).\eeq

We will not pursue the connecting maps in \eqref{e:HTcone} in this paper, but the readers should consult \cite{HT06} in case of interests.  On the $A$-side, one has the following notion of $\CP^n$-objects (and similarly for $\mathbb{RP}^n,\mathbb{HP}^n$-objects) in $A_\infty$-categories.

\bdf[\cite{Ha11}, Definition 3.1]\label{d:APn} A $\CP^n$-object (resp. $\mathbb{RP}^n,\mathbb{HP}^n$-object) in $\cA$ is a pair $(V,h)$ for $V\in Ob(\cA)$ and $h\in hom^\dag(V,V)$ such that

\begin{enumerate}[(1)]
\item $\mu^1(h)=0$ and $Hom(V,V)\cong \K[h]/h^{n+1}$ as a graded ring,
\item There is a map $\int:Hom^{!}(V,V)\rightarrow k$ such that for any $X\in\cA$, the bilinear map $Hom^{!-k}(X,V)\otimes Hom^k(V,X)\rightarrow Hom^{!}(V,V)\rightarrow \K$
    is non-degenerate.
\end{enumerate}

Here, $\dag=2,1,4$, respectively, for $\CP^n$, $\mathbb{RP}^n$ and $\mathbb{HP}^n$. $!=2n,n,4n$, respectively, for $\CP^n$, $\mathbb{RP}^n$ and $\mathbb{HP}^n$.
\edf

A typical example of $\CP^n$-object is given by an exact Lagrangian $\CP^n$ in $\fuk(M)$ for an exact symplectic manifold $M$.  In \cite{HT06} it is conjectured that the Lagrangian $\CP^n$-twist is mirror to a $\bP^n$-twist in the derived category of the mirror variety.  Based on this speculation, \cite{Ha11} constructed an algebraic version of the $\CP^n$-twist in $A_\infty$ category and conjectured that is exactly the auto-equivalence induced by a Lagrangian $\CP^n$-twist.  Our main result in this direction is to confirm this conjecture up to determination of connecting maps in an exact symplectic manifold $M$.

\bthm\label{t:catPn} Given a Lagrangian $\CP^n=S\subset M$ (resp. $\mathbb{RP}^n,\mathbb{HP}^n$).  The auto-equivalence induced by Lagrangian $S$-twist is equivalent to the following iterated cone in the category of $fun(\fuk(M),\fuk(M))$

$$Cone(hom(S,-)\otimes S[-\dag]\rightarrow hom(S,-)\otimes S\rightarrow id_{\fuk(M)}).$$
where $\dag=2,1,4$, respectively, for $\CP^n$, $\mathbb{RP}^n$ and $\mathbb{HP}^n$.
\noindent The corresponding cone in the bimodule category $Bimod(\fuk(M))$ also holds.

\ethm

\bpf The construction from Lemma \ref{l:CPnCob}, grading consideration from Lemma \ref{l: graded identity for graph of CPm/2}, \ref{l:RpnHpnSurgIdentity} and the main theorem in \cite{BC2} implies a quasi-isomorphism of iterated cones in $\fuk(M\times M^-)$
$$Cone((S\times S)[-\dag]\rightarrow S\times S\rightarrow\Delta_M)\cong Graph(\tau_{S}^{-1}).$$

The desired assertion is simply the image of this equality under the M'au-Wehrheim-Woodward functor $\Phi$.  The counterpart in the bimodule category follows replacing the functor $\Phi$ by the functor $\sG$ (See the end of Section \ref{s:review}).

\epf

Similarly, by replacing Lemma \ref{l:CPnCob} with Lemma \ref{l:fiberedTwist}, Lemma \ref{l: graded identity for graph of CPm/2} and \ref{l:RpnHpnSurgIdentity} with Lemma \ref{l:familyGrade}, we have

\bthm\label{t:familyPn} Given a projectively coisotropic manifold $C \subset M$.  The auto-equivalence induced by family projective-twist is equivalent to the following iterated cone in the category of $fun(\fuk^{\#}(M),\fuk^{\#}(M))$

$$Cone(\wt C[-\dag]\rightarrow \wt C \rightarrow id_{\fuk^{\#}(M)}).$$
%\noindent The corresponding cone in the bimodule category $Bimod(\fuk(M))$ also holds
where $\dag=2,1,4$, respectively, if the projective fiber is $\CP^l$, $\mathbb{RP}^l$ and $\mathbb{HP}^l$.
\ethm

We remark that the functor $\wt C$ should be regarded as the composite of the functors $C^t: \fuk^{\#}(M) \to \fuk^{\#}(B)$ and $C: \fuk^{\#}(B) \to \fuk^{\#}(M)$.

\subsection{Long exact sequences as cones of functors/bimodules}\label{s:functorCones}

We may recapitulate results from  Section \ref{s:proofLES} on the functor level.  Recall that Seidel proved the following categorical version of Theorem \ref{thm:long exact sequence 1} in \cite{Seidelbook} to Theorem \ref{t:sphereTwist}.

\beq\label{e:catLES} \rightarrow hom(S^n,L)\otimes S^n\xrightarrow{ev} L\rightarrow \tau_{S^n}L\xrightarrow{[1]}\eeq

This result can be considered as a consequence of our previous results in two equivalent point of views as functors and bimodules.  The first one is straightforward given the M'au-Wehrheim-Woodward's $A_\infty$-functor $\Phi$ \eqref{e:MWW}.

Given our cobordism construction, Corollary \ref{c:admissibleToDehnProduct}, Corollary \ref{c: graded idenity for graph of Sn} and Lemma \ref{l:cob}, and combining with the main result from \cite{BC2}, we indeed have a cone
in $\fuk(M\times M)$
\beq\label{e:productCone} S^n\times (S^n)^-\rightarrow \Delta\rightarrow Graph(\tau_{S^n}^{-1})\xrightarrow{[1]}.\eeq

Hence, under $\Phi$ this turns into a cone of functors
\beq\label{e:functorCone} hom(S^n,-)\otimes S^n\rightarrow Id_{Tw\fuk(M)}\rightarrow \Phi_{\tau_{S^n}}\xrightarrow{[1]}\eeq

\noindent or, if a compactly supported symplectomorphism $\phi\times id$ is applied to the cobordism, the resulting cone reads
\beq\label{e:functorCone2} hom(\p(S^n),-)\otimes S^n\rightarrow \Phi_\phi\rightarrow \Phi_{\tau_{S^n}\circ\phi}\xrightarrow{[1]}\eeq

Evaluating \eqref{e:functorCone} at any object $L\subset\fuk(M)$ hence recovers \eqref{e:catLES}, while further evaluating another object gives the cohomological version Theorem \ref{thm:long exact sequence 1}.  Corollary \ref{c:Seideltwist-fixed} follows from \eqref{e:functorCone2} considering the morphisms to the identity functor in simple cases.    For the family Dehn twist Lemma \ref{l:fiberedTwist}, we may also interpret it as a cone of functors, but we need to go to general Lagrangian correspondence framework: by the time of writing, it is not clear the functor induced by $\wt C$ has target reduced to the derived Fukaya category even for spherically coisotropic manifolds.

\subsection{Categorical automorphisms of $ADE$-singularities}

In this section, we investigate the compactly supported symplectomorphisms of Milnor fibers $W$ of an $ADE$-singularity in a categorical point of view.  In other words, we are interested in the image of the natural homomorphism $Symp_c(W)\rightarrow Aut(Fuk(W))$. The goal is to show that

\bthm\label{t:AnFun} For any compactly supported symplectomorphism $\phi\in Symp_c(W)$, $\Phi_\phi\in D^\pi Aut(\fuk(W))$ is split generated by compositions of Dehn twists along the standard vanishing cycles.
\ethm

Here $D^\pi Aut(\fuk(W))$ is considered the image in $Fun(D^\pi\fuk(W), D^\pi\fuk(W))$ induced by compactly supported symplectomorphisms.  More precisely, a symplectomorphism induces an $A_\infty$ autoequivalence on $\fuk(W)$, and $D^\pi Aut(\fuk(W))$ consists of the split closure of the image of such autoequivalences under the following functor (see \cite[(1.10)]{Seidelbook})
$$H(fun(\fuk(W),\fuk(W)))\rightarrow Fun(H(\fuk(W)), H(\fuk(W))).$$
We start by considering a slightly more generalized situation.  Recall that a \textit{weighted homogenous polynomial} $q$ satisfies $q(\lambda^{\beta_1}z_1,\dots,\lambda^{\beta_n}z_n)=\lambda^\beta q(z_1,\dots, z_n)$ for some integers $(\beta; \beta_1,\dots,\beta_n)$.  Then $q:\C^n\rightarrow \C$ has an isolated singularity at $(0,0)$, and the nearby fibers are called the \textit{Milnor fiber of the weighted homogeneous singularity $W_q$}.

Seidel \cite{SeGraded}\cite{Seidelbook} studied the symplectic nature of $W_q$ through its monodromy around $0\in\C$ as follows.  Consider $q^{-1}(D^2(1)\backslash\{0\})$, one may choose a symplectic connection which is trivial near infinity in the fibers.  This induces a monodromy $f\in Symp_c(W_q)$ by parallel transport around the origin, which decomposes into a sequence of Dehn twists along Lagrangian spheres $\{L_i\}_{i=1}^l$ by perturbing $q$ into a Lefschetz fibration.  The Lagrangian spheres $L_i$ are indeed the vanishing cycles of this Lefschetz fibration.

It is shown in \cite{SeGraded} that the iterate $f^\beta|_U=id[k_q]|_U$.  Here $k_q=2(\beta-\sum\beta_i)\in \Z$ and $U\subset W_q$ is a compact set which can be chosen arbitrarily large by varying the choice of symplectic connections.  In the rest of the section we further assume that $$\sum\beta_i\neq\beta.$$
The following observation allows one to show that $\{L_i\}_{i=1}^l$ split generates $D^\pi\fuk(W_q)$, the Fukaya category generated by compact Lagrangian branes.

\blem[\cite{Seidelbook}, (5e)]\label{l:twist-generate} Let $C$ be a twisted complex split generated by objects $L_1,\dots, L_l$, and there is an exact triangle
$$C\rightarrow L\xrightarrow{t}L[\alpha]\rightarrow C[1]$$
for some perfect complex $L$ and $\alpha\in\Z\backslash\{0\}$.  Then $L_1,\dots,L_l$ split generate $L$.
\elem

The lemma follows from part of the octahedron axiom, which asserts that the cone of the composition
\beq\label{e:compCone} t[d\alpha]\circ\cdots\circ t[\alpha]\circ t: L\rightarrow L[(d+1)\alpha]\eeq

is generated by shifted copies of $C$.  The boundedness assumption on $L$ then implies the vanishing of \eqref{e:compCone} when $d$ is large enough, as desired.

Now take the vanishing cycles $L_i$ involved in the monodromy $f$ for $W_q$.  From \eqref{e:functorCone}, there is an equality of twisted complexes
$$\Phi_{f^\beta\circ\phi }\cong Cone(\sC\rightarrow \Phi_\phi).$$

Here $\sC$ is an iterated cone formed by functors of the form $\Phi_{L_i\times\wt L_i}$, where $\wt L_i$ is a Lagrangian sphere differed from $L_i$ by $\phi^{-1}$ and a composition of twists by vanishing cycles.

We note that although $f^\beta\neq id[k_q]$ as a symplectomorphism, $\Phi_{f^\beta}=id[k_q]$ as auto-equivalences on $\fuk(W_q)$ (while the case will be drastically different when wrapped version is considered).  This still does not put us back to the framework of Lemma $\ref{l:twist-generate}$: it is unclear that $H^0(hom_{\fuk(W_q)}(\Phi_\phi,\Phi_\phi[d]))=HH^d(\fuk(W_q))=0$ for a symplectomorphism $\phi$ and large $d$ in our case, and the naive expectation that $rank(HH^*(\fuk(V)))<\infty$ does not always hold for a Stein manifold $V$ (this was pointed out to the authors by Nick Sheridan).

The good news is, the $\text{zero}^{\text{th}}$ order term of arbitrary natural transformation in $hom_{\fuk(W_q)}(\Phi_\phi,\Phi_\phi[d])$ necessarily vanishes in $0$-th cohomology.  More precisely, a natural transformation $T\in hom^0(\eF,\eG)$ consists of a sequence $T=(T^0,T^1,\dots)$, where $T^0\in hom^0(\eF X,\eG X).$

 When $\eF=id$ and $\eG=(f^\beta)^d$ and $d$ is large, $0\equiv[T^0]\in Hom^0(H(\eF),H(\eG))$. This is equivalent to the statement that $HF^0(L,L[dk_q])=0$ for any compact Lagrangian brane when $dk_q>n$.  As a result, we do have an exact triangle in $D^\pi(Aut(\fuk(W_q)))$

\beq\label{e:ConeinAut} \sC\rightarrow\Phi_\phi\xrightarrow{0}\Phi_\phi[d]\rightarrow\sC[1]\eeq

which implies functors $\Phi_{L\times L'}$ split generate $D^\pi(Aut(\fuk(W_q)))$ by Lemma \ref{l:twist-generate}, where $L$ is one of the vanishing cycle and $L'$ is a Lagrangian sphere.  We could further reduce $L'$ to a vanishing cycle as well.

\blem\label{l:vanishingGen} If $\{L_i\}$ split generate $\fuk(W)$ for any Stein manifold $W$, any functor of the form $\Phi_{K_0\times K_1}$ is split generated by $\Phi_{L\times L'}$.
\elem

\bpf We recall from \cite[2.5]{Gan13} (see \cite{SeHom} for the $A_\infty$-module version) that, given an $A_\infty$-categories $\sA,\sB,\sC$ and a $(\sB,\sC)$-bimodule $\sP$, the convolution $\Gamma_\sP:\sM\mapsto\sP\otimes_{\sB}\sM$ defines a $dg$-functor
from $(\sA,\sB)-mod$ to $(\sA,\sC)-mod$.

For the geometric situation at hand, consider $\fuk(W_q)-mod$ as $(\fuk(W_q),\mathbf{k})$-bimodules, and let $\sP_L=\sY_L^l$, the image of $L$ under the left Yondeda embedding $\sI^l$.  We then obtain the following composition of $A_\infty$ functors

\begin{align*}\fuk(W_q)
\xrightarrow{\sI^r}&\text{mod-}\fuk(W_q)
\xrightarrow{\Gamma_{\sP_L}}Bimod(\fuk(W_q))^\times \\
\xrightarrow{\sG^*}&\fuk(W_q\times W_q)^\times
\xrightarrow{\Phi}fun(\fuk(W_q),\fuk(W_q)).\end{align*}

Here $Bimod(W_q)^\times\subset Bimod(W_q)$ denotes the subcategory generated by bimodules of shape $\sY^l_K\otimes_k\sY^r_{K'}$.

A consequence is that, a cone $L_1\rightarrow L_2\rightarrow L_3\xrightarrow{[1]}L_1[1]$ in $\fuk(W_q)$ is mapped to a cone of functor $\Phi_{L\times L_1}\rightarrow\Phi_{L\times L_2}\rightarrow\Phi_{L\times L_3}\xrightarrow{[1]}\Phi_{L\times L_1}[1]$, which concludes the lemma.
\epf

\bcor\label{c:vanishingGen} $D^\pi(Aut(\fuk(W_q)))$ is split generated by $\Phi_{L\times L'}$, where $L$ and $L'$ are vanishing cycles.

\ecor

\bpf
Now we may apply the theorem by Seidel \cite{SeGraded}\cite{Seidelbook} that any compact exact Lagrangian sphere in $W_q$ is split generated by vanishing cycles.  This decomposes arbitrary $\Phi_{K_0\times K_1}$ into functors of the form $\Phi_{L\times L'}$ for $L,L'$ are both vanishing cycles by Lemma \ref{l:vanishingGen}.

\epf

Now Theorem \ref{t:AnFun} is an immediate consequence: for any two vanishing cycles $L$ and $L'$ in the $A_n$-Milnor fiber, there is a composition of Dehn twists $T_{L,L'}$ along vanishing cycles sending $L$ to $L'$ up to Hamiltonian isotopy.  Therefore, $T^{-1}_{L,L'}$ and $\tau_L\circ T^{-1}_{L,L'}$ generate $\Phi_{L\times L'}$.

The extension to $D$ and $E$-type singularties was pointed out to us by Ailsa Keating: in fact, one may conclude the same result for weighted homogeneous singularities for which the smoothing is a plumbing of $T^*S^n$ according to a \textit{simple graph}.  To see this, one notices that the only place we need to specialize to $A_n$-type Milnor fiber is to prove the existence of $T_{L,L'}$.  But for simple graph plumbings, one may always find an $A_n$-subgraph connecting $L$ and $L'$ and use the composition of Dehn twists there.

\brmk Ailsa Keating and Ivan Smith also suggested Theorem \ref{t:AnFun} should hold for an even larger class of singularities.  In particular, the conclusion holds provided that the discriminant locus in the miniversal deformation base is irreducible.  In this case, as explained to us by Denis Auroux, the fundamental group of the discriminant complement is normally generated by a single meridian, which can be translated back to the language of $T_{L,L'}$.  However, the authors do not know whether this class includes all weighted homogeneous singularities.

\ermk

\bibliography{FukRef}

\bibliographystyle{plain}

\end{document}